\documentclass{article}
\usepackage[utf8]{inputenc}
\usepackage{amsmath}
\usepackage{amsthm}
\usepackage{amssymb}
\usepackage{esint}
\usepackage[bookmarks=false,
 breaklinks=false,pdfborder={0 0 1},backref=false,colorlinks=false]
 {hyperref}

\makeatletter
\usepackage{amsfonts}
\usepackage{makeidx}
\usepackage{graphicx}
\usepackage{latexsym}
\usepackage{amsthm}
\numberwithin{equation}{section}

\newtheorem{theorem}{Theorem}[section]
\newtheorem{axiom}{Axiom}\newtheorem{corollary}[theorem]{Corollary}\newtheorem{definition}[theorem]{Definition}\newtheorem{lemma}[theorem]{Lemma}\newtheorem{proposition}[theorem]{Proposition}\newtheorem{remark}[theorem]{Remark}

\title{Generalized functions\\
in a Non-Archimedean field}
\author{Vieri Benci \thanks{
Dipartimento di Matematica, Universit\`{a} degli Studi di Pisa, Via F.
Buonarroti 1/c, Pisa, ITALY}}

\makeatother

\begin{document}
\maketitle
\begin{abstract}
Ultrafunctions are generalized functions defined on a non-Archimedean
field extending $\mathbb{R}$. They provide a framework for treating
problems that fall outside the classical theory of distributions.
In this paper we introduce an improved notion of ultrafunctions which
refines the previous constructions. This new approach allows a more
flexible functional framework than the one used in (\cite{ultra},\cite{BBG},\cite{belu2012},..,\cite{bls})
and yields a more flexible functional setting. In particular, this
approach allows us to prove the existence of ultrafunction solutions
for several classes of ill-posed evolution problems arising in partial
differential equations. \medskip{}


\smallskip{}

\end{abstract}
\tableofcontents{}

\section{Introduction}

Generalized functions play a fundamental role in the analysis of partial
differential equations, especially when classical solutions fail to
exist or when singularities appear. The theory of distributions introduced
by Schwartz provides a powerful linear framework for dealing with
such situations and has become a standard tool in modern analysis.
However, it is well known that the space of distributions $\mathcal{D}'(\Omega)$
does not possess the structure of an associative and commutative algebra
compatible with the classical product of functions. In particular,
the product of two distributions is not defined in general, a limitation
established by the Schwartz impossibility theorem \cite{Schwartz}.

Several approaches have been proposed in order to overcome this difficulty.
Among them we mention Colombeau algebras, which embed distributions
into a differential algebra constructed through regularization procedures,
and various theories based on non-Archimedean or nonstandard frameworks.
These approaches aim at extending the class of admissible nonlinear
operations while preserving the consistency of classical calculus.

The theory of ultrafunctions provides an alternative framework for
generalized functions based on a non-Archimedean field $\mathbb{E}$
containing the real numbers. Ultrafunctions are functions defined
on a hyperfinite set $\Gamma^{N}$ satisfying 
\[
\mathbb{R}^{N}\subset\Gamma^{N}\subset\mathbb{E}^{N},
\]
and taking values in $\mathbb{E}$. The space of ultrafunctions, denoted
by $V(\Gamma^{N})$, forms an algebra over $\mathbb{E}$ and contains
natural extensions of classical functions and distributions.

This framework combines several desirable features. First, ultrafunctions
are defined pointwise on the hyperfinite grid $\Gamma^{N}$, which
allows one to retain a pointwise interpretation even for highly singular
objects. Second, the space $V(\Gamma^{N})$ is closed under algebraic
operations, so nonlinear expressions involving generalized functions
are well defined. Third, ultrafunctions are equipped with a generalized
derivative 
\[
D_{i}:V(\Gamma^{N})\to V(\Gamma^{N})
\]
and with a pointwise integral 
\[
\int^{\circ}:V(\Gamma^{N})\to\mathbb{E},
\]
which extend the classical notions of derivative and integral.

More precisely, every function $f:\mathbb{R}^{N}\to\mathbb{R}$ admits
a natural extension 
\[
f^{\circ}:\Gamma^{N}\to\mathbb{E},
\]
and the pointwise integral satisfies 
\[
\int^{\circ}f^{\circ}(x)\,dx=\int f(x)\,dx
\]
for every integrable standard function. Moreover, distributions can
be canonically represented within this framework: to every distribution
$T\in\mathcal{D}'(\mathbb{R}^{N})$ corresponds an ultrafunction $T^{\circ}$
satisfying 
\[
\int^{\circ}T^{\circ}(x)\varphi^{\circ}(x)\,dx=\langle T,\varphi\rangle
\]
for all test functions $\varphi$.

An important feature of the ultrafunction framework is the hyperfinite
structure of the grid $\Gamma^{N}$. This structure introduces a form
of compactness which allows one to obtain generalized solutions for
a large class of equations, including problems that are ill posed
in the classical sense. Roughly speaking, an \emph{a priori} bound
in a suitable function space is often sufficient to guarantee the
existence of an ultrafunction solution.

Ultrafunctions were introduced in \cite{ultra} and further developed
in \cite{BBG,belu2012,benci22,bls}. Several models and applications
have been studied in these works, particularly in connection with
nonlinear PDEs and variational problems.

The purpose of the present paper is to introduce an improved model
of ultrafunctions which preserves additional structural properties
of classical differential calculus. In particular, the model constructed
here allows the generalized derivatives to satisfy symmetry properties
analogous to the classical Schwarz theorem, namely 
\[
\partial_{x}\partial_{y}=\partial_{y}\partial_{x}.
\]

The paper is organized as follows. In Section~\ref{BA} we recall
the basic axioms of the theory of ultrafunctions. Section~\ref{SA}
presents several applications to partial differential equations and
variational problems. In Section~\ref{CSU} we construct an explicit
model of the theory based on a hyperfinite grid and prove that the
axioms are satisfied. Finally, an appendix discusses the relation
between ultrafunctions and other theories of generalized functions.

\section{Preliminaries}

\subsection{Non Archimedean Fields}\label{naf}

Here we recall the basic definitions and some well-known facts regarding non-Archimedean fields.

\begin{definition} A field $\mathbb{K}$ is called ordered if there
exists a set $\mathbb{K}^{+}\subset\mathbb{K}$ such that
\begin{enumerate}
\item $x,y\in\mathbb{K}^{+}\Rightarrow x+y,xy\in\mathbb{K}^{+};$
\item $\mathbb{K}=\mathbb{K}^{+}\cup\left\{ 0\right\} \cup\mathbb{K}^{-}$
where $\mathbb{K}^{-}=\left\{ x\in\mathbb{K}\ |\ -x\in\mathbb{K}^{+}\right\} .$ 
\end{enumerate}
\end{definition}

In an ordered field the order relation is defined as follows: 
\[
x<y\Leftrightarrow y-x\in\mathbb{K}^{+}.
\]

In the following, $\mathbb{K}$ will denote an ordered field. Its
elements will be called numbers. It is well known that every ordered
field contains (a copy of) the rational numbers; hence the following
definitions make sense:

\begin{definition} Let $\mathbb{K}$ be an ordered field. Let $\xi\in\mathbb{K}$.
We say that:
\begin{itemize}
\item $\xi$ is infinitesimal if, for all positive $n\in\mathbb{N}$, $|\xi|<\frac{1}{n}$;
\item $\xi$ is finite if there exists $n\in\mathbb{N}$ such that $|\xi|<n$;
\item $\xi$ is infinite if, for all $n\in\mathbb{N}$, $|\xi|>n$ (equivalently,
if $\xi$ is not finite). 
\end{itemize}
\end{definition}

\begin{definition} An ordered field $\mathbb{K}$ is called non-Archimedean
if it contains an infinite element. \end{definition}

It is easily seen that all infinitesimals are finite and that the
inverse of an infinite number is a nonzero infinitesimal number. Infinitesimal
numbers can be used to formalize a new notion of ``closeness'':

\begin{definition} \label{def infinite closeness} We say that two
numbers $\xi,\zeta\in\mathbb{K}$ are infinitely close if $\xi-\zeta$
is infinitesimal. In this case, we write $\xi\sim\zeta$. \end{definition}

Clearly, the relation ``$\sim$'' of infinite closeness is an equivalence
relation.

\begin{theorem} \label{na}If $\mathbb{K}\supseteq\mathbb{R}$ is
an ordered field, every finite number $\xi\in\mathbb{K}$ is infinitely
close to a unique real number $r\sim\xi$. $r$ is called the \textbf{standard
part} of $\xi$ and denoted by $st(\xi)$. \end{theorem}

\textbf{Proof}: Given a finite number $\xi\in\mathbb{K}$, we set
\[
A:=\{r\in\mathbb{R}\ |\ r<\xi\},\quad B:=\{r\in\mathbb{R}\ |\ r\geq\xi\}.
\]

We have that $(A,B)$ is a Dedekind cut of $\mathbb{R}$; moreover,
since $\xi$ is finite, $A\neq\varnothing$ and $B\neq\varnothing$.
Then, by the completeness of the reals there exists $c\in\mathbb{R}$
such that 
\[
\forall a\in A,\forall b\in B,\ a\leq c\leq b.
\]

Now it is not difficult to check that $c\sim\xi$.
$\square$

\bigskip{}

\begin{corollary} \label{co1}If $\mathbb{K}\supset\mathbb{R}$ ($\mathbb{K}\neq\mathbb{R}$)
is an ordered field, then it is non-Archimedean. \end{corollary}

\textbf{Proof}: Take $\xi\in\mathbb{K}\backslash\mathbb{R}$. If $\xi$
is infinite, then $\mathbb{K}$ is non-Archimedean by definition.
If $\xi$ is finite then 
\[
\zeta:=\frac{1}{\xi-st\left(\xi\right)}
\]
is infinite, and hence $\mathbb{K}$ is non-Archimedean.

$\square$

\bigskip{}

Now we collect some basic properties of the function $st(\cdot)$.

\begin{proposition} \label{PS}Let $\xi$ and $\zeta$ be finite
numbers. Then
\begin{enumerate}
\item if $\xi\in\mathbb{R}$, $st\left(\xi\right)=\xi$;
\item \label{d}$\xi\leq\zeta\Rightarrow st\left(\xi\right)\leq st\left(\zeta\right)$;
\item $st\left(\xi+\zeta\right)=st\left(\xi\right)+st\left(\zeta\right)$;
\item $st\left(\xi\cdot\zeta\right)=st\left(\xi\right)\cdot st\left(\zeta\right)$;
\item if $st\left(\zeta\right)\neq0,$ then 
\[
st\left(\frac{\xi}{\zeta}\right)=\frac{st\left(\xi\right)}{st\left(\zeta\right)}.
\]
\end{enumerate}
\end{proposition}

\textbf{Proof:}  See any book of NSA (e.g.  \cite{keisler76}).
$\square$

\begin{definition} \label{MG}Let $\mathbb{K}$ be a non-Archimedean
field, and $\xi\in\mathbb{K}$ a number. The monad of $\xi$ is the
set of all numbers that are infinitely close to it: 
\[
\mathfrak{mon}(\xi)=\{\zeta\in\mathbb{K}:\xi\sim\zeta\},
\]
and the galaxy of $\xi$ is the set of all numbers that are finitely
close to it: 
\[
\mathfrak{gal}(\xi)=\{\zeta\in\mathbb{K}:\xi-\zeta\ \text{is finite}\}.
\]
\end{definition}

By definition it follows that the set of infinitesimal numbers is
$\mathfrak{mon}(0)$ and that the set of finite numbers is $\mathfrak{gal}(0)$.
Moreover, the standard part can be regarded as a function: 
\begin{equation}
st:\mathfrak{gal}(0)\rightarrow\mathbb{R}.\label{sh}
\end{equation}

\subsection{The $\Lambda$-limit and the Euclidean numbers}\label{lt}

In this section we present the notion of $\Lambda$-limit.  The $\Lambda$-limit
provides a simplified approach to Nonstandard Analysis (see also \cite{benci99,ultra});
it avoids the most difficult (and beautiful) features of NSA; however
it is sufficient for the purposes of this paper.

Let 
\begin{equation}
\mathfrak{L}=\mathfrak{\wp}_{fin}(\Lambda)\label{elle}
\end{equation}
be the family of finite subsets of $\Lambda$, where $\Lambda$ is
a set sufficiently large. For the purposes of this paper it is sufficient
to assume that 
\begin{equation}
\Lambda\supset\mathbb{R}\cup\mathfrak{F}(\mathbb{R})\cup\mathfrak{F}(\mathfrak{F}(\mathbb{R})).\label{esse}
\end{equation}

$\mathfrak{L}$ equipped with the partial order structure "$\subset$"
is a directed set. A function $\varphi:\mathfrak{L}\rightarrow E$
is called \textit{net} (with values in $E$) and the set of
such nets is denoted by $\mathfrak{F}(\mathfrak{L},E)$.

The usual (Cauchy) limit of a net is defined as follows:
\begin{equation}
L:=\lim_{\lambda\rightarrow\Lambda}\varphi(\lambda)\label{lim+}
\end{equation}
if and only if, $\forall\varepsilon\in\mathbb{R}^{+}$, $\exists\lambda_{0}\in\mathfrak{L}$ such that  , 
\[
|\varphi(\lambda)-L|\le\varepsilon.
\]
Notice that in the notation (\ref{lim+}),  $\Lambda$ can be regarded
as the ``point at infinity'' of $\mathfrak{L}$.  A typical example
of a limit of a net defined on $\mathfrak{L}$ is provided by the
definition of the Cauchy integral: 
\[
\int^{b}_{a}f(x)\,dx=\lim_{\lambda\rightarrow\Lambda}\sum_{x\in[a,b]\cap\lambda}f(x)(x^{+}-x),\qquad x^{+}:=\min\{y\in[a,b]\cap\lambda\mid y>x\}.
\]

We now define a new notion of limit called the $\Lambda$-\textit{limit}.
It is defined for every net $\varphi\in\mathfrak{F}(\mathfrak{L},\mathbb{R})$
and the limit value is a number in a non-Archimedean field $\mathbb{E}\supset\mathbb{R}$
called the field of the \textbf{Euclidean numbers}. Its peculiarity
consists in the fact that every net $\varphi:\mathfrak{L}\rightarrow\mathbb{R}$
always has a unique limit
\[
L=\lim_{\lambda\uparrow\Lambda}\varphi(\lambda)\in\mathbb{E}.
\]
In order to distinguish the limit (\ref{lim+}) (which
we will call the \textit{Cauchy limit}) from the $\Lambda$-limit,
we use the notations "$\lambda\rightarrow\Lambda$" and "$\lambda\uparrow\Lambda$"
respectively.

The field $\mathbb{E}$ and the $\Lambda$-limit are based on a fine
ultrafilter $\mathcal{U}$ over $\mathfrak{L}$, namely a filter such
that

\[
Q\in\mathcal{U}\Leftrightarrow\mathfrak{L}\setminus Q\notin\mathcal{U}\qquad(\mathit{ultra\;property})
\]

\[
\forall\lambda\in\mathfrak{L},\;Q[\lambda]\in\mathcal{U}\qquad(\mathit{fineness\;property})
\]
where
\[
Q[\lambda]:=\{\mu\in\mathfrak{L}\mid\mu\supseteq\lambda\}
\]
is a \textit{cone}  with \textit{vertex}  in $\lambda$.  You may think of $\mathcal{U}$ as a family of neighborhoods
of the point at infinity "$\Lambda$".  A set $Q\in\mathcal{U}$
is called \textbf{qualified}.

\begin{definition} If $\mathcal{R}$ is a relation, we say that ``$\mathcal{U}$-eventually
$\varphi(\lambda)\mathcal{R}\psi(\lambda)$'' if 
\[
\exists Q\in\mathcal{U},\ \forall\lambda\in Q,\ \varphi(\lambda)\mathcal{R}\psi(\lambda).
\]
\end{definition}

Now we set 
\[
I_{\mathcal{U}}=\{\psi\in\mathfrak{F}(\mathfrak{L},\mathbb{R})\mid\mathcal{U}\text{-eventually }\psi=0\}.
\]
Since $\mathcal{U}$ is an ultrafilter, it is well known and not difficult
to prove that $I_{\mathcal{U}}$ is a maximal ideal of the ring $\mathfrak{F}(\mathfrak{L},\mathbb{R})$.
Hence
\[
\mathfrak{F}(\mathfrak{L},\mathbb{R})/I_{\mathcal{U}}=\{[\varphi]_{\mathcal{U}}\mid\varphi\in\mathfrak{F}(\mathfrak{L},\mathbb{R})\}
\]
where

\[
[\varphi]_{\mathcal{U}}=\varphi+I_{\mathcal{U}}
\]
is a field.  Now we can define the field of Euclidean numbers $\mathbb{E}$: it
is a field such that there is an isomorphism

\begin{equation}
J:\mathfrak{F}(\mathfrak{L},\mathbb{R})/I_{\mathcal{U}}\rightarrow\mathbb{E}\label{EEE}
\end{equation}
satisfying the following relation: for all $r\in\mathbb{R}$,  if $C_{r}$
is the net identically equal to $r$,
\[
J([C_{r}]_{\mathcal{U}})=r.
\]

\begin{remark} The field of Euclidean numbers is a hyperreal field
in the sense of Nonstandard Analysis. We do not use the name ``hyperreal
numbers''\ in order to emphasize the fact that $\mathbb{E}$ is
a peculiar hyperreal field that satisfies some properties, such as
the choice of $\mathfrak{L}$, which are not shared by a generic hyperreal
field. These properties are relevant in the definition of ultrafunctions.
The explanation of the choice of the name ``Euclidean numbers''\ can
be found for example in \cite{BL2021}. \end{remark}

Now we can define the notion of $\Lambda $-limit.

\begin{definition}
\label{EU}Given a net $\varphi \in \mathfrak{F}\left( \mathfrak{L},\mathbb{R}%
\right) ,$ we set%
\begin{equation*}
\lim_{\lambda \uparrow \Lambda }\varphi (\lambda )=J\left( \left[ \varphi %
\right] _{\mathcal{U}}\right)
\end{equation*}
\end{definition}

We list some properties of the $\Lambda $-limit which easily follow from
its definition:

\begin{proposition}
\label{p2}The $\Lambda $-limit satisfies the following properties:

\begin{enumerate}
\item \label{EU1}\textit{if }$\mathcal{U}$-eventually, $\varphi (\lambda
)=\psi (\lambda )$, then 
\begin{equation*}
\lim_{\lambda \uparrow \Lambda }\varphi (\lambda )=\lim_{\lambda \uparrow
\Lambda }\psi (\lambda );
\end{equation*}

\item \label{EU1+}\textit{if }$\mathcal{U}$-eventually,\textit{\ }$\varphi
(\lambda )=r\in \mathbb{R}$, then 
\begin{equation*}
\lim_{\lambda \uparrow \Lambda }\varphi (\lambda )=r;
\end{equation*}

\item \label{pippa}if $\mathcal{U}$-eventually,$\ \varphi (\lambda )>\psi
(\lambda ),$ then%
\begin{equation*}
\lim_{\lambda \uparrow \Lambda }\varphi (\lambda )>\lim_{\lambda \uparrow
\Lambda }\psi (\lambda );
\end{equation*}

\item \label{EU2}\textit{for all }$\varphi ,\psi \in \mathfrak{F}\left( 
\mathfrak{L},\mathbb{R}\right) ,$\emph{\ }%
\begin{eqnarray*}
\lim_{\lambda \uparrow \Lambda }\varphi (\lambda )+\lim_{\lambda \uparrow
\Lambda }\psi (\lambda ) &=&\lim_{\lambda \uparrow \Lambda }\left( \varphi
(\lambda )+\psi (\lambda )\right) ; \\
\lim_{\lambda \uparrow \Lambda }\varphi (\lambda )\cdot \lim_{\lambda
\uparrow \Lambda }\psi (\lambda ) &=&\lim_{\lambda \uparrow \Lambda }\left(
\varphi (\lambda )\cdot \psi (\lambda )\right) .
\end{eqnarray*}
\end{enumerate}
\end{proposition}

If a real net $x_{\lambda }$ admits the Cauchy limit, the relation between
the two limits is given by the following identity:%
\begin{equation}
\lim_{\lambda \rightarrow \Lambda }\ x_{\lambda }=st\left( \lim_{\lambda
\uparrow \Lambda }\ x_{\lambda }\right)  \label{bn}
\end{equation}
Other important relations between the two limits are the following:

\begin{proposition}
\begin{enumerate}
\item \label{lim1}If 
\[
\lim_{n\rightarrow \infty }\ x_{n}=L\in \mathbb{R}
\]
then, the limit of the net $n\mapsto x_{\left\vert \lambda \right\vert }$,
satisfies the relation%
\[
\lim_{\lambda \uparrow \Lambda }\ x_{\left\vert \lambda \right\vert }\sim L.
\]

\item \label{lim2}If
\[
\lim_{\lambda \uparrow \Lambda }\ x_{\lambda }=\xi \in \mathbb{E}
\]
and $\xi $ is bounded, then there exists a sequence $\lambda _{n}\in 
\mathfrak{L}$ such that%
\begin{equation*}
\lim_{n\rightarrow \infty }\ x_{\lambda _{n}}=st(\xi ).
\end{equation*}
\end{enumerate}
\end{proposition}

\textbf{Proof: }\ref{lim1} - For every $\varepsilon \in \mathbb{R}^{+},$
eventually%
\begin{equation*}
\left\vert x_{\left\vert \lambda \right\vert }-L\right\vert <\varepsilon
\end{equation*}%
and by Prop. \ref{p2}.(\ref{pippa}), 
\begin{equation*}
\left\vert \lim_{\lambda \uparrow \Lambda }\ x_{\left\vert \lambda
\right\vert }-L\right\vert <\varepsilon ;
\end{equation*}
hence, since $\varepsilon $ is arbitrary, 
\begin{equation*}
\lim_{\lambda \uparrow \Lambda }\ x_{\left\vert \lambda \right\vert }-L\sim
0.
\end{equation*}
\ref{lim2} - Set $x_{0}=st(\xi )$ and for every $n\in \mathbb{N}$, take $%
\lambda _{n}$ such that $\ x_{\lambda _{n}}\in B_{1/n}(x_{0}).$
$\square $

\subsection{Some remarks on the Euclidean numbers}

In this section, we will make some remarks for the readers who are
not used to hyperreal fields. These remarks are not relevant for the
rest of the paper,  but they can be useful to familiarize the reader
with the Euclidean field.

The new (and, for someone, surprising) fact is that every net has
a $\Lambda$-limit.  Probably the first question a newcomer would ask
is the following: what is the limit of the net 
\[
\vartheta(\lambda):=(-1)^{\left\vert \lambda\right\vert }
\]
since it takes the values $+1$ if $\left\vert \lambda\right\vert $ is even or $-1$ if $\left\vert \lambda\right\vert $ is odd.

Let us see what Def.~\ref{EU} tells us. If we set 
\[
Q^{+}:=\left\{ \lambda\in\mathfrak{L}\ |\ \vartheta(\lambda)=+1\right\} ,\qquad Q^{-}:=\left\{ \lambda\in\mathfrak{L}\ |\ \vartheta(\lambda)=-1\right\} 
\]
we have that either $Q^{+}\in\mathcal{U}$ or $Q^{-}\in\mathcal{U}$.
Hence $\vartheta(\lambda)$ is $\mathcal{U}$-eventually equal to
$+1$ or to $-1$ and therefore 
\[
\lim_{\lambda\uparrow\Lambda}(-1)^{\left\vert \lambda\right\vert }=1\quad\text{or}\quad\lim_{\lambda\uparrow\Lambda}(-1)^{\left\vert \lambda\right\vert }=-1.
\]
Which alternative occurs depends on the choice of $\mathcal{U}$.
Since this and similar questions are not relevant for this paper,
it is not necessary to choose $\mathcal{U}$ in a more specific way.
For other purposes, a more accurate choice of $\mathcal{U}$ can be
found for example in \cite{BL2021}. With that choice we have that
\[
\lim_{\lambda\uparrow\Lambda}(-1)^{\left\vert \lambda\right\vert }=1.
\]
The second question a newcomer might ask concerns the limit of the
divergent net defined by 
\[
\varphi(\lambda):=\left\vert \lambda\cap\mathbb{N}\right\vert .
\]
Let us put 
\begin{equation}
\alpha:=\lim_{\lambda\uparrow\Lambda}\ \left\vert \lambda\cap\mathbb{N}\right\vert \label{alfa}
\end{equation}
What can we say about $\alpha$? By Prop.~\ref{p2}-(\ref{pippa}),
we have $\alpha\notin\mathbb{R}$. In order to give a feeling of the ``meaning'' of $\alpha$, we relate
it to other infinite numbers. If $E\subset\Lambda$, we put
\begin{equation}
\mathfrak{num}\left(E\right)=\lim_{\lambda\uparrow\Lambda}\ |\lambda\cap E|.\label{num}
\end{equation}
Here $\left\vert F\right\vert \in\mathbb{N}$ denotes the number of
elements of the finite set $F$. 
If $E$ is a finite set, the net is eventually equal to the number
of elements of $E$; then, by Prop.~\ref{EU}-(\ref{EU1}),
\[
\mathfrak{num}\left(E\right)=\left\vert E\right\vert \in\mathbb{N}.
\]
If $E$ is an infinite set, $\mathfrak{num}\left(E\right)\notin\mathbb{N}$.
Hence limits like (\ref{num}) define a mathematical entity that extends
the notion of ``number of elements of a set'' to infinite sets.
The infinite number $\mathfrak{num}(E)$ is called the \textit{numerosity}
of $E$. The theory of numerosities can be considered as an extension of the
Cantorian theory of cardinal and ordinal numbers. The reader interested
in the details and developments of the theory is referred to \cite{benci95b,BDN2003,BDNF1,BL2021,BF}.

\begin{remark} In order to familiarize a newcomer with hyperreal
fields, in section \ref{lt} we have presented a construction of the
Euclidean numbers and the $\Lambda$-limit similar to a construction
of the real numbers and the Cauchy limit. In fact, we have the following
similarities:
\begin{itemize}
\item in order to construct $\mathbb{R}$, we start from $\mathbb{Q}$;
in order to construct $\mathbb{E}$, we start from $\mathbb{R}$;
\item in the first case we take the ring $\mathfrak{C}\left(\mathbb{N},\mathbb{Q}\right)$
of the Cauchy sequences; in the second case we take the ring of all
nets $\mathfrak{F}\left(\mathfrak{L},\mathbb{R}\right)$;
\item we have that 
\[
I_{\mathcal{C}}:=\left\{ \{x_{n}\}\in\mathfrak{C}\left(\mathbb{N},\mathbb{Q}\right)\ |\ \forall\varepsilon>0,\ \text{eventually }|x_{n}|<\varepsilon\right\} 
\]
is a maximal ideal in $\mathfrak{C}\left(\mathbb{N},\mathbb{Q}\right)$
and similarly we have that
\[
I_{\mathcal{U}}:=\left\{ \psi\in\mathfrak{F}\left(\mathfrak{L},\mathbb{R}\right)\ |\ \mathcal{U}\text{-eventually }\psi=0\right\} 
\]
is a maximal ideal in $\mathfrak{F}\left(\mathfrak{L},\mathbb{R}\right)$;
\item hence 
\[
\mathfrak{C}\left(\mathbb{N},\mathbb{Q}\right)/I_{\mathcal{C}}\qquad\text{and}\qquad\mathfrak{F}\left(\mathfrak{L},\mathbb{R}\right)/I_{\mathcal{U}}
\]
are fields;
\item if we want to consider the points of $\mathbb{R}$ and $\mathbb{E}$
as atoms\footnote{In set theory, an atom $a$ is any entity that is not a set, namely
$a$ is an atom if and only if $a\neq\varnothing$ and 
\[
\forall x,\ x\notin a.
\]
In NSA it is relevant that the elements of $\mathbb{E}$ and $\mathbb{R}$
be atoms and not equivalence classes.}, we define two field isomorphisms
\[
J_{\mathbb{R}}:\mathfrak{C}\left(\mathbb{N},\mathbb{Q}\right)/I_{\mathcal{C}}\rightarrow\mathbb{R}\qquad\text{and}\qquad J_{\mathbb{E}}:\mathfrak{F}\left(\mathfrak{L},\mathbb{R}\right)/I_{\mathcal{U}}\rightarrow\mathbb{E};
\]
\item finally we define the Cauchy limit and the $\Lambda$-limit as follows:
\[
\lim_{n\rightarrow\infty}x_{n}=J_{\mathbb{R}}\left(\{x_{n}\}+I_{\mathcal{C}}\right)\qquad\text{and}\qquad\lim_{\lambda\uparrow\Lambda}x_{\lambda}=J_{\mathbb{E}}\left(\{x_{\lambda}\}+I_{\mathcal{U}}\right).
\]

\end{itemize}
\end{remark}

\subsection{$\Lambda$-limit of sets and functions}

\label{hs}

In this section we extend the notion of $\Lambda$-limit to sets and
functions.

\begin{definition} Given a net of sets $E_{\lambda}\subset\mathbb{R}^{N}$,
the $\Lambda$-limit is defined as follows: 
\begin{equation}
\lim_{\lambda\uparrow\Lambda}E_{\lambda}:=\left\{ \lim_{\lambda\uparrow\Lambda}x_{\lambda}\ \big|\ \forall\lambda\in\mathfrak{L},\ x_{\lambda}\in E_{\lambda}\right\} \label{inter}
\end{equation}
and, for short, it will usually be denoted by $E_{\Lambda}$.

If a set is the $\Lambda$-limit of a net of sets, it will be called
\textbf{internal}. If the net $E_{\lambda}$ is $\mathcal{U}$-eventually
constant, then the $\Lambda$-limit will be denoted by $E^{\ast}$.
\end{definition}

\begin{remark} Since the $\Lambda$-limit of a constant net of sets
$E_{\lambda}\equiv E$ is different from $E$, it is not a limit in
the topological sense, but rather a sort of ``algebraic'' limit.
\end{remark}

If a real function is identified with its graph, the $\Lambda$-limit
of a net of functions is well defined. In this case we write
\[
f_{\Lambda}=\lim_{\lambda\uparrow\Lambda}f_{\lambda}.
\]
Clearly, $f_{\Lambda}$ is defined for every $x\in\mathbb{E}$ and
it takes values in $\mathbb{E}$. If $x=\lim_{\lambda\uparrow\Lambda}x_{\lambda}$,
we have
\begin{equation}
f_{\Lambda}(x)=\lim_{\lambda\uparrow\Lambda}f_{\lambda}(x_{\lambda}).\label{42}
\end{equation}
If $f_{\lambda}$ is $\mathcal{U}$-eventually constant, the following
notation is usually used:
\begin{equation}
f^{\ast}(x):=\lim_{\lambda\uparrow\Lambda}f(x_{\lambda}).\label{star}
\end{equation}
However, if the meaning is clear from the context, we will simply
write $f$ instead of $f^{\ast}$.

If $W(\mathbb{R})\subset\mathfrak{F}(\mathbb{R})$ is a vector space
of functions, we will use the following notation

\[
W(\mathbb{E}):=W(\mathbb{R})^{\ast}=\left\{ \lim_{\lambda\uparrow\Lambda}f_{\lambda}\ \big|\ \forall\lambda\in\mathfrak{L},\ f_{\lambda}\in W(\mathbb{R})\right\} .
\]

\begin{definition} \label{hs1} Given a net of sets $F_{\lambda}$
such that all the sets are finite, we say that the set $F_{\Lambda}$
is \textbf{hyperfinite}. \end{definition}

Hyperfinite sets share many properties of finite sets. For example,
a hyperfinite set $F_{\Lambda}\subset\mathbb{E}$ has a maximum $x_{\text{\textsc{max}}}$
and a minimum $x_{\text{\textsc{min}}}$, respectively given by
\[
x_{\text{\textsc{max}}}=\lim_{\lambda\uparrow\Lambda}\max(F_{\lambda}),\qquad x_{\text{\textsc{min}}}=\lim_{\lambda\uparrow\Lambda}\min(F_{\lambda}).
\]
Moreover, it is possible to ``add'' the elements of a hyperfinite
set; the \textbf{hyperfinite sum} of the elements of $F_{\Lambda}$
is defined as follows:
\begin{equation}
\sum_{x\in F_{\Lambda}}x=\lim_{\lambda\uparrow\Lambda}\sum_{x\in F_{\lambda}}x.\label{sum}
\end{equation}

\begin{definition} A hyperfinite set is called a \textbf{hyperfinite
grid} if for every $x\in\mathbb{R}$ there exists $\xi\in\Gamma$
such that $\xi\sim x$. \end{definition}

From now on we shall use peculiar grids, namely grids which satisfie
the following assumption:
\begin{equation}
\mathbb{R}\subset\Gamma.\label{linda}
\end{equation}
Moreover, given $\Omega\subset\mathbb{R}^{N}$, we set
\[
\Omega^{\circ}:=\Omega^{\ast}\cap\Gamma^{N}.
\]
Then, by (\ref{linda}), we have
\[
\Omega\subset\Omega^{\circ}\subset\Omega^{\ast}.
\]
Namely $\Omega^{\circ}$ is a hyperfinite set which contains $\Omega$;
it can be considered as a sort of ``compactification'' of $\Omega$.

\begin{definition} A space of grid functions is a family $\mathfrak{F}(\Gamma)$
of internal functions 
\[
u:\Gamma\rightarrow\mathbb{R}.
\]
\end{definition}

If $f\in\mathfrak{F}(\mathbb{E})$, the restriction of $f$ to $\Gamma$
is a grid function which we denote by $f^{\circ}$. Namely, if $f=\lim_{\lambda\uparrow\Lambda}f_{\lambda}$
and $x=\lim_{\lambda\uparrow\Lambda}x_{\lambda}\in\Gamma$, we have

\begin{equation}
f^{\circ}(x)=\lim_{\lambda\uparrow\Lambda}f_{\lambda}(x_{\lambda}).\label{giusi}
\end{equation}
For every $a\in\Gamma$,
\[
\chi_{a}(x)\in\mathfrak{F}(\Gamma)
\]
is a grid function, and hence every grid function can be represented
by the sum
\[
f(x)=\sum_{a\in\Gamma}f(a)\chi_{a}(x).
\]
Given $f\in\mathfrak{F}(\mathbb{R})$, we will write $f^{\circ}$
instead of $(f^{\ast})^{\circ}$, namely

\begin{equation}
f^{\circ}(x):=(f^{\ast})^{\circ}(x)=\lim_{\lambda\uparrow\Lambda}f(x_{\lambda})=\sum_{a\in\Gamma}f^{\ast}(a)\chi_{a}(x).\label{lina}
\end{equation}
Clearly, $\mathfrak{F}(\Gamma)$ contains a unique copy $f^{\circ}$
of every function $f\in\mathfrak{F}(\mathbb{R})$.

\section{Ultrafunctions}\label{PN}

In this paper, we introduce ultrafunctions axiomatically. The consistency
of these axioms will be proved in Section \ref{CSU}.

\subsection{The basic axioms}\label{BA}

The notion of ultrafunction is based on the following elements:
\begin{itemize}
\par 
\item a suitable hyperfinite grid $\Gamma$ such that 
\[
\omega:=\lim_{\lambda\uparrow\Lambda}\omega_{\lambda}=\max(\Gamma)=-\min(\Gamma),
\]
where $\omega$  is a basic parameter of the theory;
\item a suitable vector space $V(\mathbb{R})\subset\mathfrak{F}(\mathbb{R})$.
For reasons which will be discussed in Section \ref{EF}, the best
choice is the space of \textit{epilogic functions} defined
as follows: 
\[
V(\mathbb{R}):=\left\{ f\in BV_{loc}(\mathbb{R})\ \big|\ \forall x\in\mathbb{R},\ f(x)=\lim_{\varepsilon\to0^{+}}\frac{1}{2}\big[f(x+\varepsilon)+f(x-\varepsilon)\big]\right\} .
\]
Hence a function in the $BV_{loc}(\mathbb{R})$\footnote{Here $BV_{loc}(\mathbb{R})$ denotes the space of functions of locally bounded variation and not the space of equivalence classes of functions}  is equal to a function
in  $V(\mathbb{R})$.
\item a suitable net of finite-dimensional subspaces $V_{\lambda}(\mathbb{R})$
satisfying the following properties:
\begin{itemize}
\par 
\item $\forall\lambda\in\mathfrak{L}$, 
\[
V(\mathbb{R})\cap\lambda\subseteq V_{\lambda}(\mathbb{R});
\]hence 
\[
V(\mathbb{R})=\bigcup_{\lambda\in\mathfrak{L}}V_{\lambda}(\mathbb{R}).
\]
\item $\forall \lambda\in\mathfrak{L}$,  $V_{\lambda}(\mathbb{R})$ is large enough so that 
\begin{equation}
u,v\in C^{0,1}(\mathbb{R})\cap\lambda\Rightarrow uv\in V_{\lambda}(\mathbb{R});\label{gazza}
\end{equation}
this is a technical assumption used, for example, in the proof of
Th.\ref{TU}-(\ref{U7}).
\end{itemize}
\end{itemize}
\bigskip{}

The set of \textbf{ultrafunctions} $V(\Gamma)$ is an algebra of functions
$u:\Gamma\rightarrow\mathbb{E}$ over the field $\mathbb{E}$. We
introduce it axiomatically by two axioms presented in this section.
Then we will add two other axioms in Sections \ref{I} and \ref{D}
concerning the integral and the derivative, respectively.

\begin{axiom} \label{A} \textbf{(Approximation Axiom)} $u\in V(\Gamma)$
if and only if there exists a net of functions $u_{\lambda}\in V_{\lambda}(\mathbb{R})$
such that, for every point $x=\lim_{\lambda\uparrow\Lambda}x_{\lambda}\in\Gamma$,
\[
u(x)=\lim_{\lambda\uparrow\Lambda}u_{\lambda}(x_{\lambda}).
\]
\end{axiom}

Thus every ultrafunction can be seen as the pointwise $\Lambda$-limit
of a net of functions constrained by the limitations imposed by the
choice of the net $V_{\lambda}(\mathbb{R})$. From now on, if $u\in V(\Gamma)$,
$u_{\lambda}$ will denote a net converging to $u$ in this sense.

\begin{axiom} ($\mathbf{\chi}$-\textbf{Axiom}) For every point $a\in\Gamma$,
$\chi_{a}\in V(\Gamma)$, namely there exists a net $\sigma_{a_{\lambda}}\in V_{\lambda}(\mathbb{R})$
such that 
\[
\chi_{a}(x)=\lim_{\lambda\uparrow\Lambda}\sigma_{a_{\lambda}}(x_{\lambda}).
\]
\end{axiom}

\begin{remark} If $a\in\mathbb{R}$ the definition of $\chi_{a}(x)$
is ambiguous, but in general it is clear from the context whether
$\chi_{a}(x)$ is defined only for $x\in\mathbb{R}$ or for all $x\in\Gamma$.
\end{remark}

By this axiom every ultrafunction can be written as follows: 
\begin{equation}
u(x)=\sum_{a\in\Gamma}u(a)\chi_{a}(x),\qquad u(a)\in\mathbb{E}.\label{pu}
\end{equation}
where the sum is defined by (\ref{sum}). Moreover, this axiom implies
that every real function can be extended to an ultrafunction. Given
any function $f$, we can define 
\[
f^{\circ}(x)=\sum_{a\in\Gamma}f(a)\chi_{a}(x)=\lim_{\lambda\uparrow\Lambda}\sum_{a\in\Gamma_{\lambda}}f(a_{\lambda})\sigma_{a_{\lambda}}(x_{\lambda}).
\]
Clearly, if $x\in\mathbb{R}$, $f^{\circ}(x)=f(x)$ and hence $f^{\circ}$
extends $f$ to $\Gamma$. However there are ultrafunctions which
are not extensions of real functions; for example $\chi_{a}$ when
$a\in\Gamma\setminus\mathbb{R}$.

We can distinguish two main types of ultrafunctions extending real
functions:

\bigskip{}

(1)   $[V(\mathbb{R})]^{\circ}:=\{f^{\circ}\mid f\in V(\mathbb{R})\};$

\bigskip{}

(2)  $[\mathfrak{F}(\mathbb{R})]^{\circ}:=\{f^{\circ}\mid f\in\mathfrak{F}(\mathbb{R})\}.$

\bigskip{}

If $f^{\circ}\in[V(\mathbb{R})]^{\circ}$, then $\mathcal{U}$-eventually
$f\in V_{\lambda}(\mathbb{R})$, namely 
\begin{equation}
\exists Q\in\mathcal{U},\ \forall\lambda\in Q,\ \forall x\in\mathbb{R},\ f_{\lambda}(x)=f(x).\label{ruben}
\end{equation}

If $f^{\circ}\in[\mathfrak{F}(\mathbb{R})]^{\circ}$ and $x\in\mathbb{R}$,
then $\mathcal{U}$-eventually $f_{\lambda}(x)=f(x)$, namely 
\[
\forall x\in\mathbb{R},\ \exists Q\in\mathcal{U},\ \forall\lambda\in Q,\ f_{\lambda}(x)=f(x).
\]

Notice that $[V(\mathbb{R})]^{\circ}$ and $[\mathfrak{F}(\mathbb{R})]^{\circ}$
are vector spaces over $\mathbb{R}$ but not over $\mathbb{E}$. As
we will see, it is useful to define also suitable subspaces of $V(\Gamma)$
over $\mathbb{E}$. If $W(\mathbb{R})\subset\mathfrak{F}(\mathbb{R})$
is a real vector space, we set 
\begin{equation}
W(\Gamma):=\left\{ \lim_{\lambda\uparrow\Lambda}u_{\lambda}(x_{\lambda})\ \big|\ \forall\lambda,\ u_{\lambda}\in W(\mathbb{R})\cap V_{\lambda}(\mathbb{R})\right\} =\left\{ u_{|\Gamma}\mid u\in W(\mathbb{E})\right\} .\label{W}
\end{equation}

Hence, if $W(\mathbb{R})\supseteq V(\mathbb{R})$, we have $W(\Gamma)=V(\Gamma)$,
while if $W(\mathbb{R})\nsupseteq V(\mathbb{R})$, it follows that
$W(\Gamma)\subset V(\Gamma)$.

\subsection{The pointwise integral}\label{I}

Since 
\begin{equation}
V(\mathbb{R})\subset\mathcal{L}^{1}_{loc}(\mathbb{R}),\label{VL1}
\end{equation}
the following definition is well posed.

\begin{definition} \label{DI} The \textbf{pointwise integral} 
\[
\int^{\circ}:V(\Gamma)\rightarrow\mathbb{E}
\]
is defined as follows: 
\[
\int^{\circ}u(x)\,dx=\lim_{\lambda\uparrow\Lambda}\int^{\omega_{\lambda}}_{-\omega_{\lambda}}u_{\lambda}(x)\,dx.
\]
\end{definition}

The reason for this name is given by the following result.

\begin{proposition} If $u\in V(\Gamma)$, then 
\begin{equation}
\int^{\circ}u(x)\,dx=\sum_{a\in\Gamma}u(a)d(a),\label{int}
\end{equation}
where 
\begin{equation}
d(a):=\int^{\circ}\chi_{a}(x)\,dx.\label{int4}
\end{equation}
\end{proposition}

\textbf{Proof}. By the $\mathbf{\chi}$-axiom there exists a net $\sigma_{a_{\lambda}}\in V_{\lambda}(\mathbb{R})$
such that 
\begin{equation}
\forall a=\lim_{\lambda\uparrow\Lambda}a_{\lambda}\in\Gamma,\qquad\chi_{a}=\lim_{\lambda\uparrow\Lambda}\sigma_{a_{\lambda}}.\label{chu}
\end{equation}
Hence, if $u\in V(\Gamma)$, 
\begin{equation}
u_{\lambda}(x)=\sum_{a_{\lambda}\in\Gamma_{\lambda}}u_{\lambda}(a_{\lambda})\sigma_{a_{\lambda}}(x).\label{urano}
\end{equation}
We set 
\begin{equation}
d(a_{\lambda}):=\int^{\omega_{\lambda}}_{-\omega_{\lambda}}\sigma_{a_{\lambda}}(x)\,dx.\label{chu+}
\end{equation}
Then, by (\ref{urano}), 
\begin{eqnarray*}
\int^{\omega_{\lambda}}_{-\omega_{\lambda}}u_{\lambda}(x)\,dx & = & \int^{\omega_{\lambda}}_{-\omega_{\lambda}}\sum_{a_{\lambda}\in\Gamma_{\lambda}}u_{\lambda}(a_{\lambda})\sigma_{a_{\lambda}}(x)\,dx\\
 & = & \sum_{a_{\lambda}\in\Gamma_{\lambda}}u_{\lambda}(a_{\lambda})\int^{\omega_{\lambda}}_{-\omega_{\lambda}}\sigma_{a_{\lambda}}(x)\,dx\\
 & = & \sum_{a_{\lambda}\in\Gamma_{\lambda}}u_{\lambda}(a_{\lambda})d(a_{\lambda}).
\end{eqnarray*}
Hence 
\[
\int^{\circ}u(x)\,dx=\lim_{\lambda\uparrow\Lambda}\int^{\omega_{\lambda}}_{-\omega_{\lambda}}u_{\lambda}(x)\,dx=\lim_{\lambda\uparrow\Lambda}\sum_{a_{\lambda}\in\Gamma_{\lambda}}u_{\lambda}(a_{\lambda})d(a_{\lambda})=\sum_{a\in\Gamma}u(a)d(a).
\]
$\square$

\bigskip{}

We may think of $d(a)$ as the ``measure'' of the point $a\in\Gamma$.
The pointwise integral extends the usual Lebesgue integral from $V_{c}(\mathbb{R})$
to $V(\Gamma)$. In fact, if $f\in V_{c}(\mathbb{R})$, $\mathcal{U}$-eventually
$\mathrm{supp}(f)\subset[-\omega_{\lambda},\omega_{\lambda}]$ and
hence
\begin{equation}
\int^{\circ}f^{\circ}(x)\,dx=\lim_{\lambda\uparrow\Lambda}\int^{\omega_{\lambda}}_{-\omega_{\lambda}}f(x)\,dx=\lim_{\lambda\uparrow\Lambda}\int f(x)\,dx=\int f(x)\,dx.\label{mara}
\end{equation}

Of course, if $f$ does not have compact support, the pointwise integral
may assume infinite values. For example

\[
\int^{\circ}e^{x}\,dx=\lim_{\lambda\uparrow\Lambda}\int^{\omega_{\lambda}}_{-\omega_{\lambda}}e^{x}\,dx=\lim_{\lambda\uparrow\Lambda}\left(e^{\omega_{\lambda}}-e^{-\omega_{\lambda}}\right)=e^{\omega}-e^{-\omega}.
\]
Equality (\ref{mara}) cannot hold for every Lebesgue integrable function
with compact support. In fact, if $a\in\mathbb{R}$,
\[
\int\chi_{a}(x)\,dx=0
\]
but, by (\ref{int}) and (\ref{int4}),
\[
\int^{\circ}\chi_{a}(x)\,dx=d(a)\neq0
\]
at least for some $a$. This fact is quite natural: when we work in
a non-Archimedean framework infinitesimals matter and cannot be neglected,
as happens in the Riemann and Lebesgue integrals. This also shows
that it is necessary to use a different symbol in order to distinguish
the pointwise integral from the Lebesgue integral (here we use $\int^{\circ}$).

\begin{theorem} \label{LI} If $f\in\mathcal{L}^{1}(\mathbb{R})$,
then 
\[
\int^{\circ}f^{\circ}(x)\,dx\sim\int f(x)\,dx.
\]
\end{theorem}

\textbf{Proof}. Since $V_{c}(\mathbb{R})$ is dense in $\mathcal{L}^{1}(\mathbb{R})$,
there exists a net $f_{\lambda}\in V_{\lambda}(\mathbb{R})$ such
that
\[
\lim_{\lambda\rightarrow\Lambda}\int^{\omega_{\lambda}}_{-\omega_{\lambda}}|f_{\lambda}-f|\,dx=0.
\]

Then,
\begin{eqnarray*}
\left|\int^{\circ}f^{\circ}(x)\,dx-\int f(x)\,dx\right| & = & \left|\lim_{\lambda\uparrow\Lambda}\int^{\omega_{\lambda}}_{-\omega_{\lambda}}f_{\lambda}(x)\,dx-\lim_{\lambda\rightarrow\Lambda}\int^{\omega_{\lambda}}_{-\omega_{\lambda}}f_{\lambda}(x)\,dx\right|\\
 & \sim & \left|\lim_{\lambda\rightarrow\Lambda}\int^{\omega_{\lambda}}_{-\omega_{\lambda}}f_{\lambda}(x)\,dx-\lim_{\lambda\rightarrow\Lambda}\int^{\omega_{\lambda}}_{-\omega_{\lambda}}f_{\lambda}(x)\,dx\right|=0.
\end{eqnarray*}
Hence 
\[
\int^{\circ}f^{\circ}(x)\,dx\sim\int f(x)\,dx.
\]
$\square$

\bigskip{}

Now we assume the following axiom, which is independent of the basic
axioms but is completely natural.

\begin{axiom} \label{IA} \textbf{(Integral axiom)} We require that
\[
\forall a\in\Gamma,\qquad\int^{\circ}\chi_{a}(x)\,dx>0.
\]
\end{axiom}

Since $d(a)=\int^{\circ}\chi_{a}(x)\,dx>0$, the pointwise integral
allows us to define the scalar product

\begin{equation}
\int^{\circ}u(x)v(x)\,dx=\sum_{x\in\Gamma}u(x)v(x)d(x)\label{psc}
\end{equation}
and the norm of an ultrafunction

\[
\|u\|=\left(\sum_{a\in\Gamma}|u(a)|^{2}d(a)\right)^{1/2}=\left(\int^{\circ}|u(x)|^{2}dx\right)^{1/2}.
\]

Moreover, the integral axiom also allows the definition of the \textbf{delta
(Dirac) ultrafunction}
\begin{equation}
\delta_{a}(x)=\frac{\chi_{a}(x)}{d(a)}.\label{dirac2}
\end{equation}
As expected, we have
\begin{equation}
\int^{\circ}u(x)\delta_{a}(x)\,dx=\sum_{x\in\Gamma}u(x)\frac{\chi_{a}(x)}{d(a)}d(x)=u(a).\label{delta}
\end{equation}
By (\ref{dirac2}), the delta ultrafunctions are mutually orthogonal
with respect to the pointwise scalar product. Indeed,
\[
\int^{\circ}\delta_{a}(x)\delta_{b}(x)\,dx=\frac{1}{d(a)d(b)}\int^{\circ}\chi_{a}(x)\chi_{b}(x)\,dx=\frac{\delta^{b}_{a}}{d(a)^{2}}.
\]
Hence, if normalized, they provide an orthonormal basis, called the\textbf{delta-basis},
\[
\left\{ \frac{\delta_{a}(x)}{\sqrt{\delta_{a}(a)}}\right\} _{a\in\Gamma}=\left\{ \frac{\chi_{a}(x)}{\sqrt{d(a)}}\right\} _{a\in\Gamma}.
\]

\subsection{The generalized derivative}\label{D}

The derivative of a function $f\in C^{1}(\mathbb{R})$ will be denoted
by $\partial f$. If $f\in C^{1}(\mathbb{R})\cap V(\mathbb{R})=C^{1,1}(\mathbb{R})$,
then $\partial f\in V(\mathbb{R})$ and hence it is natural to define
the generalized derivative $D$ as follows: 
\begin{equation}
Du(x)=\lim_{\lambda\uparrow\Lambda}\partial u_{\lambda}(x_{\lambda}).\label{11}
\end{equation}

However, this restriction is too limiting for many applications: we
need a notion of \textit{generalized} derivative that includes, in
some sense, the notion of \textit{weak} derivative. In particular,
we require that $\forall u\in W^{1,1}_{loc}(\mathbb{R})\cap V(\mathbb{R})$
and $\forall v\in V(\Gamma)$ 
\begin{equation}
\int^{\circ}Du^{\circ}(x)v(x)\,dx=\lim_{\lambda\uparrow\Lambda}\int^{\omega_{\lambda}}_{-\omega_{\lambda}}\partial u_{\lambda}v_{\lambda}\,dx.\label{l2}
\end{equation}

Moreover, we would like the derivative of the Heaviside function,
$H(x):=\frac{1}{2}[sign(x)+1]$, to satisfy 
\begin{equation}
DH(x)=\delta_{0}(x).\label{13}
\end{equation}

For this reason we have assumed assumed that 
\begin{equation}
V(\mathbb{R})\subset BV_{loc}(\mathbb{R})\label{V3}
\end{equation}
This assumption allows us to define a generalized derivative satisfying
(\ref{11}) and (\ref{l2}) and which is defined for \textit{every}
ultrafunction.

We recall that $BV_{loc}(\mathbb{R})$ denotes the set of locally
bounded variation functions. It is well known that if $f\in BV_{loc}(\mathbb{R})$,
its derivative $\partial f$ is a Radon measure $\mu_{\partial f}$.
For every Radon measure $\mu$ and every Borel function $\varphi$,
we use the notation 
\begin{equation}
\langle\mu,\varphi\rangle:=\int\varphi(x)\,d\mu.\label{1923}
\end{equation}

We are therefore led to the following definition.

\begin{definition} \label{DA} The \textbf{generalized derivative}
\[
D:V(\Gamma)\rightarrow V(\Gamma)
\]
is defined as follows: for every $u\in V(\Gamma)$ and every $a\in\Gamma$
\begin{equation}
Du(a)=\lim_{\lambda\uparrow\Lambda}\langle\partial u_{\lambda},\delta_{a_{\lambda}}\rangle,\label{lillina}
\end{equation}
where $\delta_{a_{\lambda}}\in V_{\lambda}(\mathbb{R})$ is a net
of functions such that 
\[
\lim_{\lambda\uparrow\Lambda}\delta_{a_{\lambda}}=\delta_{a}.
\]
\end{definition}

If $f\in C^{1,1}_{c}(\mathbb{R})$, then $\partial f\in V(\mathbb{R})$
and we have 
\[
Df^{\circ}(a)=\lim_{\lambda\uparrow\Lambda}\langle\delta_{a_{\lambda}},\partial f\rangle=\lim_{\lambda\rightarrow\Lambda}\int^{\omega_{\lambda}}_{-\omega_{\lambda}}\partial f(x)\delta_{a_{\lambda}}\,dx=(\partial f)^{\circ}(a).
\]

Hence our definition agrees with the usual notion of derivative, at
least for functions in $C^{1,1}_{c}(\mathbb{R})$. Moreover it is
easy to check that also (\ref{l2}) and (\ref{13}) are satisfied.

It is useful to require that the generalized derivative satisfy some
other properties which cannot be deduced from the basic axioms. Thus
we assume the following axiom.

\begin{axiom} \label{AD} \textbf{(Derivative axiom)} The \textbf{generalized
derivative} satisfies the following properties:
\begin{enumerate}
\par 
\item \label{2+} for every point $a\in\Gamma$, 
\begin{equation}
\mathfrak{supp}[D\chi_{a}]\subset\mathfrak{mon}(a);\label{loc}
\end{equation}
\item \label{4} for every $u\in V(\Gamma)$ 
\begin{equation}
Du=0\Leftrightarrow u=c\mathbf{1}^{\circ},\qquad c\in\mathbb{E}.\label{A1}
\end{equation}
where $\mathbf{1}$ denotes the ultrafunction identically equal to
$1$.
\end{enumerate}
\end{axiom}

Let us now make some remarks on this axiom. If $f\in C^{1}(\mathbb{R})$,
then 
\begin{equation}
supp(\partial f)\subseteq supp(f).\label{88}
\end{equation}

However it is not possible to require this relation for all ultrafunctions.
Indeed, if it were true, we would obtain 
\[
D\delta_{a}=k\delta_{a},\qquad k\in\mathbb{E},
\]
which is clearly contradictory. Hence (\ref{loc}) can be seen as
a weak formulation of (\ref{88}) guaranteeing a form of locality
of the generalized derivative. By virtue of this requirement it follows
that 
\[
Df^{\circ}(a)=(\partial f)^{\circ}(a)
\]
not only when $f\in C^{1,1}_{c}(\mathbb{R})$, but it is sufficient
that $f\in C^{1,1}$ in a neighborhood of $a$.

As far as (\ref{A1}) is concerned, clearly the implication "$\Leftarrow$"
is an immediate consequence of Def.~\ref{DA}, whereas the implication
"$\Rightarrow$" is not trivial. Indeed, if $f\in C^{1}(\Omega)$,
"$\Rightarrow$" holds only if $\Omega$ is connected. Hence this
implication cannot be deduced from the basic axioms and must be stated
as an independent axiom.

One might also wish to require the Leibniz rule 
\[
D(uv)=Duv+uDv
\]
but it is easy to check that it cannot be satisfied by the idempotent
functions $\chi_{E}\neq\mathbf{1}^{\circ}$. Indeed, by the Leibniz
rule we would have 
\[
D\chi^{2}_{E}=2\chi_{E}D\chi_{E}
\]
and since $\chi^{2}_{E}=\chi_{E}$ we deduce 
\[
D\chi_{E}=2\chi_{E}D\chi_{E}.
\]

Hence for every $x\in E$ 
\begin{equation}
D\chi_{E}(x)=0,\label{schw}
\end{equation}
which contradicts (\ref{A1}) and any reasonable generalization of
the notion of derivative. In fact the Schwartz impossibility theorem
states that the Leibniz rule cannot be satisfied by any differential
algebra containing the continuous functions (see \cite{Schwartz},
\cite{algebra}).

Therefore, if we want a notion of generalized functions suitable for
nonlinear problems, we need an algebra of functions and hence we are
forced to renounce the Leibniz rule for \textit{all} pairs of functions.
Nevertheless, this rule is satisfied by sufficiently regular ultrafunctions
(see Sec.~\ref{RU}).

\subsection{The space of epilogic functions}\label{EF}

In this section we discuss some properties of $V(\mathbb{R})$. First
of all we equip $V(\mathbb{R})$ with the following family of seminorms:

\begin{equation}
\left\Vert f\right\Vert _{V([-n,n])}:=\left\Vert \partial f\right\Vert _{\mathfrak{M}([-n,n])}+\|f\|_{C^{0}([-n,n])}
\label{topology}
\end{equation}

where $\mathfrak{M}([-n,n])$ denotes the vector space of Radon measures
on $[-n,n]$ endowed with the seminorm

\[
\left\Vert \mu\right\Vert _{\mathfrak{M}([-n,n])}:=\sup\left\{ \frac{\langle\mu,\varphi\rangle}{\|\varphi\|_{C^{0}([-n,n])}}\ \big|\ \varphi\in C^{0}([-n,n])\right\} .
\]

Given a function $f\in BV_{loc}(\mathbb{R})$, the limits

\[
\lim_{\varepsilon\to0^{+}}f(x+\varepsilon),\qquad\lim_{\varepsilon\to0^{+}}f(x-\varepsilon)
\]

exist. Hence the operator

\[
(\cdot)^{\text{\textsc{epl}}}:BV_{loc}(\mathbb{R})\rightarrow V(\mathbb{R}),\qquad f(x)\mapsto f^{\text{\textsc{epl}}}(x)=\lim_{\varepsilon\to0^{+}}\frac{1}{2}\big[f(x+\varepsilon)+f(x-\varepsilon)\big]
\]

is well defined and the space of epilogic functions can be characterized
as

\[
V(\mathbb{R})=\left\{ f\in BV_{loc}(\mathbb{R})\ \big|\ f^{\text{\textsc{epl}}}(x)=f(x)\right\} .
\]

We now list some properties of $V(\mathbb{R})$ that will be useful
in the following.

\begin{theorem} \label{diana} The following properties hold:
\begin{enumerate}
\par 
\item \label{b2} If $f\in V(\mathbb{R})$, every point is a Lebesgue point,
namely

\[
f(x)=\lim_{\varepsilon\to0^{+}}\left(\frac{1}{\varepsilon}\int^{x+\varepsilon/2}_{x-\varepsilon/2}f(y)\,dy\right).
\]

\item \label{b3} $V(\mathbb{R})$ is a module over the ring $C^{0,1}(\mathbb{R})$,
namely $\varphi\in C^{0,1}(\mathbb{R})$ and $f\in V(\mathbb{R})$
imply $\varphi f\in V(\mathbb{R})$.
\item \label{UC} $V(\mathbb{R})$ is closed with respect to the topology
of uniform convergence.
\item \label{U4} $V(\mathbb{R})$ is complete with respect to the seminorms  (\ref{topology}).
\item \label{UR} If $f\in V(\mathbb{R})$, then it is Riemann integrable.
\item \label{US}   If $f\in V(\mathbb{R})$, then\[f(b)-f(a)=\left\langle \partial f,\chi_{[a,b]}\right\rangle \] 
\end{enumerate}
\end{theorem}

\textbf{Proof}.

(\ref{b2}) This easily follows from the definition.

(\ref{b3}) Given $f\in V([-n,n])$ and $\varphi\in C^{0,1}([-n,n])$,
\begin{eqnarray*}
(\varphi f)^{\text{\textsc{epl}}}(x) 
& = & \lim_{\varepsilon\to0^{+}}\left(\frac{1}{\varepsilon}\int^{x+\varepsilon/2}_{x-\varepsilon/2}\varphi(y)f(y)\,dy\right)\\
 & = & \lim_{\varepsilon\to0^{+}}\left(\frac{\varphi(x)}{\varepsilon}\int^{x+\varepsilon/2}_{x-\varepsilon/2}f(y)\,dy\right)\\
 & = & \varphi(x)f^{\text{\textsc{epl}}}(x)=\varphi(x)f(x).
\end{eqnarray*}

(\ref{UC}) Let $u_{n}$ be a sequence of epilogic functions uniformly
convergent to $u_{0}$. If we choose $\varepsilon_{0}>0$ and take
$n$ sufficiently large, then

\[
\|u_{n}-u_{0}\|_{C^0}\le\varepsilon_{0}.
\]

Hence 
\begin{eqnarray*}
\|u_{n}-u^{\text{\textsc{epl}}}_{0}\|_{C^0} 
& = & \sup\left|\lim_{\varepsilon\to0^{+}}\left(\frac{1}{\varepsilon}\int^{x+\varepsilon/2}_{x-\varepsilon/2}u_{n}(y)\,dy\right)-\lim_{\varepsilon\to0^{+}}\left(\frac{1}{\varepsilon}\int^{x+\varepsilon/2}_{x-\varepsilon/2}u_{0}(y)\,dy\right)\right|\\
 & = & \sup\left|\lim_{\varepsilon\to0^{+}}\left(\frac{1}{\varepsilon}\int^{x+\varepsilon/2}_{x-\varepsilon/2}|u_{n}(y)-u_{0}(y)|\,dy\right)\right|\\
 & \le & \|u_{n}-u_{0}\|_{C^0}\le\varepsilon_{0}.
\end{eqnarray*}

Thus

\[
\lim_{n\to\infty}\|u_{n}-u^{\text{\textsc{epl}}}_{0}\|_{\sup}=0.
\]
(\ref{U4}) It follows from the completeness of $BV$  and (\ref{UC})

(\ref{UR}) Every $f\in BV_{loc}(\mathbb{R})$ is the difference of
two monotone increasing functions. The conclusion follows from the
fact that monotone functions are Riemann integrable. 

(\ref{US}) Let $\{f_{n}\}$ be a sequence of $C^{1}$-functions converging to $f$ in the norm (\ref{topology});  then we have that
\[f_{n}(b)-f_{n}(a)=\int^{b}_{a}f_{n}(x)\,dx=\left\langle \partial f_{n},\chi_{[a,b]}\right\rangle .\]
Taking the limit we get the result.

$\square$

If $f\in W^{1,\infty}_{loc}(\mathbb{R}^{N})$, its weak partial derivative
$[g]_{L^{\infty}_{loc}}\in L^{\infty}_{loc}(\mathbb{R}^{N})$, and
hence it makes sense to define $\partial f$ pointwise by setting

\begin{equation}
\partial f=g^{\text{\textsc{epl}}}.\label{poppa+}
\end{equation}

From now on $\partial_{i}f$ will denote such a function. Using this
notation, by (\ref{l2}) we have

\[
Df^{\circ}=(\partial f)^{\circ}
\]
for every $f\in W^{1,\infty}_{loc}(\mathbb{R})$ (and not only for
$f\in C^{1,1}_{loc}(\mathbb{R})$).
\bigskip{}

\textbf{Example}. If $f(x):=\max\{0,x\}$, then

\[
\partial f(x)=H(x)=\frac{1}{2}[sign(x)+1].
\]

\begin{remark} Since $\|f\|_{BV([-n,n])}=0$ does not imply $f(x)=0$,
usually one works with equivalence classes $[f]_{BV_{loc}}$. Unfortunately
this space is not suitable here since a function $f\in[f]_{BV_{loc}}$
is not pointwise defined, and hence the operator $(\cdot)^{\circ}$
is not well defined.

For example, if we want to define $\left(\frac{x}{|x|}\right)^{\circ}$
we are forced to choose an arbitrary real value at $x=0$. We have
overcome this difficulty by taking the space of \textbf{epilogic functions}
\[
V(\mathbb{R})=\{f\in BV_{loc}(\mathbb{R})\mid f^{\text{\textsc{epl}}}(x)=f(x)\}.
\]

The name comes from the Greek $\varepsilon\pi\iota\lambda o\gamma\eta$
(= ``choice''), since each function in $V(\mathbb{R})$ is a selected
representative of an equivalence class in $[f]_{BV_{loc}}$. \end{remark}

\begin{remark} $V(\mathbb{R})$ can also be characterized as the
largest space of pointwise defined functions such that

\[
V'([-n,n])\subseteq V([-n,n]).
\]

This shows that $V(\mathbb{R})$ is the largest space in which it
is possible to define a generalized derivative as in Def.~\ref{DA}.
\end{remark} 

\subsection{Regular and smooth ultrafunctions}\label{RU}

As we already remarked, the Leibniz rule does not hold for ultrafunctions
and it is not possible to define a generalized derivative which satisfies
this property. Hence it is interesting to investigate the subspaces
of ultrafunctions for which the Leibniz rule holds and, more generally,
to determine spaces in which many of the usual properties of smooth
functions are satisfied.

\begin{definition} \label{ddd} For every $m\in\mathbb{N}\cup\{\infty\}$
we set 
\begin{equation}
V^{m}(\Gamma):=\left\{ u\in V(\Gamma)\ \big|\ u_{\lambda}\in C^{m,1}(\mathbb{R})\cap\lambda\right\} .\label{pipi+}
\end{equation}

If $u\in V^{m}(\Gamma)$ we say that $u$ is \textit{$m$-regular}.
If $u\in V^{\infty}(\Gamma)$ we say that $u$ is \textit{smooth}.
\end{definition}

Let us now describe the main properties of the spaces $V^{m}(\Gamma)$.

\begin{theorem} \label{TU} The spaces of regular ultrafunctions
satisfy the following properties:
\begin{enumerate}
\par 
\item \label{U3} If $f\in C^{m,1}(\mathbb{R})$, then $f^{\circ}\in V^{m}(\Gamma)$.
\item \label{U1} If $u\in V^{m}(\Gamma)$ and $m\ge1$, then $Du\in V^{m-1}(\Gamma)$.
\item \label{U00} If $u,v\in V^{m}(\Gamma)$, then 
\[
\int^{\circ}uv\,dx=\lim_{\lambda\uparrow\Lambda}\int^{\omega_{\lambda}}_{-\omega_{\lambda}}u_{\lambda}v_{\lambda}\,dx.
\]
\item \label{U0} If $u\in V^{m}(\Gamma)$, then 
\[
Du=\lim_{\lambda\uparrow\Lambda}(\partial u_{\lambda}).
\]
\item \label{U8} If $u\in V^{1}(\Gamma)$, then 
\[
\mathfrak{supp}(Du)\subset\mathfrak{supp}(u).
\]
\item \label{U7} If $u,v\in V^{1}(\Gamma)$, then the Leibniz rule holds:
\[
D(uv)=Duv+uDv.
\]
\end{enumerate}
\end{theorem}

\textbf{Proof}.

(\ref{U3}) It follows directly from (\ref{pipi+}).

(\ref{U1}) Trivial.

(\ref{U00}) If $u,v\in V^{m}(\Gamma)$ then $u_{\lambda},v_{\lambda}\in C^{0,1}(\mathbb{R})\cap\lambda$
and by (\ref{gazza}) we have $u_{\lambda}v_{\lambda}\in V_{\lambda}(\mathbb{R})$.
Hence, by Def.~\ref{DI},

\[
\lim_{\lambda\uparrow\Lambda}\int^{\omega_{\lambda}}_{-\omega_{\lambda}}u_{\lambda}v_{\lambda}\,dx=\int^{\circ}uv\,dx.
\]

(\ref{U0}) By (\ref{lillina}), for every $v\in V^{0}(\Gamma)$,

\[
\int^{\circ}Du(x)v(x)\,dx=\lim_{\lambda\uparrow\Lambda}\int^{\omega_{\lambda}}_{-\omega_{\lambda}}\partial u_{\lambda}v_{\lambda}\,dx.
\]

Moreover, since $\partial u_{\lambda},v_{\lambda}\in C^{0,1}(\mathbb{R})\cap\lambda$,
by (\ref{gazza}) we have $\partial u_{\lambda}v_{\lambda}\in V_{\lambda}(\mathbb{R})$.
Hence

\[
\int^{\circ}Du(x)v(x)\,dx=\int^{\circ}\left(\lim_{\lambda\uparrow\Lambda}\partial u_{\lambda}(x)v_{\lambda}(x)\right)dx=\int^{\circ}\left(\lim_{\lambda\uparrow\Lambda}\partial u_{\lambda}(x)\right)v(x)\,dx.
\]

From this the conclusion follows.

(\ref{U8}) If $x\in\mathfrak{supp}(Du)$, by (\ref{U0}), $\mathcal{U}$-eventually
$x_{\lambda}\in supp(\partial u_{\lambda})$ and hence $x_{\lambda}\in supp(u_{\lambda})$.
Therefore $x\in\mathfrak{supp}(u)$.

(\ref{U7}) By (\ref{gazza}), $u_{\lambda}v_{\lambda}\in V_{\lambda}(\mathbb{R})$.
Then by (\ref{U0})

\begin{eqnarray*}
D(uv) & = & \lim_{\lambda\uparrow\Lambda}\left[\partial(u_{\lambda}v_{\lambda})\right]\\
 & = & \lim_{\lambda\uparrow\Lambda}\left[\partial u_{\lambda}v_{\lambda}+u_{\lambda}\partial v_{\lambda}\right]\\
 & = & \lim_{\lambda\uparrow\Lambda}(\partial u_{\lambda})\,v+u\,\lim_{\lambda\uparrow\Lambda}(\partial v_{\lambda})\\
 & = & Duv+uDv.
\end{eqnarray*}

$\square$

\bigskip{}

A special role is played by the space of smooth ultrafunctions $V^{\infty}(\Gamma)$;
in fact, as we will see in Section \ref{ud}, it contains a ``copy''
of every distribution.

\begin{remark} \label{fss} It is possible to define different types
of regular ultrafunctions. Namely we may choose different subspaces
of $V(\Gamma)$ satisfying suitable conditions. For example we can
consider

\[
W^{m}(\Gamma^{N})=\left\{ u\in V(\Gamma)\ \big|\ u_{\lambda}\in W^{m}(\mathbb{R})\cap V_{\lambda}(\mathbb{R})\right\} .
\]

We have $V^{m}(\Gamma)\subset W^{m}(\Gamma)$ and hence the functions
in $W^{m}(\Gamma)$ satisfy fewer properties. Of course the choice
of a particular space depends on the problems we want to study. An
analogy can be made with the theory of distributions: in that case
the spaces $C^{m}$ and the Sobolev spaces $W^{m,p}$ can be viewed
as subspaces of $\mathcal{D}'$ exhibiting different kinds of regularity.
\end{remark}

\subsection{The interior part of the grid}

If $f,g\in C^{1}_{c}(\mathbb{R})$, then 
\begin{equation}
\int\partial f\,g\,dx=-\int f\,\partial g\,dx.\label{A2+}
\end{equation}

This equality is of primary importance in the theory of weak derivatives,
distributions, the calculus of variations, etc. Usually (\ref{A2+})
is deduced from the Leibniz rule 
\[
\partial(fg)=\partial f\,g+f\,\partial g.
\]

However, it is inconsistent to assume that the Leibniz rule holds
for every pair of ultrafunctions (see the discussion at the end of
Section \ref{RU}). Hence we need to investigate when an analogue
of (\ref{A2+}) holds for ultrafunctions which are not regular. Namely,
we need to translate the expression ``$f$ has compact support''
into the framework of ultrafunctions.

To this aim, it is useful to define the \textbf{interior part of the
grid} as follows:

\begin{equation}
\Gamma_{\mathfrak{int}}=\left\{ a\in\Gamma\ \big|\{a\}\cup\ supp(\partial\delta_{a_{\lambda}})\subset(-\omega_{\lambda},\omega_{\lambda})\right\} .\label{bd}
\end{equation}

\begin{lemma} If $a,b\in\Gamma_{\mathfrak{int}}$, then 
\[
\int^{\circ}D\delta_{a}(x)\delta_{b}(x)\,dx=-\int^{\circ}D\delta_{b}(x)\delta_{a}(x)\,dx.
\]
\end{lemma}

\textbf{Proof}. If $a,b\in\Gamma_{\mathfrak{int}}$, then $supp(\delta_{a_{\lambda}})\;\text{and}\;supp(\partial\delta_{a_{\lambda}})\subset(-\omega_{\lambda},\omega_{\lambda})$
and hence
\begin{eqnarray*}
\int^{\circ}D\delta_{a}(x)\delta_{b}(x)\,dx & = & D\delta_{a}(b)\\
 & = & \lim_{\lambda\uparrow\Lambda}\langle\partial\delta_{b_{\lambda}},\delta_{a_{\lambda}}\rangle\\
 & = & -\lim_{\lambda\uparrow\Lambda}\langle\partial\delta_{a_{\lambda}},\delta_{b_{\lambda}}\rangle\\
 & = & -D\delta_{b}(a)\\
 & = & -\int^{\circ}D\delta_{b}(x)\delta_{a}(x)\,dx.
\end{eqnarray*}

$\square$

\bigskip{}

An immediate consequence of the above lemma is the following result.

\begin{proposition} If $a\in\Gamma_{\mathfrak{int}}$, then 
\begin{equation}
D\delta_{a}(a)=0,\qquad D^{2}\delta_{a}(a)<0.\label{Dd}
\end{equation}
\end{proposition}

\textbf{Proof}. By (\ref{delta}) we have

\[
D\delta_{a}(a)=\int^{\circ}D\delta_{a}(x)\delta_{a}(x)\,dx.
\]

On the other hand, by the previous lemma,

\[
\int^{\circ}D\delta_{a}(x)\delta_{a}(x)\,dx=-\int^{\circ}\delta_{a}(x)D\delta_{a}(x)\,dx=-\int^{\circ}D\delta_{a}(x)\delta_{a}(x)\,dx=0.
\]

Moreover,

\[
D^{2}\delta_{a}(a)=\int^{\circ}D^{2}\delta_{a}(x)\delta_{a}(x)\,dx=-\int^{\circ}[D\delta_{a}(x)]^{2}\,dx<0.
\]

$\square$

\bigskip{}

We now set\footnote{ Here we introduce the notation $\mathfrak{span}$ meaning that every
element in $\mathfrak{span}(X)$ is a hyperfinite sum (see (\ref{sum}))
of elements of $X$ (and not a finite sum). }

\begin{equation}
V_{0}(\Gamma)=\mathfrak{span}\left\{ \delta_{a}\mid a\in\Gamma_{\mathfrak{int}}\right\} .\label{V0}
\end{equation}

\begin{theorem} \label{ilona} If $u,v\in V_{0}(\Gamma)$, then 
\[
\int^{\circ}Du(x)v(x)\,dx=-\int^{\circ}u(x)Dv(x)\,dx.
\]
\end{theorem}

\textbf{Proof}. If $a,b\in\Gamma_{\mathfrak{int}}$, by the previous
lemma and (\ref{dirac2}) we have

\begin{eqnarray*}
\int^{\circ}D\chi_{b}(x)\chi_{a}(x)\,dx & = & \frac{1}{\delta_{a}(a)\delta_{b}(b)}\int^{\circ}D\delta_{a}(x)\delta_{b}(x)\,dx\\
 & = & -\frac{1}{\delta_{a}(a)\delta_{b}(b)}\int^{\circ}D\delta_{b}(x)\delta_{a}(x)\,dx\\
 & = & -\int^{\circ}D\chi_{a}(x)\chi_{b}(x)\,dx.
\end{eqnarray*}

By (\ref{pu}) we have

\begin{eqnarray*}
\int^{\circ}Du(x)v(x)\,dx & = & \int^{\circ}D\left(\sum_{a\in\Gamma}u(a)\chi_{a}(x)\right)\left(\sum_{b\in\Gamma}v(b)\chi_{b}(x)\right)dx\\
 & = & \int^{\circ}\left(\sum_{a\in\Gamma}u(a)D\chi_{a}(x)\right)\left(\sum_{b\in\Gamma}v(b)\chi_{b}(x)\right)dx\\
 & = & \sum_{a,b\in\Gamma}u(a)v(b)\int^{\circ}D\chi_{a}(x)\chi_{b}(x)\,dx\\
 & = & -\sum_{a,b\in\Gamma}u(a)v(b)\int^{\circ}\chi_{a}(x)D\chi_{b}(x)\,dx\\
 & = & -\int^{\circ}u(x)Dv(x)\,dx.
\end{eqnarray*}

$\square$

\subsection{The antiderivative}

We now investigate the problem of the antiderivative of an ultrafunction.
Since $\ker D=span\{\mathbf{1}^{\circ}\}$, the operator $D$ is not
invertible and the equation 
\[
DF=f,\qquad f\in V(\Gamma)
\]
in general does not admit a solution. However we have the following
result.

\begin{theorem} \label{TT} Let 
\[
D_{0}:V_{0}(\Gamma)\rightarrow V_{0}(\Gamma)
\]
denote the generalized derivative restricted to $V_{0}(\Gamma)$,
namely 
\[
D_{0}u(x)=\sum_{a\in\Gamma_{\mathfrak{int}}}Du(a)\chi_{a}(x).
\]
Then $D_{0}$ is invertible. \end{theorem}

This theorem is based on the following lemma.

\begin{lemma} The set $\{\delta'_{a}\}_{a\in\Gamma_{\mathfrak{int}}}$
is a basis of $V_{0}(\Gamma)$, where for brevity we have set 
\[
\delta'_{a}=D\delta_{a}.
\]
\end{lemma}

\textbf{Proof}. Since $V_{0}(\Gamma)$ is a space of hyperfinite dimension
and $\{\delta'_{a}\}_{a\in\Gamma_{\mathfrak{int}}}$ is a set of $|\Gamma_{\mathfrak{int}}|$
vectors, it is sufficient to prove that these vectors are linearly
independent.

Assume that 
\[
\sum_{a\in\Gamma_{\mathfrak{int}}}k_{a}\delta'_{a}=0
\]
and we have to prove that $\forall a\in\Gamma_{\mathfrak{int}}$,
$k_{a}=0$. We have 
\[
\sum_{a\in\Gamma_{\mathfrak{int}}}k_{a}\delta'_{a}=D\left(\sum_{a\in\Gamma_{\mathfrak{int}}}k_{a}\delta_{a}\right)
\]
and hence, by Axiom \ref{AD}-(\ref{4}), 
\[
\sum_{a\in\Gamma_{\mathfrak{int}}}k_{a}\delta_{a}=c\mathbf{1}^{\circ}.
\]

Multiplying both sides by $\delta_{\omega}$ and integrating, we obtain
\[
\sum_{a\in\Gamma_{\mathfrak{int}}}k_{a}\int^{\circ}\delta_{a}\delta_{\omega}\,dx=c\int^{\circ}\mathbf{1}^{\circ}\delta_{\omega}\,dx=c.
\]

Since $\omega\notin\Gamma_{\mathfrak{int}}$, the left-hand side of
the above equation is $0$, hence $c=0$. Therefore 
\[
\sum_{a\in\Gamma_{\mathfrak{int}}}k_{a}\delta_{a}=0.
\]

Multiplying again both sides by $\delta_{b}$, with $b\in\Gamma_{\mathfrak{int}}$,
and integrating, we obtain 
\[
0=\sum_{a\in\Gamma_{\mathfrak{int}}}k_{a}\int^{\circ}\delta_{a}(x)\delta_{b}(x)\,dx=k_{b}\,\delta_{b}(b).
\]

Hence for every $b\in\Gamma_{\mathfrak{int}}$ we have $k_{b}=0$.

\[
\square
\]

\textbf{Proof of Th.~\ref{TT}}. For every $u\in V_{0}(\Gamma)$
we have to solve the equation 
\begin{equation}
D_{0}U=u.\label{57}
\end{equation}

By the lemma above, $u$ can be written as 
\[
u(x)=\sum_{a\in\Gamma_{\mathfrak{int}}}u_{a}\delta'_{a}(x),\qquad u_{a}\in\mathbb{E}.
\]

Then 
\[
U(x)=\sum_{a\in\Gamma_{\mathfrak{int}}}u_{a}\delta_{a}(x)
\]
solves equation (\ref{57}).

\[
\square
\]

\begin{corollary} \label{CO} For every $f\in V(\Gamma)$ and every
$c\in\mathbb{E}$ there exists a function $F\in V(\Gamma)$ such that
for all $x\in\Gamma_{\mathfrak{int}}$ 
\[
DF(x)=f(x),
\]
and 
\[
F(0)=c.
\]
\end{corollary}

\textbf{Proof}. Set 
\[
u(x)=\sum_{a\in\Gamma_{\mathfrak{int}}}f(a)\delta_{a}(x).
\]

If $U(x)$ is an antiderivative of $u(x)$, then the function 
\[
F(x)=U(x)+[c-U(0)]\,\mathbf{1}^{\circ}(x)
\]
satisfies the requirements.

\begin{remark} In conclusion, the structure introduced in this section
can be interpreted as a hyperfinite model of distribution theory.
The space $V(\Gamma)$ plays the role of a nonlinear algebra of generalized
functions, while the subspace $V_{0}(\Gamma)$ behaves as an analogue
of the space of test functions. In this framework the derivative and
the integral admit a purely algebraic formulation. \end{remark}

\section{Several-variable ultrafunctions}\label{SVU}

\subsection{Basic Definitions}

The space of ultrafunctions in $N$ variables is defined as follows:
\[
V(\Gamma^{N})=\bigotimes^{N}_{i=1}V(\Gamma)=\underbrace{V(\Gamma)\otimes\cdots\otimes V(\Gamma)}_{N\text{ times}}.
\]

The following facts follow immediately from the basic axioms.
\begin{itemize}
\par 
\item \textbf{(Approximation Property)} $u\in V(\Gamma^{N})$ if and only
if there exists a net 
\[
u_{\lambda}\in V_{\lambda}(\mathbb{R})\otimes\cdots\otimes V_{\lambda}(\mathbb{R})
\]
such that, for every point 
\[
x=\lim_{\lambda\uparrow\Lambda}x_{\lambda}\in\Gamma^{N},
\]
we have 
\[
u(x)=\lim_{\lambda\uparrow\Lambda}u_{\lambda}(x_{\lambda}).
\]
\item \textbf{($\boldsymbol{\chi}$-Property)} For every point 
\[
a=(a_{1},\ldots,a_{N})\in\Gamma^{N},\qquad x=(x_{1},\ldots,x_{N})\in\Gamma^{N},
\]
we set 
\[
\chi_{a}(x):=\chi_{a_{1}}(x_{1})\cdots\chi_{a_{N}}(x_{N})\in V(\Gamma^{N}),
\]
and a generic ultrafunction in $N$ variables can be written as 
\[
u(x)=\sum_{a\in\Gamma^{N}}u(a)\chi_{a}(x)=\sum_{a\in\Gamma^{N}}u(a_{1},\ldots,a_{N})\chi_{a_{1}}(x_{1})\cdots\chi_{a_{N}}(x_{N}).
\]
\end{itemize}
If $f$ is a real function defined on $\mathbb{R}^{N}$, then $f^{\circ}$
can be written as 
\[
f^{\circ}(x)=f^{\circ}(x_{1},\ldots,x_{N})=\sum_{a\in\Gamma^{N}}f^{\circ}(a_{1},\ldots,a_{N})\chi_{a_{1}}(x_{1})\cdots\chi_{a_{N}}(x_{N}).
\]

The main difference with the one-variable case is that the space 
\[
\bigotimes^{N}_{i=1}V(\mathbb{R})
\]
is not complete with respect to the seminorms 
\[
\|f\|_{V(Q_{n})}:=\|f_{|_{Q_{n}}}\|_{\sup},\qquad Q_{n}=[-n,n]^{N}.
\]
We denote by $V(\mathbb{R}^{N})$ the completion of 
\[
\bigotimes^{N}_{i=1}V(\mathbb{R}).
\]
If 
\[
f^{\circ}\in[V(\mathbb{R}^{N})]^{\circ}:=\{\,f^{\circ}\mid f\in V(\mathbb{R}^{N})\,\},
\]
the relation (\ref{ruben}) is no longer satisfied. However it can
be replaced by the following one: 
\[
\exists Q\in\mathcal{U},\;\forall\lambda\in Q,\;\forall x\in\mathbb{G}_{\lambda},\;f_{\lambda}(x)=f(x),
\]
where 
\[
\mathbb{G}_{\lambda}:=\{(x_{1},\ldots,x_{N})\in\mathbb{R}^{N}\mid\exists i,\;x_{i}\in\Gamma_{\lambda}\}.
\]

Nevertheless, this fact does not prevent the development of the theory.

\subsection{The integral and the derivative in several variables}

The pointwise integral with respect to the variable $x_{N}$ is defined
in the obvious way: 
\begin{equation}
\int^{\circ}u(x)\,dx_{N}=\sum_{x_{N}\in\Gamma}u(x_{1},\ldots,x_{N})\,d(x_{N})=U(x_{1},\ldots,x_{N-1})\in V(\Gamma^{N-1}).\label{pico}
\end{equation}

Similarly one defines the integral with respect to the other variables.
The integral of $u$ over the whole grid $\Gamma^{N}$ is given by

\begin{equation}
\int^{\circ}u(x)\,dx=\int^{\circ}\!\cdots\!\int^{\circ}u(x)\,dx_{1}\cdots dx_{N}\in\mathbb{E}.\label{pico2}
\end{equation}

\begin{theorem} If $u\in V(\mathbb{R}^{N})$, then 
\[
\int^{\circ}u(x)\,dx=\lim_{\lambda\uparrow\Lambda}\int_{Q_{\lambda}}u_{\lambda}(x)\,dx,
\]
where 
\begin{equation}
u_{\lambda}\in\bigotimes^{N}_{i=1}V_{\lambda}(\mathbb{R}),\qquad Q_{\lambda}=[-\omega_{\lambda},\omega_{\lambda}]^{N}.\label{QL}
\end{equation}
\end{theorem}

\textbf{Proof}. Fix $(a_{1},\ldots,a_{N-1})\in\Gamma^{N-1}$. The
ultrafunction

\[
x_{N}\mapsto u(a_{1},\ldots,a_{N-1},x_{N})=\sum_{x_{N}\in\Gamma}u(a_{1},\ldots,a_{N-1},a_{N})\chi_{a_{N}}(x_{N})
\]

belongs to $V(\Gamma)$. Hence, putting $a_{\lambda}=(a_{1,\lambda},\ldots,a_{N-1,\lambda})$,
by Axiom~\ref{IA}

\[
U(a_{1},\ldots,a_{N-1})=\lim_{\lambda\uparrow\Lambda}\int^{\omega_{\lambda}}_{-\omega_{\lambda}}u_{\lambda}(a_{\lambda},x_{N})\,dx_{N}.
\]

Renaming the variables we obtain

\[
U_{\lambda}(x_{1},\ldots,x_{N-1})=\int^{\omega_{\lambda}}_{-\omega_{\lambda}}u_{\lambda}(x_{\lambda},x_{N})\,dx_{N}.
\]

Iterating this procedure we obtain

\begin{eqnarray*}
\int^{\circ}\!\cdots\!\int^{\circ}u(x)\,dx_{N-1}dx_{N} & = & \int^{\circ}U(x_{1},\ldots,x_{N-1})\,dx_{N-1}\\
 & = & \lim_{\lambda\uparrow\Lambda}\int^{\omega_{\lambda}}_{-\omega_{\lambda}}U_{\lambda}(x_{\lambda})\,dx_{N-1}\\
 & = & \lim_{\lambda\uparrow\Lambda}\int^{\omega_{\lambda}}_{-\omega_{\lambda}}\!\!\cdots\!\int^{\omega_{\lambda}}_{-\omega_{\lambda}}u_{\lambda}(x)\,dx_{N-1}dx_{N}.
\end{eqnarray*}

Iterating the same argument for all variables we obtain

\begin{eqnarray*}
\int^{\circ}u(x)\,dx & = & \int^{\circ}\!\cdots\!\int^{\circ}u(x)\,dx_{1}\cdots dx_{N}\\
 & = & \lim_{\lambda\uparrow\Lambda}\int^{\omega_{\lambda}}_{-\omega_{\lambda}}\!\!\cdots\!\int^{\omega_{\lambda}}_{-\omega_{\lambda}}u_{\lambda}(x)\,dx_{1}\cdots dx_{N}\\
 & = & \lim_{\lambda\uparrow\Lambda}\int_{Q_{\lambda}}u_{\lambda}(x)\,dx.
\end{eqnarray*}

$\square$

\begin{corollary} If $f\in V_{c}(\mathbb{R}^{N})$, then 
\[
\int^{\circ}f^{\circ}(x)\,dx=\int f(x)\,dx.
\]
\end{corollary}

\textbf{Proof}. If $\lambda$ is sufficiently large, then $\,supp(f)\subset Q_{\lambda}$.

$\square$

\begin{remark} This result is not trivial. Indeed, since (\ref{ruben})
is not satisfied in $\mathbb{R}^{N}$, for some $x\in\mathbb{R}^{N}$
, we have that $f_{\lambda}(x)\neq f(x)$. \end{remark}

The generalized partial derivative is defined as follows:

\[
D_{i}u(a)=\lim_{\lambda\uparrow\Lambda}\langle\partial_{i}u_{\lambda},\delta_{a_{\lambda}}\rangle,
\]
where 
\[
\delta_{a_{\lambda}}(x)=\delta_{{a_{\lambda,1}}}(x_{1})\cdots\delta{a_{\lambda,N}}.
\]
By the linearity of $D_{i}$ we also have 
\begin{eqnarray*}
D_{i}u(x) & = & D_{i}\left(\sum_{a\in\Gamma^{N}}u(a)\chi_{a_{1}}(x_{1})\cdots\chi_{a_{N}}(x_{N})\right)\\
 & = & \sum_{a\in\Gamma^{N}}u(a)D_{i}\left[\chi_{a_{1}}(x_{1})\cdots\chi_{a_{N}}(x_{N})\right]\\
 & = & \sum_{a\in\Gamma^{N}}u(a)\chi_{a_{1}}(x_{1})\cdots D_{i}\chi_{a_{i}}(x_{i})\cdots\chi_{a_{N}}(x_{N}).
\end{eqnarray*}

\begin{theorem} \label{dery} If $f\in C^{1,1}_{loc}(\mathbb{R}^{N})$,
then 
\[
D_{i}f^{\circ}(x)=(\partial_{i}f)^{\circ}(x).
\]
\end{theorem}

\textbf{Proof}. For fixed $(a_{1,\lambda},\ldots,a_{N,\lambda})$
the function 
\[
x\mapsto f(a_{1,\lambda},\ldots,x,\ldots,a_{N,\lambda})
\]
belongs to $C^{1,1}_{loc}(\mathbb{R})$. Hence, by Th.~\ref{TU}-(\ref{U3}),

\[
(\partial_{i}f)^{\circ}(a_{1},\ldots,x_{i},\ldots,a_{N})=D_{i}f^{\circ}(a_{1},\ldots,x_{i},\ldots,a_{N}).
\]

Using the expansion in the $\chi$-basis we obtain

\begin{eqnarray*}
(\partial_{i}f)^{\circ}(x) & = & \sum_{a\in\Gamma^{N}}f^{\circ}(a)\chi_{a_{1}}(x_{1})\cdots D_{i}\chi_{a_{i}}(x_{i})\cdots\chi_{a_{N}}(x_{N})\\
 & = & D_{i}f^{\circ}(x).
\end{eqnarray*}

$\square$

\begin{corollary} \label{CC} If $f\in C^{1,1}_{loc}(\mathbb{R}^{N})$,
then

\[
\sum_{a\in\Gamma^{N}}f^{\circ}(a)\,D_{i}\chi_{a}(x)=\sum_{a\in\Gamma^{N}}(\partial_{i}f)^{\circ}(a)\,\chi_{a}(x).
\]
\end{corollary}

\subsection{Ultrafunctions and measures}\label{DM}

By the extension axiom, we have seen that to every real function $f$
we can associate a ultrafunction $f^{\circ};$ next we will see how
we can associate an ultrafunction to every Radon measure $\mu\in\mathfrak{M}(\mathbb{R}^{N})$.

\begin{definition} \label{dual}Given a measure $\mu\in\mathfrak{M}(\mathbb{R}^{N}),$
$\forall v\in V\left(\Gamma^{N}\right),$\ we set 
\begin{equation}
\int^{\circ}\tilde{\mu}\left(x\right)v(x)dx=\lim_{\lambda\uparrow\Lambda}\ \left\langle \mu,v_{\lambda}\right\rangle _{\lambda}\label{98}
\end{equation}
\end{definition}

By this definition, taking $v=\delta_{a}$, we have that
\begin{equation}
\tilde{\mu}\left(a\right)=\int^{\circ}\tilde{\mu}(x)\delta_{a}(x)dx=\lim_{\lambda\uparrow\Lambda}\ \left\langle \mu,\delta_{a_{\lambda}}\right\rangle _{\lambda}\label{97}
\end{equation}

\textbf{Example 1:} If $\mathbf{\delta}_{a}$ is the Dirac measure,
then 
\[
\mathbf{\tilde{\delta}}_{a}=\delta_{a}
\]
where $\delta_{a}\in V(\Gamma^{N})$ is the Dirac ultrafunction defined
by (\ref{dirac2}).

\bigskip{}
\textbf{Example 2:} If $f\in L^{1}_{loc}\left(\mathbb{R}^{N}\right)$,
$f$ can be identified with the measure $\mu$ whose density is $f;$
then, we have that
\begin{equation}
\int^{\circ}\tilde{f}\left(x\right)v(x)dx=\lim_{\lambda\uparrow\Lambda}\int f\left(x\right)v_{\lambda}(x)dx\label{97+}
\end{equation}

\begin{remark} Exanple 2 showws that it is necessary to introduce
the symbol $(\tilde{\cdot})$ to distinguish $\tilde{f}$ from $f^{\circ}$.
\end{remark}

Let us analyze the difference between the operators $(^{\circ})\ $and
$\left(\tilde{\cdot}\right)$ when they are applied to $f\in\mathfrak{F}\left(\mathbb{E}^{N}\right)\cap\mathfrak{M}(\mathbb{R}^{N})=\mathcal{L}^{1}_{loc}\left(\mathbb{R}^{N}\right):$
by the definition, we have that 
\[
f^{\circ}(a)=\lim_{\lambda\uparrow\Lambda}f(a_{\lambda}),
\]
\[
\tilde{f}(a)=\lim_{\lambda\uparrow\Lambda}\int f(x)\delta_{a_{\lambda}}(x)dx;
\]
then, if $f\in\mathcal{L}^{1}_{c}\left(\mathbb{R}^{N}\right),$
\begin{itemize}
\item $\forall x\in\mathbb{R}^{N},$ $f^{\circ}(x)=f(x),$ and $\int^{\circ}f^{\circ}dx\sim\int f\ dx;$
\item $\forall a.e.\ x\in\mathbb{R}^{N},\ \ \tilde{f}(x)\sim f(x),$ and$\ \int^{\circ}\tilde{f}\ dx=\int f\ dx,$
in fact, 
\[
\int^{\circ}\tilde{f}\ dx=\lim_{\lambda\uparrow\Lambda}\int_{Q_{\lambda}}f\ dx=\lim_{\lambda\uparrow\Lambda}\int f\ dx=\int f\ dx.
\]
\end{itemize}
\textbf{Example 3: }If $f\in C^{0,1}\left(\mathbb{R}^{N}\right)=V\left(\mathbb{R}^{N}\right)\cap\mathfrak{M}(\mathbb{R}),$
then
\begin{equation}
\tilde{f}(x)=f^{\circ}(x);\label{bu+}
\end{equation}
in fact, $f^{\circ}\in V^{0}(\Gamma)$ and by Th. \ref{TU}-(\ref{U00}),
$\forall v\in V^{0}(\Gamma)$
\[
\int^{\circ}\tilde{f}\left(x\right)v(x)dx=\lim_{\lambda\uparrow\Lambda}\int f\left(x\right)v_{\lambda}(x)dx=\int^{\circ}f^{\circ}\left(x\right)v(x)dx.
\]

\textbf{Example 4:} If $M\subset\mathbb{R}^{N}$ is a manifold whose
measure is denoted by $\mu_{M},$ $\forall u\in V(\Gamma),$ it makes
sense to define
\begin{equation}
\int^{\circ}_{\widetilde{M}}u(x)dx:=\int^{\circ}\tilde{\mu}_{M}(x)u(x)dx.\label{bi+}
\end{equation}

\textbf{Example 5:} In particular, if $\Omega\in\mathbb{R}^{N}$ is
a bounded measurable set, then 
\begin{equation}
\int^{\circ}_{\widetilde{\Omega}}u(x)dx=\int^{\circ}\tilde{\chi}_{\Omega}(x)u(x)dx=\lim_{\lambda\uparrow\Lambda}\int_{\Omega}u_{\lambda}(x)dx.\label{bi++}
\end{equation}

\subsection{The vicinity of a set}

Given an open set $\Omega\subset\mathbb{R}^{N}$ and $f\in C^{1}\left(\Omega\right)$
(namely, it is the restriction to $\Omega$ of a function $f\in C^{1}\mathfrak{(\mathbb{R})}$),
then the value of $\nabla f(x_{0})$ in a point $x_{0}\in\Omega$
depends only on the values which $f$ takes in $\Omega$ since 
\[
\nabla f(x_{0})=\lim_{\substack{x\in\Omega,\\
x\rightarrow x_{0}
}
}\ \nabla f(x_{0})
\]
If $\nabla f$ is not continuous, $\nabla f$ is not defined, but
$Df{^{\circ}}$ makes sense; however $Df{^{\circ}}(x_{0})$ in a point
$x_{0}\sim\partial\Omega{^{\circ}}$ depends on the values which $f$
takes in suitable points $y\sim x_{0}$ even if $y\notin\Omega{^{\circ}}.$
Hence, if $u\notin V^{1}(\Gamma)$, it might happen that 
\[
\mathfrak{supp}(Du)\nsubseteq\mathfrak{supp}(u).
\]
Thus we are lead to the following definition:

\begin{definition} \label{47}For any internal set $\Theta\subseteq\Gamma^{N}$
(see (\ref{inter})), the \textbf{vicinity} of $\Theta$ is defined
as follows: 
\[
\mathfrak{vic}\left(\Theta\right):=\left[\bigcup\limits_{a\in\Theta}\{a\}\cup\mathfrak{supp}(D\delta_{a})\right];
\]
the \textbf{restricted interior }is defined as follows 
\[
\mathfrak{int}(\Theta):=\left\{ x\in\Theta\ |\;\{x\}\cup\mathfrak{supp}(D\delta_{x})\subset\Theta\right\} ;
\]
the \textbf{enlarged boundary} $\mathfrak{bd}\left(\Theta\right)$
is defined as follows 
\[
\mathfrak{bd}\left(\Theta\right):=\mathfrak{vic}\left(\Theta\right)\backslash\mathfrak{int}(\Theta)
\]
\end{definition}

\textbf{Example 1}: If $\Omega\subset\mathbb{R}^{N}$ is an open set
and $\Theta=\Omega^{\circ},$ then 
\[
\mathfrak{vic}\left(\Omega^{\circ}\right)\supset\overline{\Omega}^{\circ};\ \ \mathfrak{bd}\left(\Omega^{\circ}\right)=\mathfrak{vic}\left(\partial\Omega^{\circ}\right)\supset\partial\Omega^{\circ};\ \ \mathfrak{int}\left(\Omega^{\circ}\right)\subset\Omega^{\circ}.
\]

\textbf{Example 2:} If $\Theta=\left\{ x_{0}\right\} ,$ then, with
some abuse of notation, we will write $\mathfrak{vic}\left(x_{0}\right)$
instead of $\mathfrak{vic}\left(\left\{ x_{0}\right\} \right)$ and
we have that 
\[
\left\{ x_{0}\right\} \subsetneqq\mathfrak{vic}\left(x_{0}\right);\ \mathfrak{bd}\left(\left\{ x_{0}\right\} \right)=\mathfrak{vic}\left(x_{0}\right);\ \mathfrak{int}\left(\left\{ x_{0}\right\} \right)=\varnothing.
\]
Notice that, by Axiom \ref{AD}-(\ref{loc}), $\mathfrak{vic}\left(x_{0}\right)\subset\mathfrak{mon}(x_{0}).$

\textbf{Example 3:} If $\Theta=\Gamma^{N},$ then $\mathfrak{vic}\left(\Gamma^{N}\right)=\Gamma^{N}$;
the set
\begin{equation}
\mathfrak{bd}\left(\Gamma^{N}\right)=\left\{ (x_{1},..,x_{N})\in\Gamma^{N}\ |\ \exists i,\ x_{i}\in\mathfrak{bd}\left(\Gamma\right)\right\} \label{bi}
\end{equation}
will be \textbf{called enlarged boundary of infinity} and the restricted
interior of $\Gamma^{N}$ is given by
\[
\Gamma_{\mathfrak{int}}=\left\{ x\in\Gamma^{N}\ |\ \exists i,\ x_{i}\in\mathfrak{bd}\left(\Gamma\right)\right\} .
\]
If\textbf{\ }$N=1,$ it is not difficult to check that $\Gamma_{\mathfrak{int}}$
agrees with (\ref{bd}).

By the definition of vicinity, we have that $\forall u\in V(\Gamma),$

\[
\mathfrak{supp}(Du)\subseteq\mathfrak{vic}\left(\mathfrak{supp}(u)\right)
\]
and hence
\[
\mathfrak{supp}(D^{m}u)\subseteq\mathfrak{vic(...(vic}\left(\mathfrak{supp}(u)\right):=\mathfrak{vic}^{m}\left(\mathfrak{supp}(u)\right);
\]
in any case $\forall m\in\mathbb{N}$, 
\[
\mathfrak{supp}(D^{m}u)\subseteq\left\{ x\in\Gamma^{N}\ |\ x\sim\mathfrak{supp}(u)\right\} .
\]

\subsection{Some examples of generalized derivative}\label{sgd}

If $f\in C^{1,1}_{loc}(\mathbb{R}^{N})$, by Th.~\ref{TU}-(\ref{U3})
we have 
\begin{equation}
D_{i}f^{\circ}=(\partial_{i}f)^{\circ}.\label{der}
\end{equation}
Notice that (\ref{der}) may fail if $f\in C^{1}(\mathbb{R}^{N})\setminus C^{1,1}_{loc}(\mathbb{R}^{N})$.
For example take 
\[
f(x):=\int^{x}_{0}\left(t\sin\frac{1}{t^{2}}\right)dt.
\]
In this case, since $x\sin\frac{1}{x^{2}}\notin BV$, we have $\partial f\notin V(\mathbb{R})$;
hence for some $x\in\mathfrak{mon}(0)\cap\Gamma$ it may happen that
\[
Df(x)\neq x\sin\frac{1}{x^{2}}.
\]
More generally, if $f\in V(\mathbb{R}^{N})\setminus C^{1,1}_{loc}(\mathbb{R}^{N})$,
then $\partial_{i}f\in\mathfrak{M}(\mathbb{R}^{N})$ and by (\ref{lillina})
and (\ref{97}) we obtain 
\begin{equation}
D_{i}f(a)=\lim_{\lambda\uparrow\Lambda}\langle\partial_{i}f,\delta_{a_{\lambda}}\rangle=\widetilde{\partial_{i}f}(a).\label{pipi}
\end{equation}

Let us now consider some examples.

\textbf{Example 1.} If 
\[
H(x)=\frac{1}{2}\,[\mathrm{sign}(x)+1]
\]
is the Heaviside function, since $H\in V(\mathbb{R})$ we obtain 
\[
DH(x)=\delta_{0}(x).
\]

\textbf{Example 2.} Let $f(x):=\max(1-x^{2},0)$. Since $f\in C^{0,1}_{c}(\mathbb{R})$
we have

\begin{equation}
Df^{\circ}=(\partial f)^{\circ}=\begin{cases}
-2x & \text{if }x\in(-1,1),\\
1 & \text{if }x=-1,\\
-1 & \text{if }x=1,\\
0 & \text{if }x\notin[-1,1].
\end{cases}\label{moro}
\end{equation}
Since $\partial f\in V(\mathbb{R})$, it follows that
\[
D^{2}f=\widetilde{\partial(\partial f)}=\begin{cases}
-2 & \text{if }x\in(-1,1),\\[6pt]
-2\delta_{1}(1)=\frac{-2}{d(1)} & \text{if }x=1,\\[6pt]
2\delta_{-1}(-1)=\frac{2}{d(-1)} & \text{if }x=-1,\\[6pt]
0 & \text{if }x\notin[-1,1].
\end{cases}
\]

\textbf{Example 3.} If 
\[
\theta_{(a,b)}:=\chi^{\text{\textsc{epl}}}_{[a,b]}=\chi_{(a,b)}+\frac{1}{2}\chi_{\{a,b\}},
\]
then

\begin{eqnarray*}
\int^{\circ}D\theta^{\circ}_{(a,b)}v\,dx & = & \lim_{\lambda\uparrow\Lambda}\langle\partial\theta_{(a,b)},v_{\lambda}\rangle\\
 & = & \lim_{\lambda\uparrow\Lambda}\langle\delta_{a}-\delta_{b},v_{\lambda}\rangle\\
 & = & \lim_{\lambda\uparrow\Lambda}\big(v_{\lambda}(a)-v_{\lambda}(b)\big)\\
 & = & v(a)-v(b)\\
 & = & \int^{\circ}(\delta_{a}-\delta_{b})v\,dx.
\end{eqnarray*}
Hence, 
\begin{equation}
D\theta^{\circ}_{(a,b)}=\delta_{a}-\delta_{b}.\label{gru}
\end{equation}

\textbf{Example 4.} Taking $\chi_{[a,b]}$, since $\chi_{[a,b]}\notin V(\mathbb{R})$
the situation is slightly more complicated. Indeed,

\[
\chi_{[a,b]}=\theta_{(a,b)}+\frac{1}{2}\chi_{a}+\frac{1}{2}\chi_{b}.
\]
Using (\ref{dirac2}) we obtain

\begin{eqnarray*}
D\chi^{\circ}_{[a,b]} & = & D\theta^{\circ}_{(a,b)}+\frac{1}{2}D\chi_{a}+\frac{1}{2}D\chi_{b}\\
 & = & \delta_{a}-\delta_{b}+\frac{d(a)}{2}\delta'_{a}+\frac{d(b)}{2}\delta'_{b}.
\end{eqnarray*}

\subsection{The Poincaré inequality}

Let us see a variant of the Poincaré inequality, which will be useful
in the applications to PDEs.

\begin{theorem} Let $\Omega$ be a connected measurable set and let
$x_{0}\in\Omega^{\circ}$; then there exists a constant $C(\Omega,x_{0})$
such that 
\begin{equation}
\left[|u(x_{0})|^{2}+\int^{\circ}_{\Omega}|Du(x)|^{2}dx\right]\geq C(\Omega,x_{0})\int^{\circ}_{\Omega}|u(x)|^{2}dx\label{PI+}
\end{equation}
\end{theorem}

\textbf{Proof}: We set 
\[
V(\Omega^{\circ}):=\left\{ u_{|_{\Omega^{\circ}}}\ |\ u\in V(\Gamma^{N})\right\} 
\]
and 
\[
W(\Omega^{\circ}):=\left\{ \left[u(x)-u(x_{0})\right]_{|_{\Omega^{\circ}}}\ |\ u\in V(\Gamma^{N})\right\} .
\]

If $w\in W(\Omega^{\circ})$ and $Dw=0$, by (\ref{A1}) $w$ is a
constant function. Then, since $w(x_{0})=0$, we have that $w=0$.
Hence the quadratic form 
\[
u\mapsto|u(x_{0})|^{2}+\int^{\circ}_{\Omega}|Du(x)|^{2}dx
\]
is positive definite; then, since $\dim\left[W(\Omega^{\circ})\right]$
is hyperfinite, 
\[
C(\Omega,x_{0}):=\min_{u\neq0}\frac{|u(x_{0})|^{2}+\int^{\circ}_{\Omega}|Du(x)|^{2}dx}{\int^{\circ}_{\Omega}|u(x)|^{2}dx}>0
\]
and we have that $\forall u\in V(\Omega^{\circ})$ 
\[
C(\Omega,x_{0})\int^{\circ}_{\Omega}|u(x)|^{2}dx\le\left[|u(x_{0})|^{2}+\int^{\circ}_{\Omega}|Du(x)|^{2}dx\right].
\]

$\square$

\medskip{}

Notice that $C(\Omega,x_{0})$, in general, is an infinitesimal number.

\begin{remark} If $\Omega$ is not connected, the Poincaré inequality
can be formulated in the usual way 
\begin{equation}
C(\Omega)\int^{\circ}_{\Omega}|u(x)|^{2}dx\leq\left[\int^{\circ}_{\Omega}|Du(x)|^{2}dx+|\bar{u}|^{2}\right]\label{PI}
\end{equation}
where 
\[
\bar{u}=\frac{\int^{\circ}_{\Omega}u(x)\ dx}{\int^{\circ}_{\Omega}dx}
\]
and 
\[
C(\Omega):=\min_{u\neq0}\frac{|\bar{u}|^{2}+\int^{\circ}_{\Omega}|Du(x)|^{2}dx}{\int^{\circ}_{\Omega}|u(x)|^{2}dx}.
\]
\end{remark}

\begin{remark} The generalized Poincaré inequality holds even when
the measure of $\Omega$ is infinite; in this case also $C(\Omega)$
is an infinitesimal number. \end{remark} 

\subsection{The pointwise integral restricted to a set}

If $\Theta\subset\Gamma$ is an internal set, taking account of (\ref{int}),
it is natural to define 
\[
\int^{\circ}_{\Theta}u(x)\,dx:=\sum_{a\in\Theta}u(a)d(a).
\]
However, if $\Omega\subset\mathbb{R}^{N}$ is an open set and $f$
is a continuous function, in general 
\[
\int^{\circ}_{\Omega^{\circ}}f^{\circ}(x)\,dx\neq\int_{\Omega}f(x)\,dx
\]
since the values of $f$ on $\partial\Omega$ are relevant. Hence
we need a good definition of 
\[
\int^{\circ}_{\Omega}u(x)\,dx,\;\;\Omega\subset \mathbb{R}^N
\]
such that in the smooth case it agrees with $\int_{\Omega}u(x)\,dx$.
There are (at least) two reasonable definitions: 
\[
\int^{\circ}\tilde{\theta}_{\Omega}(x)u(x)\,dx\qquad\text{and}\qquad\int^{\circ}\theta^{\circ}_{\Omega}(x)u(x)\,dx,
\]
where $\theta_{\Omega}=\chi^{\textsc{epl}}_{\Omega}$.  In fact, if $f\in C^{0,1}$, since $\theta_{\Omega}f\in V(\mathbb{R})$ we have 
\[
\int^{\circ}\theta^{\circ}_{\Omega}(x)f^{\circ}(x)\,dx=\int\theta_{\Omega}(x)f(x)\,dx=\int_{\Omega}f(x)\,dx.
\]

On the other hand, $\tilde\theta_{\Omega}=\tilde{\chi}_{\Omega}$ and by (\ref{bi++}) 
\[
\int^{\circ}\tilde{\theta}_{\Omega}(x)f^{\circ}(x)\,dx=\int_{\Omega}f(x)\,dx.
\]

Even if, by (\ref{bi++}), the second definition seems more natural,
the first one is more convenient for applications.

\begin{definition}\label{pippo} If $\Omega\subset\mathbb{R}^{N}$
is a measurable set, for every $u\in V(\Gamma)$ we define 
\[
\int^{\circ}_{\Omega}u(x)\,dx:=\int^{\circ}\theta^{\circ}_{\Omega}(x)u(x)\,dx.
\]
\end{definition}

In conclusion, given $u$ and $\Omega$, we can define three different
quantities 
\[
\int^{\circ}_{\Omega^{\circ}}u(x)\,dx,\qquad\int^{\circ}_{\Omega}u(x)\,dx,\qquad\int^{\circ}_{\tilde\Omega}u(x)\,dx=\int^{\circ}\tilde\theta_{\Omega}(x)u(x)\,dx,
\]
which in general differ by an infinitesimal quantity due to the contribution of the boundary.

In particular, if $\Omega\subset\mathbb{R}^{N}$ is a bounded open set, we have 
\[
\int^{\circ}_{\Omega}u(x)\,dx=\int^{\circ}_{\Omega^{\circ}}u(x)\,dx+\int^{\circ}_{\partial\Omega^{\circ}}\theta^{\circ}_{\Omega}(x)u(x)\,dx=\sum_{a\in\Omega^{\circ}}u(a)d(a)+\sum_{a\in\partial\Omega^{\circ}}\theta^{\circ}_{\Omega}(a)u(a)d(a).
\]
and, if $\partial\Omega$ is smooth, 
\[
\theta_{\Omega}(x)=\begin{cases}
1 & \text{if }x\in\Omega,\\
0 & \text{if }x\notin\Omega,\\
\frac{1}{2} & \text{if }x\in\partial\Omega.
\end{cases}
\]
Hence 
\begin{eqnarray*}
\oint_{\Omega}u(x)\,dx & = & \int^{\circ}_{\Omega^{\circ}}u(x)\,dx+\frac{1}{2}\int^{\circ}_{\partial\Omega^{\circ}}u(x)\,dx\\
 & = & \sum_{a\in\Omega^{\circ}}u(a)d(a)+\frac{1}{2}\sum_{a\in\partial\Omega^{\circ}}u(a)d(a).
\end{eqnarray*}

The choice given by Definition \ref{pippo} allows one to generalize
the fundamental theorem of calculus in the framework of ultrafunctions
(and also the Gauss divergence theorem; see Section \ref{GDT}).

\begin{theorem}[Fundamental Theorem of Calculus]\label{TFC} If $a,b\in\mathbb{R}$,
then for every $v\in V(\Gamma)$ 
\[
\int^{\circ}_{(a,b)}Dv\,dx=v(b)-v(a).
\]
\end{theorem}
\begin{proof}
For every $v\in V(\Gamma)$, since $\theta_{(a,b)}\in V(\mathbb{R})$,
by (\ref{gru}) we have 
\[
\int^{\circ}_{(a,b)}Dv(x)\,dx=\int^{\circ}\theta^{\circ}_{(a,b)}Dv\,dx=-\int^{\circ}D\theta^{\circ}_{(a,b)}v\,dx=-\langle\partial\theta_{(a,b)},v_{\Lambda}\rangle.
\]

Since $\partial\theta_{(a,b)}=\delta_a-\delta_b$, we get
\[
\int_{(a,b)}^{\circ} Dv(x)\,dx
=-\langle \delta_a-\delta_b, v_\Lambda\rangle
=
v(b)-v(a).
\]
$\square$
\end{proof}

Theorem \ref{TFC} is not trivial since it holds even
when $v(x)$ is a very wild function, for example when $v=\sqrt{\delta_{a}}$
with $a\in\Gamma$, or $v=g^{\circ}$ when $g$ is a non-measurable
function.

Theorem \ref{TFC} allows one to write the usual explicit formula for the
unique antiderivative in $V_{0}(\Gamma)$ of a generic ultrafunction
$v$: 
\[
F:=D^{-1}_{0}v.
\]

Indeed, if $DF=v$, integrating from $0$ to $x\in\mathbb{R}$ we
get 
\[
\int^{\circ}_{(0,x)}DF(t)\,dt=\int^{\circ}_{(0,x)}v(t)\,dt
\]
and hence 
\[
F(x)=F(0)+\int^{\circ}_{(0,x)}v(t)\,dt.
\]
$\square$

\subsection{The Gauss' divergence theorem}\label{GDT}

In this section we want to generalize the Gauss' divergence theorem
in the framework of the ultrafunction (see also \cite{gauss}); in
particular it is interesting to analyze the case when $\partial\Omega$
is not smooth.

First, we will examine the smooth case. If $\Omega\subset\mathbb{R}^{N}$
is a bounded set with smooth boundary and $\mathbf{\phi}$ is a smooth
vector field, we have that
\[
\int_{\Omega}\nabla\cdot\mathbf{\phi}\ dx=\int\mathbf{\phi}\cdot\mathbf{n}_{\Omega}\ dS^{N-1}
\]
where $S^{N-1}$ denotes the $(N-1)$-dimensional measure
of $\partial\Omega$ and $\mathbf{n}_{\Omega}(x)$ is the exterior
normal derivative. We have the following result:

\begin{theorem} \label{ruby}Let $\Omega\subset\mathbb{R}^{N}$ be
a bounded open set with a $C^{1}$-boundary and let $\theta_{\Omega}=\chi^{\text{\textsc{epl}}}_{\Omega};$
then
\[
\int^{\circ}_{\Omega}D\cdot\mathbb{\phi}\ dx=\int^{\circ}_{\Omega}\mathbb{\phi}\cdot\mathbf{n}^{\circ}_{\Omega}\ |D\theta^{\circ}_{\Omega}|dx
\]
where  $\mathbf{n}_{\Omega}(x)$  is a $C^1$-extension of the normal derivative  to all $\mathbb{R}^{N}$  and,
$\forall x\in\mathfrak{\partial}\Omega$,  
\[
\mathbf{n}_{\Omega}^\circ(x)=-\frac{D\theta^{\circ}_{\Omega}(x)}{|D\theta^{\circ}_{\Omega}(x)|}.
\]
\end{theorem}

\textbf{Proof}: Since $\partial\Omega$ is $C^{1}$,\ we can extend
$\mathbf{n}_{\Omega}(x)$ to all $\mathbb{R}^{N}.$ Then, the Gauss divergence theorem can be wrutten as follows:
\[
\left\langle \nabla\cdot\phi,\ \theta_{\Omega}\right\rangle =\left\langle S^{N-1},\ \mathbf{n}_{\Omega}\cdot\phi\right\rangle 
\]
where $\nabla\cdot\phi$ is a vector valued measure.  But
\[
\left\langle \nabla\cdot\phi,\ \theta_{\Omega}\right\rangle =\sum_{i}\left\langle \partial_{i}\phi_{i},\ \theta_{\Omega}\right\rangle =-\sum_{i}\left\langle \partial_{i}\theta_{\Omega},\ \phi_{i}\right\rangle =-\left\langle \nabla\theta_{\Omega},\ \phi\right\rangle 
\]
and hence, $\forall\phi_{\lambda}\in\left[V^{N}_{\lambda}\left(\mathbb{R}\right)\right]^{N}$
\[
-\left\langle \nabla\theta_{\Omega},\ \phi_{\lambda}\right\rangle =\left\langle S^{N-1},\ \mathbf{n}_{\Omega}\cdot\phi_{\lambda}\right\rangle 
\]
Then, taking the $\Lambda$-limit,
\[
\int^{\circ}D\theta^{\circ}_{\Omega}\ \phi\ dx=-\int^{\circ}\mathbf{n}^{\circ}_{\Omega}\cdot\phi\;\tilde{S}^{N-1} dx
\]
Hence, since this equality holds for every $\phi\in \left[ V^{N}\left(
\Gamma \right) \right] ^{N},$  
\[
D\theta _{\Omega }^{\circ }=-\tilde{S}^{N-1}\mathbf{n}%
_{\Omega };
\]
since  $D\theta _{\Omega }^{\circ }$  is parallel to  $\mathbf{n}_{\Omega }$
\[
|D\theta _{\Omega }^{\circ }(x)|=\tilde{S}^{N-1}\ \ \ \text{and }\ \ \mathbf{n}_{\Omega }=-\frac{D\theta _{\Omega }^{\circ
}}{|D\theta _{\Omega }^{\circ }(x)|}.
\]
$\square$
\medskip{}

Theorem \ref{ruby} holds if $\partial \Omega $ is of class $C^{1}$, but
it suggests the "right" generalization of the notion of normal
derivative.  Given a measurable set $E\subset \mathbb{R}^{N},$ we define 
\begin{equation}
\mathbf{n}_{E}^{\circ }(x)=\left\{ 
\begin{array}{cc}
-\frac{D\theta _{E}^{\circ }}{|D\theta _{E}^{\circ }(x)|} & if\ \ \left\vert
D\theta _{E}^{\circ }(x)\right\vert \neq 0\  \\ 
&  \\ 
0 & if\ \ \left\vert D\theta _{E}^{\circ }(x)\right\vert =0%
\end{array}%
\ \right.  \label{127}
\end{equation}%
It is surprizing that $\mathbf{n}_{E}^{\circ }(x)$ makes sense even if $E$
consists of a single point $x_{0}.$ Clearly in this case,  $\mathfrak{supp}%
\left( \mathbf{n}_{\left\{ x_{0}\right\} }^{\circ }\right) \subset \mathfrak{%
vic}\left( \{x_{0}\}\right) \subset \mathfrak{mon}\left( x_{0}\right) .$

\begin{theorem}
\label{B}(\textbf{Generalized Gauss' divergence theorem}) Let $\Phi%
:\Gamma \rightarrow \left[ V(\Gamma )\right] ^{N}$ be a (ultrafunctions)
vector field and let $E\subseteq \Gamma $ be an internal set; then%
\begin{equation*}
\int_{E}^{\circ }D\cdot {\Phi }\ dx=\int_{\mathfrak{bd}\left(
E\right) }^{\circ }{\Phi }\cdot \mathbf{n}_{E}^{\circ }\ |D\theta
_{E}^{\circ }|\ dx
\end{equation*}
\end{theorem}

\textbf{Proof}: By (\ref{127}) we have that%
\begin{eqnarray*}
\int^{\circ }D\cdot \Phi \ dx &=&\int^{\circ }D\cdot \Phi ~\theta
_{E}^{\circ }\ dx=-\int^{\circ }\Phi \cdot D\theta _{E}^{\circ }~dx \\
&=&-\int_{\mathfrak{bd}\left( E\right) }^{\circ }\Phi \cdot \frac{D\theta
_{E}^{\circ }}{|D\theta _{E}^{\circ }|}~|D\theta _{E}^{\circ
}|~dx=\int^{\circ }\Phi \cdot \mathbf{n}_{E}^{\circ }\ |D\theta _{E}^{\circ
}|\ dx.
\end{eqnarray*}

$\square $

As we have seen, if $\partial \Omega $ is sufficiently smooth then $|D\theta
_{\Omega }^{\circ }|\ =\tilde{S}^{N-1}$; if not $%
|D\theta _{\Omega }^{\circ }|$ is sort of "generalized" measure such that $%
\int^{\circ }|D\theta _{\Omega }^{\circ }|\varphi ^{\circ }dx$ may assume
infinite values even if $\varphi \in C_{c}^{0}(\mathbb{R}^{N})$.  It is
remarkable that this identity holds for any vector valued ultrafunction $%
\mathbf{\phi };$  in particular, it holds for  ${\Phi }=\phi^\circ$  even when $\phi$ is a vector field with no regularity. It would be
interesting to investigate what happens when $\partial \Omega $ is a fractal
set of dimension $d$ and, in particular, to investigate the relation of the
infinite number $\int^{\circ }|D\theta _{\Omega }^{\circ }(x)|dx$  with the
fractal dimension $d.$

\subsection{Ultrafunctions and distributions}\label{ud}

One of the most important features of ultrafunctions is that they
can be seen (in a sense that we will make precise in this section)
as a generalization of distributions.

\begin{definition} We say that an ultrafunction $u$ is \textbf{distribution-like}
($DL$) if there exists a distribution $T$ such that for every $\varphi\in C^{\infty}_{c}(\mathbb{R}^{N})$
\[
\int^{\circ}u(x)\varphi^{\circ}(x)\,dx=\langle T,\varphi\rangle.
\]
\end{definition}

\textbf{Example 1.} If $f\in V(\mathbb{R}^{N})$, then $f^{\circ}$
is distribution-like since $f\varphi\in V(\mathbb{R}^{N})$ and hence
\[
\int^{\circ}f^{\circ}(x)\varphi^{\circ}(x)\,dx=\int f(x)\varphi(x)\,dx=\langle T_{f},\varphi\rangle
\]
where $T_{f}$ is the distribution induced by $f$.

\textbf{Example 2.} If $f\in\mathcal{L}^{1}_{loc}$, then $\tilde{f}$
is distribution-like since 
\[
\int^{\circ}\tilde{f}(x)\varphi^{\circ}(x)\,dx=\int f(x)\varphi(x)\,dx=\langle T_{f},\varphi\rangle.
\]

In general, if $f\in\mathcal{L}^{1}_{loc}$ we have 
\[
\int^{\circ}f^{\circ}(x)\varphi^{\circ}(x)\,dx\sim\int f(x)\varphi(x)\,dx
\]
and hence $f^{\circ}$ is not distribution-like but it is \emph{almost
distribution-like}.

\textbf{Example 3.} If $f^{\circ}$ is $DL$, then also $D^{m}f^{\circ}$
is $DL$ since 
\begin{eqnarray*}
\int^{\circ}D^{m}f^{\circ}(x)\varphi^{\circ}(x)\,dx & = & (-1)^{m}\int^{\circ}f^{\circ}(x)D^{m}\varphi^{\circ}(x)\,dx\\
 & = & (-1)^{m}\int f(x)\partial^{m}\varphi(x)\,dx=\langle\partial^{m}T_{f},\varphi\rangle.
\end{eqnarray*}

In particular, if $f\in C^{0,1}$, then $D^{m}f^{\circ}$ is $DL$
and we can associate the distribution $T=\partial^{m}T_{f}$ to the
ultrafunction $D^{m}f^{\circ}(x)$.

\bigskip{}

It is easy to see that:

\begin{proposition}\label{P1} For any distribution $T$ there exists
a distribution-like ultrafunction $u_{T}$. \end{proposition}

\textbf{Proof.} We split $V(\Gamma^{N})$ as follows 
\[
V(\Gamma^{N})=C^{\infty}_{c}(\Gamma^{N})\oplus W
\]
where $W$ is a complementary space of $C^{\infty}_{c}(\Gamma^{N})$.
Let 
\[
P_{\infty}:V(\Gamma^{N})\to C^{\infty}_{c}(\Gamma^{N})
\]
be the corresponding projection.

For every $v\in V(\Gamma^{N})$ we set 
\[
\int^{\circ}u_{T}v\,dx=\lim_{\lambda\uparrow\Lambda}\langle T,(P_{\infty}v)_{\lambda}\rangle.
\]

Then for every $\varphi\in C^{\infty}_{c}(\mathbb{R}^{N})$ 
\[
\int^{\circ}u_{T}(x)\varphi^{\circ}(x)\,dx=\lim_{\lambda\uparrow\Lambda}\langle T,\varphi\rangle=\langle T,\varphi\rangle.
\]

$\square$

\bigskip

Clearly $u_{T}$ is not univocally defined since, in the proof of Prop. \ref%
{P1}, the splitting $C_{c}^{\infty }(\Gamma ^{N})\oplus W$ can be chosen
arbitrarily. So it make sense to set%
\begin{equation*}
\left[ u\right] _{\mathfrak{D}^{\prime }}=\left\{ v\in V(\Gamma ^{N})\ |\
v\approx _{\mathfrak{D}^{\prime }}u\right\}
\end{equation*}%
where%
\begin{equation*}
v\approx _{\mathfrak{D}^{\prime }}u:\Leftrightarrow \forall \varphi \in
C_{c}^{\infty }(\mathbb{R}^{N}),\ \int^{\circ }\left( u-v\right) \varphi 
{{}^\circ}%
dx=0
\end{equation*}%
Then there is a bijective map 
\begin{equation}
\Psi :\mathfrak{D}^{\prime }\rightarrow V_{DL}/\approx _{\mathfrak{D}%
^{\prime }}  \label{psi}
\end{equation}%
where $V_{DL}$ is the set of distribution like ultrafunction and%
\begin{equation*}
\Psi (T)=\left\{ u\in V(\Gamma ^{N})\ |\ \forall \varphi \in C_{c}^{\infty }(%
\mathbb{R}^{N}),\ \int^{\circ }u\varphi 
{{}^\circ}%
dx=\left\langle T,\varphi \right\rangle \right\}
\end{equation*}%
is a class of ultrafunctions equivalent to a distribution.
The linear map is $\Psi $ consistent with the distributional derivative,
namely:

\begin{proposition}
If $\Psi (T)=\left[ u\right] _{\mathfrak{D}^{\prime }}$ then $\Psi (\partial
_{i}T)=\left[ D_{i}u\right] _{\mathfrak{D}^{\prime }}.$
\end{proposition}

\textbf{Proof}: If $\Psi (T)=\left[ u\right] _{\mathfrak{D}^{\prime }},$ then

\begin{equation*}
\int^{\circ }D_{i}u\varphi 
{{}^\circ}%
~dx=-\int^{\circ }uD_{i}\varphi 
{{}^\circ}%
dx
\end{equation*}%
Since $\varphi \in C_{c}^{\infty }(\mathbb{R}^{N})\cap V(\mathbb{R}^{N}),$
then by Th. \ref{TU}, $D_{i}\varphi 
{{}^\circ}%
=\left( \partial _{i}\varphi \right) 
{{}^\circ}%
$ and so%
\begin{eqnarray*}
\int^{\circ }Du\varphi 
{{}^\circ}%
~dx &=&-\int^{\circ }uD\varphi 
{{}^\circ}%
dx=-\int^{\circ }u\left( \partial _{i}\varphi \right) 
{{}^\circ}%
dx \\
&=&-\left\langle T,\partial _{i}\varphi \right\rangle =\left\langle \partial
_{i}T,\varphi \right\rangle
\end{eqnarray*}%
Hence $\left[ D_{i}u\right] _{\mathfrak{D}^{\prime }}=\Psi (\partial _{i}T).$

$\square $

At this point it is a natural question to ask if there exists a linear map 
\begin{equation*}
\Phi :\mathfrak{D}^{\prime }\rightarrow V(\Gamma ^{N})
\end{equation*}%
which selects in any equivalence class $\left[ u\right] _{\mathfrak{D}%
^{\prime }}=\Psi \left( T\right) $ a distribution-like ultrafunction $\Phi
\left( T\right) $ in a way consistent with the distributional derivative,
namely%
\begin{equation}
\Phi \left( \partial _{i}T\right) =D_{i}\Phi \left( T\right)  \label{mer}
\end{equation}

Actually this goal can be achieved in several ways. For example, using the
argument of Example 3, every distribution of finite order can be univocally
represented by $D^{m+1}f^{\circ }$ with $f\in C^{1}(\mathbb{R}).$

Now, we will achieve this goal by via the splitting%
\begin{equation*}
V(\Gamma ^{N})=V^{\infty }(\Gamma ^{N})\oplus V^{\infty }(\Gamma
^{N})^{\perp }.
\end{equation*}%
where $V^{\infty }(\Gamma )$ has been define by Def. \ref{ddd}.

If we denote by $\Pi _{\infty }u$ and $\Pi _{\infty }^{\perp }u$ the
relative "orthogonal" projection of $u$ on $V^{\infty }(\Gamma )\ $and $%
V^{\infty }(\Gamma )^{\perp },$ every ultrafunction $u$ can be split as
follows%
\begin{equation}
u=\Pi _{\infty }u+\Pi _{\infty }^{\perp }u;  \label{181}
\end{equation}%
$\Pi _{\infty }u$ will be called the \textbf{smooth} \textbf{part} of $u$
and $\Pi _{m}^{\perp }u$ the \textbf{irregular part} of $u.$

\begin{definition}
\label{cina}For every $T\in \mathfrak{D}^{\prime }$, we denote by $T^{\circ
} $ the unique ultrafunction in $V^{\infty }(\Gamma ^{N})$ such that$\
\forall v\in V^{\infty }(\Gamma ^{N})$\ 
\begin{equation}
\int^{\circ }T^{\circ }(x)v(x)dx=\lim_{\lambda \uparrow \Lambda
}\left\langle T,v_{\lambda }\right\rangle .  \label{mer1}
\end{equation}
\end{definition}

This definition makes sense; in fact $v_{\lambda }\mapsto \lim_{\lambda
\uparrow \Lambda }\left\langle T,v_{\lambda }\right\rangle $ is a linear
functional over $V^{\infty }(\Gamma ^{N}),$ and hence there esists a unique
ultrafuncionfunction $T^{\circ }$ in $V^{\infty }(\Gamma ^{N})$ which
satisfies (\ref{mer1}). Clearly, $T^{\circ }$ is a $DL$-ultrafunction since $%
\forall \varphi \in C_{c}^{\infty }(\mathbb{R}^{N})$, $\varphi 
{{}^\circ}%
\in V^{\infty }(\Gamma ^{N})$ and hence 
\begin{equation*}
\int^{\circ }T^{\circ }(x)\varphi 
{{}^\circ}%
dx=\lim_{\lambda \uparrow \Lambda }\ \left\langle T,\varphi \right\rangle
=\left\langle T,\varphi \right\rangle .
\end{equation*}

\begin{theorem}
The map $T\mapsto T^{\circ }$ defined by (\ref{mer1}) satisfies (\ref{mer}),
namely,%
\begin{equation*}
D_{i}T^{\circ }=\left( \partial _{i}T\right) ^{\circ }
\end{equation*}
\end{theorem}

\textbf{Proof}: By Th. \ref{TU}, we have that $\forall v\in V_{c}^{\infty
}(\Gamma ^{N}),$ 
\begin{eqnarray*}
\int^{\circ }D_{i}T^{\circ }v~dx &=&-\int^{\circ }T^{\circ
}D_{i}v~dx=-\int^{\circ }T^{\circ }\left( \partial _{i}v\right) 
{{}^\circ}%
dx \\
&=&-\lim_{\lambda \uparrow \Lambda }\left\langle T,\partial _{i}v_{\lambda
}\right\rangle =\lim_{\lambda \uparrow \Lambda }\left\langle \partial
_{i}T,v_{\lambda }\right\rangle \\
&=&\int^{\circ }\left( \partial _{i}T\right) ^{\circ }v~dx
\end{eqnarray*}%
$\square $

\textbf{Example}: The delta ultrafunction $\delta _{a}$ is distribution-like
since for every $\varphi \in C_{c}^{\infty }(\mathbb{R}^{N})$, we have%
\begin{equation*}
\int^{\circ }\delta _{a}\varphi 
{{}^\circ}%
(x)dx=\varphi (a)=\lim_{\lambda \uparrow \Lambda }\ \left\langle T_{\delta
_{a}},\varphi \right\rangle ;
\end{equation*}%
(here we have used the simbol $T_{\delta _{a}}$ to distinguish the
distribution from the ultrafunction $\delta _{a}$). However 
\begin{equation}
\delta _{a}\neq T_{\delta _{a}}^{\circ }  \label{buti}
\end{equation}%
Actually, according to (\ref{181}),%
\begin{equation*}
\delta _{a}=T_{\delta _{a}}^{\circ }+\Pi _{\infty }^{\perp }\delta _{a},
\end{equation*}%
namely $T_{\delta _{a}}^\circ $ is the smooth part of $\delta _{a}.$ Here, it is
necessary to be careful since the Dirac measure $\mathbf{\delta }_{a}$ and
the Dirac distribution $T_{\delta _{a}}^{\circ }$ produce different
ultrafunctions, since $\delta _{a}\neq T_{\delta
_{a}}^{\circ }$.

\bigskip

Every function $f\in \mathcal{L}^{1}$ defines a distribution $T_{f};$ then,
given $f\in \mathcal{L}^{1}$, we can define three ultrafunctions: $f%
{{}^\circ}%
$,$\ \tilde{f}$ and $T_{f}^{\circ }.$ If $f\in V_{c}^{\infty }(\mathbb{R}),$
then, 
\begin{equation*}
f^{\circ }=\ \tilde{f}=T_{f}^{\circ }.
\end{equation*}

In the other cases, both $\tilde{f}$ and $T_{f}^{\circ }$ can be considered
as "approximations" of $f^{\circ }$. In particular the relation between $\tilde{f}$ and $%
T_{f}^{\circ }$ is described by the following proposition:

\begin{proposition}
\label{pio}If $f\in \mathcal{L}^{1},$ then%
\begin{equation*}
T_{f}^{\circ }=\Pi _{\infty }\tilde{f},
\end{equation*}%
namely $T_{f}^{\circ }$ is the smooth part of $\tilde{f}.$
\end{proposition}

\textbf{Proof}: By (\ref{181}), we have that $\forall v\in V^{\infty
}(\Gamma ^{N}),$ 
\begin{eqnarray*}
\int^{\circ }\Pi _{\infty }\tilde{f}v~dx &=&\int^{\circ }\tilde{f}\Pi
_{\infty }v~dx=\int^{\circ }\tilde{f}v~dx \\
&=&\lim_{\lambda \uparrow \Lambda }\left\langle f,v_{\lambda }\right\rangle
=\lim_{\lambda \uparrow \Lambda }\left\langle T_{f},v_{\lambda
}\right\rangle =\int^{\circ }T_{f}^{\circ }v~dx
\end{eqnarray*}%
Since both $T_{f}^{\circ }$ and $\Pi _{\infty }f%
{{}^\circ}%
\in V_{c}^{\infty }(\Gamma ^{N})$, the conclusion follows.

$\square $

\subsection{Time-dependent ultrafunctions}\label{TDU}

In evolution problems the time variable plays a different role than
the space variables; the functional spaces used in these problems,
such as $C^{k}([0,T],W^{1}_{0}(\Omega))$, reflect this fact. The
same is true in the framework of ultrafunctions. This section is devoted
to the description of the appropriate ultrafunction spaces for evolution
problems.

In static problems, the basic idea is their reduction to a hyperfinite
"discrete space"; in evolution problems this strategy is not optimal
since, under many aspects, discrete dynamical systems are more complex
than continuous dynamical systems. Hence our strategy consists in
keeping the space discrete while using a continuum time variable.
In this case, an evolution partial differential equation is reduced
to an ordinary differential equation with a hyperfinite number of variables.

We recall that (see Sec.  \ref{hs}), for every 
\[
c=\lim_{\lambda\uparrow\Lambda}c_{\lambda}(t_{\lambda})\in C^{k}(\mathbb{E}),
\]
its time derivative is given by
\begin{equation}
\partial_{t}c(t):=\lim_{\lambda\uparrow\Lambda}\partial_{t}c_{\lambda}(t_{\lambda}),\qquad t=\lim_{\lambda\uparrow\Lambda}t_{\lambda}.\label{nd}
\end{equation}

\begin{definition}\label{lella3} The space of \textbf{time-dependent
ultrafunctions} of order $k$ is defined as 
\[
C^{k}(\mathbb{E},V(\Gamma^{N}))=C^{k}(\mathbb{E})\otimes V(\Gamma^{N}).
\]
\end{definition}

Every time-dependent ultrafunction can be represented by the following
hyperfinite sum: 
\begin{equation}
u(t,x)=\sum_{a\in\Gamma^{N}}c_{a}(t)\chi_{a}(x),\qquad c_{a}\in C^{k}(\mathbb{E}).\label{u}
\end{equation}

The notion of generalized derivative in the space variable is defined
by linearity: 
\begin{equation}
D_{i}u(t,x)=D_{i}\!\left(\sum_{a\in\Gamma^{N}}c_{a}(t)\chi_{a}(x)\right)=\sum_{a\in\Gamma^{N}}c_{a}(t)D_{i}\chi_{a}(x).\label{dx}
\end{equation}

The time derivative in $C^{1}(\mathbb{E},V(\Gamma^{N}))$ is defined
by means of (\ref{nd}): 
\begin{equation}
\partial_{t}u(t,x)=\partial_{t}\!\left(\sum_{a\in\Gamma^{N}}c_{a}(t)\chi_{a}(x)\right)=\sum_{a\in\Gamma^{N}}\partial_{t}c_{a}(t)\chi_{a}(x).\label{dt}
\end{equation}

If $f\in\mathfrak{F}(\mathbb{R}^{N+1})$, the corresponding time-dependent
ultrafunction is given by 
\[
f^{\circ}(t,x)=\sum_{a\in\Gamma^{N}}f^{*}(t,a)\chi_{a}(x),
\]
where $f^{*}(t,a)$ is defined by (\ref{star}), namely 
\[
f^{*}(t,x)=\lim_{\lambda\uparrow\Lambda}f(t_{\lambda},x_{\lambda}).
\]

Hence 
\[
f^{\circ}:\mathbb{E}\times\Gamma^{N}\to\mathbb{E}
\]
is an extension of $f$.

In particular, if $f\in C^{1}(\mathbb{R},C^{1}(\mathbb{R}^{N}))\simeq C^{1}(\mathbb{R}^{N+1})$,
the time partial derivative is defined in the natural way: 
\begin{equation}
(\partial_{t}f)^{*}(t,x)=\lim_{\lambda\uparrow\Lambda}\partial_{t}f(t_{\lambda},x_{\lambda}).\label{ddt}
\end{equation}

\begin{theorem}\label{TEV} If $f\in C^{k}(\mathbb{R},C^{m,1}_{c}(\mathbb{R}^{N}))$
and $m\ge0$, then 
\[
\partial_{t}f^{\circ}(t,x)=\sum_{a\in\Gamma^{N}}(\partial_{t}f)^{*}(t,a)\chi_{a}(x).
\]
Moreover, for $i=1,\dots,N$ and $m\ge1$, 
\[
D_{i}f^{\circ}\in C^{k}(\mathbb{E},V^{m-1}(\Gamma^{N}))
\]
and 
\[
D_{i}f^{\circ}(t,x)=\sum_{a\in\Gamma^{N}}(\partial_{i}f)^{\circ}(t,a)\chi_{a}(x).
\]
\end{theorem}

\textbf{Proof.} The first equality follows immediately from (\ref{dt})
and (\ref{ddt}). To prove the second statement, fix $t$; then the
result follows from Corollary~\ref{CC}.

$\square$

\section{Some applications}\label{SA}

In this section we sketch how the theory of fine ultrafunctions can
be used in the study of Partial Differential Equations. In the framework
of ultrafunctions, a very large class of problems is well posed and
has solutions. Very often, difficult \textit{a priori} estimates are
not necessary in order to prove existence, but only to understand
the properties of the solution (qualitative analysis).

In particular, if one considers a problem arising in Physics or in
Geometry, it is interesting to investigate whether the generalized
solutions describe the physical or the geometric phenomenon. We refer
to \cite{ultra},\cite{BBG},\cite{benciISO},\cite{belu2012},\cite{belu2013},
\cite{blc},\cite{milano},..., \cite{bls}, where problems of this
type have been treated in the framework of ultrafunctions.

In this section we limit ourselves to presenting some new examples
in order to illustrate the use of fine ultrafunctions, with particular
emphasis on the study of ill-posed problems. Obviously, each example
is treated superficially; a deeper analysis of each case would probably
deserve a separate paper.

\subsection{Ultrafunction solutions of PDE's}\label{evp copy(1)}

Let $\Omega\subset\mathbb{R}^{N}$ be an open set and let 
\[
A(x,\partial_{i}):\mathcal{D}^{m}_{A}(\Omega)\rightarrow C(\Omega)
\]
be a differential operator of order $m$. Here $\mathcal{D}^{m}_{A}(\Omega)$
denotes the domain of $A(x,\partial_{i})$, including the boundary
conditions.

We consider the following problem: find $u\in\mathcal{D}^{m}_{A}(\Omega)$
such that 
\[
A(x,\partial_{i})[u]=0.
\]

This problem can be translated into the following one: find $u\in\mathcal{D}^{m}_{A}(\Omega^{\circ})$
such that 
\begin{equation}
A^{\circ}(x,D_{i})[u]=0.\label{P}
\end{equation}

Here $\mathcal{D}^{m}_{A}(\Omega^{\circ})$ and 
\[
A^{\circ}:\mathcal{D}^{m}_{A}(\Omega^{\circ})\rightarrow V(\Gamma^{N})
\]
are the natural extensions of $\mathcal{D}^{m}_{A}(\Omega)$ and $A$
respectively (see e.g. (\ref{666+})).

The most general way to treat a problem such as (\ref{P}) is a variant
of the Faedo--Galerkin method. For every $\lambda\in\mathfrak{L}$
we consider the finite dimensional approximation 
\[
u_{\lambda}\in V_{\lambda}(\Omega)\cap\mathcal{D}^{m}_{A}(\Omega)
\]

\begin{equation}
A^{\circ}(x,D_{i,\lambda})u_{\lambda}(x)=0\label{P2}
\end{equation}

where 
\[
D_{i,\lambda}u=\left\langle \partial_{i}u_{\lambda},\delta_{a_{\lambda}}\right\rangle _{\lambda}.
\]

Let $S_{\lambda}$ be the set of approximate solutions of problem
(\ref{P}). If $\mathcal{U}$-eventually $S_{\lambda}\neq\varnothing$,
then also 
\[
S_{\Lambda}:=\{u\mid u_{\lambda}\in S_{\lambda}\}\neq\varnothing
\]
and, taking the $\Lambda$-limit, every $u\in S_{\Lambda}$ is a solution
of (\ref{P}).

\begin{theorem} \label{667} If $w\in\mathcal{D}^{m}_{A}(\Omega)\cap C^{m,1}_{loc}(\Omega)$
is a classical solution of Problem (\ref{P}), then $w^{\circ}$ is
a solution of (\ref{P2}). \end{theorem}

\textbf{Proof.} If $w$ is a classical solution, then for $\lambda$
sufficiently large and for every $x\in\Gamma^{N}_{\lambda}\cap\Omega$
\[
A(x,D_{i,\lambda})[w(x)]=A(x,\partial_{i})[w(x)]=0.
\]

Hence 
\[
w^{\circ}=\lim_{\lambda\uparrow\Lambda}w(x)\in S_{\Lambda}.
\]

$\square$

\bigskip{}

In general, a problem admits an ultrafunction solution even when the
classical problem has no solution, since the approximate solutions
$u_{\lambda}$ may fail to converge in $\mathcal{D}^{m}_{A}(\Omega)$.

\bigskip{}

\textbf{Example 1.} Probably the simplest example is provided by the
following Dirichlet problem 
\[
-\Delta u=f(x)\quad\text{in }\Omega
\]

\[
u(x)=0\quad\text{for }x\in\partial\Omega
\]

when $f\notin L^{1}(\Omega)$. In this case the problem has no solution,
not even in the sense of distributions.

The problem can be reformulated in the framework of ultrafunctions
as 
\[
-D^{2}u=f^{\circ}(x)\quad\text{in }\Omega^{\circ}
\]

\[
u(x)=0\quad\text{for }x\in\partial\Omega^{\circ}.
\]

Using the generalized Faedo--Galerkin method it is not difficult
to see that this problem admits an ultrafunction solution. The details
when $\partial\Omega$ is not smooth will be discussed in the next
section.

\bigskip{}

Sometimes it is not convenient to employ the Faedo--Galerkin approximation
and it is preferable to use other methods.

\bigskip{}

\textbf{Example 2.} Consider the previous problem in dimension $1$:
\[
-D^{2}u=f^{\circ}(x)\quad\text{in }[0,1]^{\circ}
\]

\[
u(0)=u(1)=0.
\]

In this case we can use the antiderivative as in Theorem \ref{TT}.
The solution is 
\[
u(x)=D^{-2}_{0}f^{\circ}(x)-D^{-2}_{0}f^{\circ}(0)+\left[D^{-2}_{0}f^{\circ}(0)-D^{-2}_{0}f^{\circ}(1)\right]x.
\]

If $f$ is given explicitly we can write $u(x)$ almost explicitly.
For instance, if 
\[
f(x)=\begin{cases}
\frac{1}{x^{p}}, & p>2,\quad x\in(0,1],\\
0, & x=1,
\end{cases}
\]

then 
\[
u(x)=\frac{1}{(p-1)(p-2)}\frac{\chi_{\Theta}}{x^{p-2}}+\psi(x)+A+Bx
\]

where 
\[
\Theta=[0,1]^{\circ}\setminus\mathfrak{vic}^{2}(1),
\]
$\psi(x)$ satisfies $\mathfrak{supp}(\psi)\subset\mathfrak{vic}^{2}(1)$,
and $A,B$ are determined by the boundary conditions: 
\[
A=0,
\]

\[
B=-\psi(1).
\]

Hence 
\[
u(x)=\frac{1}{(p-1)(p-2)}\frac{\chi_{\mathfrak{int}([0,1]^{\circ})}}{x^{p-2}}+\psi(x)-\psi(1)x.
\]

Notice that $\psi(1)$ is an infinite number; hence for $x\in(0,1)$
also $u(x)$ is infinite. Clearly this example is not physically relevant,
but it shows how ultrafunctions can be used to investigate even awkward
situations.

\bigskip{}

\textbf{Example 3.} Sometimes a solution can be seen directly; for
example, by (\ref{Dd}) it is immediate that $\beta\delta_{a}(x)$
($\beta\in\mathbb{E}$, $a\in\Gamma_{\mathfrak{int}}$) is a solution
of 
\[
uDu=0.
\]

Other solutions are 
\begin{equation}
u_{\Theta}(x):=\sum_{a\in\Theta}\beta_{a}\delta_{a}(x)\label{elena}
\end{equation}

where $\Theta\subset\Gamma_{\mathfrak{int}}$ is a hyperfinite set
such that for every $a,b\in\Theta$ 
\begin{equation}
a\neq b\Rightarrow\mathfrak{vic}(a)\cap\mathfrak{vic}(b)=\varnothing.\label{eleonora}
\end{equation}

\subsection{Boundary value problems}\label{soe}

In this section we deal with the following semilinear differential
operator 
\begin{equation}
A(x,\nabla)u:=-\nabla\cdot\left[k(x,u)\nabla u\right]+f(x,u)\label{666}
\end{equation}
where $f(x,u)$ is a function and $k(x,u)$ is either a function or
an $(N\times N)$ matrix. For the moment we assume that $k$ and $f$
are smooth in order to avoid technicalities that are not relevant
for our purposes.

This operator can be easily extended to ultrafunctions by setting
\begin{equation}
A^{\circ}(x,D)u:=-D\cdot\left[k(x,u)Du\right]+f(x,u)=0\qquad\text{in }\Gamma^{N}\label{666+}
\end{equation}
where $k$ and $f$ denote the natural extensions of the corresponding
functions, namely $k=k^{*}$ and $f=f^{*}$ (see (\ref{star})).

Let $\Omega\subseteq\mathbb{R}^{N}$ be an open set and consider the
following equation with Dirichlet boundary conditions 
\begin{equation}
-\nabla\cdot\left[k(x,u)\nabla u\right]+f(x,u)=0\qquad\text{in }\Omega\label{ab}
\end{equation}

\begin{equation}
u(x)=0\qquad\text{for }x\in\partial\Omega.\label{b}
\end{equation}

A function satisfying (\ref{ab}), (\ref{b}) is called a classical
solution if 
\[
u\in C^{2}(\Omega)\cap C^{0}(\overline{\Omega}).
\]

This problem can be translated into the framework of ultrafunctions
by setting 
\begin{equation}
-D\cdot\left[k(x,u)Du\right]+f(x,u)=0\qquad\text{in }\Omega^{\circ}\label{b'}
\end{equation}

\begin{equation}
u(x)=0\qquad\text{if }x\in\partial\Omega^{\circ}.\label{b''}
\end{equation}

If $w\in C^{2,1}_{0}(\overline{\Omega})$, then by Theorem~\ref{TU}-(\ref{U3})
the function $w^{\circ}$ is a solution of (\ref{b'}), (\ref{b''}).

If instead 
\[
w\in C^{2}(\Omega)\cap C^{0}(\overline{\Omega})
\]
but $w$ is not the restriction to $\Omega$ of a $C^{2,1}$ function,
then in general $w^{\circ}$ satisfies equation (\ref{b'}) only for
points 
\[
x\in\mathfrak{int}(\Omega^{\circ}).
\]

Indeed, if 
\[
x\in\mathfrak{bd}(\Omega)\cap\Omega^{\circ},
\]
the value of $Du$ may depend on the values of $u$ at points $y\notin\Omega^{\circ}$
(see also Section~\ref{sgd}).

An ultrafunction solution takes into account the irregularities infinitely
close to $\partial\Omega$, while a classical solution in $C^{2}(\Omega)\cap C^{0}(\overline{\Omega})$
does not detect them. Thus, in general, (\ref{b'}), (\ref{b''})
is not the most appropriate translation of (\ref{ab}), (\ref{b}).
Since the problem has a variational structure, a better translation
is given by the weak formulation.

\begin{definition} \label{FF}

We set 
\[
V_{0}(\Omega^{\circ}):=\left\{ v_{|\overline{\Omega^{\circ}}}\;\middle|\;v\in V(\Gamma^{N}),\;v_{|\partial\Omega^{\circ}}=0\right\} .
\]

We say that $u\in V_{0}(\Omega^{\circ})$ is an ultrafunction solution
of (\ref{ab}), (\ref{b}) if 
\begin{equation}
\forall v\in V_{0}(\Omega^{\circ}),\qquad\int^{\circ}_{\Omega}\left[k(x,u)Du\cdot Dv+f(x,u)v\right]dx=0.\label{ab'}
\end{equation}

\end{definition}

Obviously, if $w$ is a classical weak solution in $C^{1,1}_{loc}(\Omega)$,
then 
\[
w^{\circ}_{|\overline{\Omega^{\circ}}}
\]
is an ultrafunction solution.

If $\partial\Omega$ is not sufficiently regular, in order to investigate
what happens near $\partial\Omega^{\circ}$ it is useful to write
the solution of (\ref{ab'}) in strong form.

For every $a\in\Omega^{\circ}$ we have $\delta_{a}\in V_{0}(\Omega^{\circ})$.
Hence, if $u_{|\Omega^{\circ}}$ is a solution of (\ref{ab'}),

\begin{eqnarray*}
0 & = & \int^{\circ}_{\Omega}\left[k(x,u)Du\cdot D\delta_{a}+f(x,u)\delta_{a}\right]dx\\
 & = & \int^{\circ}\theta^{\circ}_{\Omega}(x)\left[k(x,u)Du\cdot D\delta_{a}+f(x,u)\delta_{a}\right]dx\\
 & = & \int^{\circ}\left(-D\cdot\left[\theta^{\circ}_{\Omega}(x)k(x,u)Du\right]+\theta^{\circ}_{\Omega}(x)f(x,u)\right)\delta_{a}\,dx\\
 & = & -D\cdot\left[\theta^{\circ}_{\Omega}(a)k(a,u)Du\right]+\theta^{\circ}_{\Omega}(a)f(a,u).
\end{eqnarray*}

Hence 
\begin{equation}
-D\cdot\left[\theta^{\circ}_{\Omega}(x)k(x,u)Du\right]+\theta^{\circ}_{\Omega}(x)f(x,u)=0\qquad\text{if }x\in\Omega^{\circ}.\label{dbc}
\end{equation}

\begin{equation}
u(x)=0\qquad\text{if }x\in\partial\Omega^{\circ}.\label{dbc+}
\end{equation}

In general, for some 
\[
x\in\mathfrak{bd}(\Omega^{\circ})\cap\Omega^{\circ}
\]
it may happen that 
\[
D\cdot\left[\theta^{\circ}_{\Omega}k(x,u)Du\right]\neq D\cdot\left[k(x,u)Du\right].
\]
However equation (\ref{b'}) is satisfied for all 
\[
x\in\mathfrak{int}(\Omega^{\circ}).
\]

\bigskip{}

If a problem has no classical solution, usually one looks for weak
solutions in some Sobolev space or in a space of distributions. However,
if there are no weak solutions, we may look for ultrafunction solutions.

For instance, if in (\ref{666}) the function $k(x,u)$ changes sign,
the problem becomes ill posed and generally has no distributional
solution. However in the framework of ultrafunctions we have the following
result.

\begin{theorem} \label{678}

If there exist $c,R\in\mathbb{E}^{+}$ such that 
\begin{equation}
\|u\|=R\Rightarrow\int^{\circ}_{\Omega}\left[k(x,u)|Du|^{2}+f(x,u)u\right]dx\ge c>0,\label{pip+}
\end{equation}
then problem (\ref{ab'}) admits at least one solution.

\end{theorem}

\textbf{Proof}

The result is a standard application of the Brouwer fixed point theorem
and the Faedo--Galerkin method.

We set 
\[
V_{0,\lambda}(\Omega^{\circ}):=\left\{ v_{|\overline{\Omega^{\circ}}}\mid v\in V_{\lambda}(\Gamma^{N}),\;v_{|\partial\Omega^{\circ}}=0\right\} .
\]

Let 
\[
\mathcal{A}(u)=-D_{\lambda}\cdot\left[k(x,u)D_{\lambda}u\right]+f(x,u).
\]

Define 
\[
M:=\max_{\|u\|\le R}\|\mathcal{A}(u)\|,\qquad b:=\frac{2c}{M^{2}}.
\]

Set 
\[
\mathcal{B}(u)=u-b\mathcal{A}(u).
\]

Then $\mathcal{B}$ is continuous on 
\[
V_{0,\lambda}(\Omega^{\circ})\cap B_{R},\qquad B_{R}=\{u\in V(\Gamma^{N}):\|u\|\le R\}.
\]

Indeed

\begin{eqnarray*}
\|\mathcal{B}(u)\|^{2} & = & \|u\|^{2}-2b\int^{\circ}_{\Omega}\mathcal{A}(u)u\,dx+b^{2}\|\mathcal{A}(u)\|^{2}\\
 & \le & \|u\|^{2}-2bc+b^{2}M^{2}\\
 & = & \|u\|^{2}.
\end{eqnarray*}

Hence by the Brouwer fixed point theorem there exists $u_{0,\lambda}$
such that

\[
\mathcal{B}(u_{0,\lambda})=u_{0,\lambda},
\]

namely

\[
u_{0,\lambda}-b\mathcal{A}(u_{0,\lambda})=u_{0,\lambda}
\]

and therefore

\[
\mathcal{A}(u_{0,\lambda})=0.
\]

Taking the $\Lambda$-limit gives the desired ultrafunction solution.

$\square$ 

\subsection{Some examples of boundary value problems}

Let us consider the following simple case: 
\begin{equation}
-\nabla\cdot\left[k(u)\nabla u\right]=h(x)\qquad\text{in }\Omega,\qquad u_{|_{\partial\Omega^{\circ}}}=0.\label{677}
\end{equation}

If $k(u)\geq b>0$ and $k'(u)\geq0$, the operator 
\[
u\mapsto-\nabla\cdot\left[k(u)\nabla u\right]
\]
is a maximal monotone operator in $W^{1}_{0}(\Omega)$. Hence it is
well known that it has a unique solution. Moreover, if $h\in C^{0,1}_{c}(\Omega)$
and $\partial\Omega$ is smooth, then this solution belongs to 
\[
W^{1}_{0}(\Omega)\cap C^{2,1}(\Omega).
\]

By Theorem \ref{667}, the solution $w$ of (\ref{677}) produces
the unique ultrafunction solution $w^{\circ}$.

Next we assume only that 
\begin{equation}
k(u)\geq-b,\qquad h\in\mathfrak{F}(\Omega)\label{679}
\end{equation}
and we are interested in the case in which there exists a non empty
set $\Omega^{-}\subset\Omega$ such that 
\begin{equation}
\forall x\in\Omega^{-},\qquad k(u)<0.\label{671}
\end{equation}

This is a classical ill-posed problem. In the framework of ultrafunctions
we have the following result.

\begin{theorem} If 
\[
k(u)\geq k_{0}>0\qquad\text{for }|u|\geq M
\]
then problem (\ref{677}) has at least one ultrafunction solution.
\end{theorem}

\textbf{Proof}

We set 
\[
\|D\|=\max\left\{ |D\chi_{a}|\;|\;a\in\Gamma^{N}\right\} .
\]

Then for every set $E$ we have 
\begin{eqnarray*}
\int^{\circ}_{E}|Du(x)|^{2}dx & \leq & \int^{\circ}_{E}\left|\sum_{a\in E}u(a)D\chi_{a}(x)\right|^{2}dx\\
 & \leq & \int^{\circ}_{E}\left|\|D\|\sum_{a\in E}u(a)\chi_{a}(x)\right|^{2}dx\\
 & \leq & \|D\|^{2}\int^{\circ}_{E}|u(x)|^{2}dx.
\end{eqnarray*}

Now we can apply Theorem \ref{678}. We set 
\[
-k_{0}:=\min k(u)
\]
and 
\[
\Omega^{+}=\{x\in\Omega^{\circ}\;|\;|u(x)|\geq M\},\qquad\Omega^{-}=\Omega^{\circ}\setminus\Omega^{+}.
\]

By the Poincaré inequality (\ref{PI+}), taking $x_{0}\in\partial\Omega$,

\begin{eqnarray*}
\int^{\circ}_{\Omega}\left[k(u)|Du|^{2}-hu\right]dx & = & \int^{\circ}_{\Omega^{+}}k(u)|Du|^{2}dx+\int^{\circ}_{\Omega^{-}}k(u)|Du|^{2}dx-\int^{\circ}_{\Omega}hu\,dx\\
 & \geq & k_{0}\int^{\circ}_{\Omega^{+}}|Du(x)|^{2}dx-k_{0}\int^{\circ}_{\Omega^{-}}|Du(x)|^{2}dx-\|h\|\,\|u\|\\
 & = & k_{0}\int^{\circ}_{\Omega}|Du(x)|^{2}dx-2k_{0}\int^{\circ}_{\Omega^{-}}|Du(x)|^{2}dx-\|h\|\,\|u\|\\
 & \geq & k_{0}C(\Omega,x_{0})\|u\|^{2}-\|h\|\,\|u\|-2k_{0}\|D\|^{2}M^{2}\int^{\circ}_{\Omega^{-}}dx.
\end{eqnarray*}

If $\|u\|=R$ we obtain

\[
\int^{\circ}_{\Omega}\left[k(u)|Du|^{2}-h(x)u\right]dx\geq k_{0}C(\Omega,x_{0})R^{2}-\|h\|R-2k_{0}\|D\|^{2}M^{2}\int^{\circ}_{\Omega^{-}}dx.
\]

Since the last term does not depend on $R$, for $R$ sufficiently
large we obtain (\ref{pip+}).

$\square$

\bigskip{}

If $h=0$ and $k(S)=0$, some solutions of problem (\ref{677}) can
be written explicitly. They are given by

\begin{equation}
u_{\Theta}(x):=S+\sum_{a\in\Theta}\beta_{a}\delta_{a}(x)\label{tete}
\end{equation}

where $\Theta\subset\Omega^{\circ}$ is a set satisfying (\ref{eleonora}).

Indeed, if $x\in\Theta$, by (\ref{Dd})

\[
Du_{\Theta}(x)=\sum_{a\in\Theta}\beta_{a}D\delta_{a}(x)=0.
\]

If instead $x\notin\Theta$, then $u_{\Theta}(x)=S$ and therefore
$k(u_{\Theta})=k(S)=0$.

Hence 
\[
k(u_{\Theta})Du_{\Theta}=0
\]
and consequently 
\[
D\cdot\left[k(u_{\Theta})Du_{\Theta}\right]=0.
\]

Although these solutions may appear meaningless, they represent some
stationary solutions of the evolution equation

\[
\partial_{t}u=\nabla\cdot\left[k(u)\nabla u\right]
\]

which models certain physical and biological phenomena even when $k(u)$
is not always positive (see Section~\ref{evp}). 

\subsection{Calculus of variations}

In this section we show how the framework of ultrafunctions can be
applied

to problems in the calculus of variations. As a model case we consider
the

minimization of the functional 
\begin{equation}
J(u)=\int_{\Omega}F(x,u,\nabla u)\,dx,\qquad\Omega\subset\mathbb{R}^{N}.\label{J}
\end{equation}
where $\Omega\subseteq\mathbb{R}^{N}$ is a bounded open set. The
natural space where to work is a linear subspace of $C^{1}(\overline{\Omega})$
with suitable boundary conditions which we denote by $C^{1}_{\text{\textsc{bc}}}(\Omega)$.

If we translate this problem into an appropriate subspace of ultrafunctions
$V_{\text{\textsc{bc}}}(\Omega^{\circ})\subseteq\{v_{|_{\overline{\Omega^{\circ}}}}\ |\ v\in V(\Gamma^{N})\}$,
the functional (\ref{J}) becomes 
\begin{equation}
J^{\circ}(u):=\int^{\circ}_{\Omega}F(x,u,Du)\,dx.\label{JJ}
\end{equation}

We have the following result.

\begin{theorem} \label{nadia} If there exists a number $R\in\mathbb{E}^{+}$
such that $\forall u\in V_{\text{\textsc{bc}}}(\Omega^{\circ})$ with
$\|u\|\ge R$, there exists $u_{0}\in V_{\text{\textsc{bc}}}(\Omega^{\circ})$
with $\|u_{0}\|<R$ such that 
\begin{equation}
J^{\circ}(u)>J^{\circ}(u_{0}),\label{dis}
\end{equation}
then $J^{\circ}(u)$ has a minimizer $\bar{u}$ in $V_{\text{\textsc{bc}}}(\Omega^{\circ})$.
Moreover, if $J(u)$ has a minimizer $w\in C^{1}(\overline{\Omega})$,
then for every $x\in\overline{\Omega}$ we have 
\[
\bar{u}(x)=w^{\circ}(x).
\]
\end{theorem}

\textbf{Proof.}

For every $\lambda$ the functional $J$ is defined on the finite
dimensional space $V_{\lambda}$. Hence it admits a minimizer $\bar{u}_{\lambda}$.
Passing to the $\Lambda$--limit we obtain a minimizer $\bar{u}$
of $J^{\circ}$. $\square$

\bigskip{}

If $\partial\Omega$ is not sufficiently smooth, a possible minimizer
of $J$ may be achieved in a larger space $W^{1}_{\text{\textsc{bc}}}(\Omega)\supseteq C^{1}_{\text{\textsc{bc}}}(\Omega)$.
If the solution lies in $C^{1}_{\text{\textsc{bc}}}(\Omega)$, then
it coincides with the minimizer in $W^{1}_{\text{\textsc{bc}}}(\Omega)$
in the obvious sense. In the case of ultrafunctions the situation
is different. First of all, if (\ref{dis}) is satisfied, the solution
exists in every reasonable space; however, if we choose different
spaces, infinitesimal quantities have to be taken into account and
we might obtain different minimizers which behave in slightly different
ways.

Let us see an example.

\bigskip{}

\textbf{Example.} Let us consider the functional defined in $C^{1}(0,1)$
\begin{equation}
J(u)=\int_{[0,1]}\left[\left(|\partial u|^{2}-1\right)^{2}+u^{2}\right]dx\label{JL}
\end{equation}
which exhibits the well known Lavrentiev phenomenon: every minimizing

sequence \$u\_n\$ converges uniformly to \$0\$, while

\[
0=\lim_{n\to\infty}J(u_{n})\neq J(0)=1.
\]

In the framework of ultrafunctions a possible choice of the functional
space is $V^{1}_{0}((0,1)^{\circ})$. By Theorem \ref{nadia}, the
minimizer $\bar{u}\in V^{1}_{0}((0,1)^{\circ})$ exists and it is
easy to check that it satisfies the following Euler--Lagrange equation:
\begin{equation}
D\!\left[\theta^{\circ}_{(0,1)}(x)\left(|D\bar{u}|^{2}-1\right)D\bar{u}\right]-\frac{1}{2}\theta^{\circ}_{(0,1)}\bar{u}=\psi(x)\label{glu}
\end{equation}
where $\psi$ appears as a Lagrange multiplier and satisfies 
\[
\forall v\in V^{1}_{0}((0,1)^{\circ}),\qquad\int^{\circ}\psi v\,dx=0.
\]

However, it is possible to minimize the functional in $V_{0}((0,1)^{\circ})\supset V^{1}_{0}((0,1)^{\circ})$.
In this case we obtain a minimizer $\bar{u}_{0}$ such that 
\[
J^{\circ}(\bar{u}_{0})<J^{\circ}(\bar{u})
\]
and 
\[
D\!\left[\theta^{\circ}_{(0,1)}\left(|D\bar{u}_{0}|^{2}-1\right)D\bar{u}_{0}\right]-\frac{1}{2}\theta^{\circ}_{(0,1)}\bar{u}_{0}=0.
\]

Thus the function $\psi$ in equation (\ref{glu}) can be interpreted
as a ``structural force'' which prevents the graph of $u$ from
forming angles. The presence of this force increases the energy level
and we obtain 
\[
J^{\circ}(\bar{u})>J^{\circ}(\bar{u}_{0})>0.
\]

In both cases the energy exceeds by an infinitesimal quantity the
limit $\lim_{n\to\infty}J(u_{n})=1$. This example shows that the
choice of the functional space may affect the

structure of the minimizers. In the framework of ultrafunctions this

dependence appears only through infinitesimal quantities, but it may
still

reflect different physical or geometrical interpretations of the model.
Hence the choice of a particular space depends on the phenomenon that
we want to describe, provided that infinitesimal differences are relevant.

Using the variational structure it is very easy to prove the existence
of solutions for ill posed problems of different types. For example
it is possible to consider an overdetermined problem imposing redundant
boundary conditions such as 
\[
u(x)=0\quad\text{and}\quad Du(x)\cdot\mathbf{n}_{\Omega}(x)=0\quad\text{if }x\in\partial\Omega^{\circ}.
\]

Also in this case the solutions of this problem satisfy equation (\ref{b'})
in $\mathfrak{int}(\Omega^{\circ})$.

\textbf{Example.} Let us consider the following problem 
\[
D^{2}u=1\qquad\text{in }(-1,1)^{\circ}
\]
with boundary conditions 
\[
u(1)=u(-1)=0,\qquad Du(1)\cdot\mathbf{n}_{\Omega}(1)=Du(-1)\cdot\mathbf{n}_{\Omega}(-1)=0.
\]

The solution $\overline{u}$ is the minimizer of 
\[
J^{\circ}(u):={\int^{\circ}_{-1}}^{1}\left(|Du|^{2}+u\right)dx
\]
in 
\[
\{u_{|_{[-1,1]^{\circ}}}\ |\ u(1)=u(-1)=Du(1)=Du(-1)=0\},
\]
but it cannot be written explicitly. We can say that 
\[
u(x)=\frac{1}{2}x^{2}-c,\qquad c\sim\frac{1}{2}
\]
for $x\in\mathfrak{int}([-1,1]^{\circ})$, but we cannot say much
about $x\in\mathfrak{bd}([-1,1]^{\circ})$.

Obviously we have 
\[
J^{\circ}(\bar{u})>\int^{1}_{-1}\left|\partial\!\left(\frac{1}{2}x^{2}-\frac{1}{2}\right)\right|^{2}dx.
\]

The infinitesimal increase of energy is due to the effect of the Neumann
boundary conditions.

\subsection{Evolution problems}\label{evp}

Let $\Omega\subset\mathbb{R}^{N}$ be an open set and let 
\[
A(x,\partial_{i}):\mathcal{D}_{A}(\Omega)\to C(\Omega)
\]
be a differential operator. Here $\mathcal{D}_{A}(\Omega)$ denotes
the domain of $A(x,\partial_{i})$ including the boundary conditions
when $\Omega\neq\mathbb{R}^{N}$.

We consider the following Cauchy problem: given $u_{0}(x)\in\mathcal{D}_{A}(\Omega)$
find 
\begin{equation}
u\in C^{1}(\mathbb{R},\mathcal{D}_{A}(\Omega)),\label{ae}
\end{equation}
\begin{equation}
\partial_{t}u=A(x,\partial_{i})[u],\label{be}
\end{equation}
\begin{equation}
u(0,x)=u_{0}(x).\label{ce}
\end{equation}

A function satisfying (\ref{ae})--(\ref{ce}) is called a classical
solution. We want to translate this problem into the framework of
ultrafunctions.

Because of the nature of the Cauchy problem, as discussed in Section
\ref{TDU}, it is not convenient to work in $V(\Gamma^{N+1})$. It
is better to use the space $C^{1}(\mathbb{E},V(\Gamma^{N}))$ defined
in Definition \ref{lella3}.

Assume that the boundary conditions contained in $\mathcal{D}_{A}(\Omega)$
can be translated into a domain $\mathcal{D}_{A}(\Omega^{\circ})\subset V(\Gamma^{N})$.

Setting 
\[
C^{1}(\mathbb{E},\mathcal{D}^{\circ}_{A}(\Omega))=\{u\in C^{1}(\mathbb{E},V(\Gamma^{N}))\ |\ \forall t\in\mathbb{E},\ u(t,x)\in\mathcal{D}^{\circ}_{A}(\Omega)\}
\]
the problem (\ref{ae})--(\ref{ce}) becomes 
\begin{equation}
u\in C^{1}(\mathbb{E},\mathcal{D}^{\circ}_{A}(\Omega)),\label{aep}
\end{equation}
\begin{equation}
\partial_{t}u=A^{\circ}(x,D_{i})[u],\label{be'}
\end{equation}
\begin{equation}
u(0,x)=u_{0}(x).\label{ce'}
\end{equation}

Here $\partial_{t}$ is the natural extension of the classical time
derivative defined in (\ref{dt}). A solution of (\ref{aep})--(\ref{ce'})
will be called an ultrafunction solution of (\ref{ae})--(\ref{ce}).

If $w$ is a classical solution, then by Theorem \ref{TEV} $w^{\circ}$
is an ultrafunction solution.

Also in the evolution case the conditions guaranteeing the existence
of a solution are very weak.

\begin{theorem} \label{evol} Assume that $A(x,\partial_{i})[u]$
restricted to $V_{\lambda}(\Omega)$ is locally Lipschitz continuous
in $u$. Then there exists $T_{\Lambda}$ such that problem (\ref{aep})--(\ref{ce'})
has a unique ultrafunction solution for $t\in[0,T_{\Lambda})$.

Moreover, if an \emph{a priori} bound exists, then there is a unique
ultrafunction solution in $C^{1}(\mathbb{E},\mathcal{D}^{\circ}_{A}(\Omega))$.
\end{theorem}

\textbf{Proof.} Since $\mathcal{D}^{\circ}_{A}(\Omega)$ is a hyperfinite
dimensional vector space, the result follows from standard results
for ODEs.

$\square$

\begin{theorem} \label{dolores} If 
\begin{equation}
|A(x,D_{i})[u]|\le c_{1}+c_{2}|u|,\label{cielo}
\end{equation}
then problem (\ref{aep})--(\ref{ce'}) has a unique global in time
ultrafunction solution $u(t,x)$. \end{theorem}

\textbf{Proof.} Since $\mathcal{D}^{\circ}_{A}(\Omega)$ is hyperfinite
dimensional, the result follows again from classical theorems for
ODEs.

$\square$

\subsection{Some examples of evolution problems}

\textbf{Example 1.} Let $\Omega$ be a bounded open set and consider
the following problem: 
\begin{equation}
u\in C^{1}(\mathbb{R},\mathcal{D}_{A}(\Omega)),
\end{equation}
\begin{equation}
\partial_{t}u=\nabla\cdot\left[k(u)\nabla u\right],\label{liliana}
\end{equation}
\begin{equation}
u(0,x)=u_{0}(x).
\end{equation}

If $k(u)\geq k_{0}>0$, then (\ref{liliana}) is a parabolic equation
and the problem, under very weak assumptions, admits a classical solution.
If $k(u)<0$ for some $u\in\mathbb{R}$, then the problem is ill posed
and in general classical solutions do not exist.

Let us translate this problem into the framework of ultrafunctions.
Taking into account the results of Section \ref{soe}, we obtain 
\begin{equation}
u\in C^{1}(\mathbb{E},\mathcal{D}^{\circ}_{A}(\Omega^{\circ})),\label{aa1}
\end{equation}
\begin{equation}
\partial_{t}u=D\cdot\left[\theta^{\circ}_{\Omega}k(u)Du\right]\qquad x\in\Omega^{\circ},\label{aa2}
\end{equation}
\begin{equation}
u(0,x)=u^{\circ}_{0}(x).\label{aa3}
\end{equation}

Hence we can apply Theorem \ref{dolores} and obtain the existence
of a unique global solution.

\begin{theorem} \label{dolores+} If 
\begin{equation}
|k(r)|\leq k_{0},\label{billo}
\end{equation}
then problem (\ref{aa1}), (\ref{aa2}), (\ref{aa3}) has a global
solution. \end{theorem}

In many applications of equation (\ref{liliana}), $u$ represents
a density and $k(u)\nabla u$ its flux. Hence, by the generalized
Gauss theorem \ref{B}, the total mass $\int^{\circ}_{\Omega}u(x)\,dx$
is preserved up to the flux through $\partial\Omega$: 
\begin{eqnarray*}
\partial_{t}\int^{\circ}_{\Omega}u(t,x)dx & = & \int^{\circ}_{\Omega}\partial_{t}u(t,x)dx\\
 & = & \int^{\circ}_{\Omega}D\cdot\left[k(u)Du\right]dx\\
 & = & \int^{\circ}_{\mathfrak{bd}(\Omega)}k(u)Du\cdot\mathbf{n}_{\Omega}\,d\mu_{\partial\Omega}.
\end{eqnarray*}

If we want to model a situation where the flux of $u$ cannot cross
$\partial\Omega$, we impose Neumann boundary conditions. This corresponds
to choosing 
\[
\mathcal{D}^{\circ}_{A}(\Omega^{\circ})=V(\Omega^{\circ}).
\]

If for some $S\in\mathbb{R}$ we have $k(S)=0$, then the ultrafunctions
of the form (\ref{tete}) are stationary solutions of this problem,
and it is possible to analyze their stability and dynamics in several
situations.

\bigskip{}

\textbf{Example 2.} Consider the following conservation law: 
\[
u\in C^{1}(\mathbb{R},C^{1}(\mathbb{R}^{N})),
\]
\[
\partial_{t}u=\nabla\cdot F(x,u),
\]
\[
u(0,x)=u_{0}(x),
\]
where $F:\mathbb{R}^{N+1}\to\mathbb{R}^{N}$ is smooth.

This problem is not well posed and when $N>1$ very little is known.
However it is well posed in the framework of ultrafunctions: 
\[
u\in C^{1}(\mathbb{E},V(\Gamma^{N})),
\]
\begin{equation}
\partial_{t}u=D\cdot F(x,u),\label{x}
\end{equation}
\[
u(0,x)=u^{\circ}_{0}(x).
\]

The existence of a unique solution follows from Theorem \ref{evol}.
Moreover by Theorem \ref{B} we obtain 
\[
\partial_{t}\int^{\circ}u\,dx=\int^{\circ}D\cdot F(x,u)dx=0
\]
provided suitable boundary conditions at infinity are imposed, such
as 
\[
F(x,\cdot)=0\qquad x\in\mathfrak{bd}(\Gamma^{N}),
\]
where $\mathfrak{bd}(\Gamma^{N})$ was defined in (\ref{bi}).

A particular case of (\ref{x}) is Burgers' equation 
\[
\partial_{t}u=-u\partial_{x}u,
\]
\[
u(0,x)=u_{0}(x)\ge0.
\]

It is well known that this equation admits infinitely many weak solutions
preserving the mass. Among them, the entropy solution describes the
phenomena occurring in fluid mechanics.

In the framework of ultrafunctions Burgers' equation can be written
as 
\[
u\in C^{1}(\mathbb{E},V(\Gamma)),
\]
\begin{equation}
\partial_{t}u=-D_{x}\!\left(\frac{u|u|}{2}\right),\label{blue}
\end{equation}
\begin{equation}
u(0,x)=u^{\circ}_{0}(x),\qquad Du(\omega)=Du(-\omega)=0.\label{bluf}
\end{equation}

The right hand side of (\ref{blue}) does not satisfy (\ref{cielo});
however the existence of global solutions can still be established.
Indeed 
\begin{eqnarray*}
\partial_{t}\int^{\circ}|u|^{3}dx & = & 3\int^{\circ}(u|u|)\partial_{t}u\,dx\\
 & = & -\frac{3}{2}\int^{\circ}(u|u|)D_{x}(u|u|)dx\\
 & = & \frac{3}{2}\int^{\circ}D_{x}(u|u|)(u|u|)dx=0.
\end{eqnarray*}

Hence the solutions of (\ref{blue}) preserve also the quantity $\int^{\circ}|u|^{3}dx$.
These solutions are different from the entropy solution. The viscosity
solution can be modeled in the ultrafunction framework by 
\[
\partial_{t}u=-D_{x}\!\left(\frac{u^{2}}{2}\right)+\nu D^{2}_{x}u
\]
where $\nu$ is a suitable infinitesimal (see \cite{blbur}).

Finally we remark that in the ultrafunction framework equation (\ref{blue})
differs from the transport equation 
\begin{equation}
\partial_{t}u=-uD_{x}u\label{bue+}
\end{equation}
even when $u\ge0$. In fact at singular points 
\[
D_{x}(u^{2})\neq2uD_{x}u
\]
and (\ref{bue+}) is not a conservation law since it does not have
the form (\ref{x}). Moreover condition (\ref{cielo}) does not hold,
so global existence is not guaranteed by Theorem \ref{dolores}. Nevertheless
we have the following result.

\begin{theorem} If $u_{0}(x)\ge0$, then there exists a global solution
such that $u(t,x)\ge0$ for every $x\in\mathbb{E}$. Moreover if $u$
has compact support, the mass is preserved. \end{theorem}

\textbf{Proof.} Let $[0,T_{\Lambda})$ be the interval of existence
of the solution. If $u(t_{0},x)\ge0$ for every $x$ and some $t_{0}<T_{\Lambda}$,
then $u(t,x)\ge0$ for $t>t_{0}$. Indeed if $u(t_{0},x_{0})=0$ for
some $x_{0}\in\Gamma$, then $u(t,x_{0})=0$ for $t>t_{0}$.

Moreover 
\[
\partial_{t}\int^{\circ}u\,dx=-\int^{\circ}uD_{x}u\,dx=\int^{\circ}D_{x}u\,u\,dx=0,
\]
hence 
\[
\int^{\circ}u(t,x)\,dx=\int^{\circ}u_{0}(x)\,dx.
\]

Since $u(t,x)\ge0$, also $\int^{\circ}|u|dx$ remains bounded and
therefore the solution exists for all $t\in[0,T]$.

$\square$ 

\section{A model for ultrafunctions}\label{CSU}

In order to prove the consistency of Axioms $1,\ldots,4$ we construct
a model.
This construction is rather involved since there are many details to take account of. Hence we outline the general strategy:
\begin{enumerate}
    \item We construct a suitable grid  $\Gamma_\eta$ of mesh $\eta.$
    \item We define a hyperfinite space  $W(\mathbb{E})$ which contains the natural extensions of the epilogic functions.  The functions in $W(\mathbb{E})$ can be well approximated by step functions of mesh $\eta$  (see Def. \ref{61} ).
    \item We take the space $\breve{V}(\mathbb{E})$ of all the step functions relative to the grid $\Gamma_\eta$  and we deform some of them in order to get a space  $\mathring{V}(\mathbb{E})$  isomorphic and very "close" to $\breve{V}(\mathbb{E})$ and which contains  $W(\mathbb{E})$.
    \item We define a step derivative on $\breve{V}(\mathbb{E})$ and we show that it satisfies some desired properties.
    \item We show that the properties of the step derivative and other useful properties of $\breve{V}(\mathbb{E})$ can be transferred to $\mathring{V}(\mathbb{E})$ .  Finally we build  $V(\Gamma)=V(\Gamma_\eta)$ over $\mathring{V}(\mathbb{E}).$
\end{enumerate}

\subsection{The hyperfinite grid}

If $A\subset\mathbb{R}$, we can define a peculiar hyperfinite set
containing $A$ by setting 
\[
A^{\circledast}:=\left\{ \lim_{\lambda\uparrow\Lambda}x_{\lambda}\ \big|\ \forall\lambda\in\mathfrak{L},\ x_{\lambda}\in A \cap \lambda\right\} .
\]
In the following the sets $\mathbb{R}^{\circledast}$ and $\mathbb{Q}^{\circledast}$will be the most relevant ones.

Given an infinite number $\omega\in\mathbb{N}^{*}$ such that $\mathbb{R}^{\circledast}\subset(-\omega,\omega)$
and an infinitesimal number $\eta=\frac{1}{\beta}$ with $\beta\in\mathbb{N}^{*}$,
the set 
\[
G_{\eta}:=\{m\eta\mid|m\eta|\le\omega\}
\]
is a hyperfinite grid with mesh $\eta$  and numerosity $2\beta \omega +1$,  i.e.,  
\[
\mathfrak{num}\,(G_{\eta}) = \lim_{\lambda\uparrow_{\mathcal{U}}\Lambda}|G_{\eta}\cap \lambda| = 2\beta \omega + 1
.\] We take $\beta$ sufficiently large so that 
\begin{equation}
\eta=\frac{1}{\beta}<\frac{1}{\omega^{2}}.\label{lilla}
\end{equation}
Moreover we can take $\beta$ so that 
\[
\mathbb{Q}\subset\mathbb{Q}^{\circledast}\subset G_{\eta}.
\]
In fact, since  $\mathbb{Q}^{\circledast}$ is hyperfinite,  the least common denominator  $\tau$  of the elements of    $\mathbb{Q}^{\circledast}$  is well defined; hence  it is  sufficient to to require that $\beta$  be a multiple of $\tau$. 

For later purposes we need a grid $\Gamma$ such that $\mathbb{R}\subset\Gamma$.
Hence we need to modify $G_{\eta}$.  For this reason, we take $\eta$ so small that every infinitesimal interval 
\[
\left[h-\frac{1}{2}\eta,\,h+\frac{1}{2}\eta\right]
\]
contains at most one element of $\mathbb{R}^{\circledast}$. Then for every $h\in G_{\eta}$ we set 
\[
\breve{h}=\begin{cases}
h & \text{if }\left[h-\frac{1}{2}\eta,h+\frac{1}{2}\eta\right]\cap\mathbb{R}^{\circledast}=\varnothing,\\[6pt]
b & \text{if }b\in\left[h-\frac{1}{2}\eta,h+\frac{1}{2}\eta\right]\cap\mathbb{R}^{\circledast}.
\end{cases}
\]
The goal of this construction is to obtain a hyperfinite grid $G_{\eta}$
which contains all real numbers  according to  \ref{linda}.  In fact,  we define 
\begin{equation}
\Gamma_{\eta}=\{\breve{h}\mid h\in G_{\eta}\}.\label{g}
\end{equation}

Notice that if $a\in\left[h-\frac{1}{2}\eta,h+\frac{1}{2}\eta\right]$
then actually 
\[
a\in\left(h-\frac{1}{2}\eta,h+\frac{1}{2}\eta\right).
\]
Indeed, if $a\in\mathbb{Q}^{\circledast}$ then $a=h$; if instead
$a\in\mathbb{R}^{\circledast}\setminus\mathbb{Q}^{\circledast}$ then
$a$ is irrational and therefore $a\neq h\pm\frac{1}{2}\eta$.

In correspondence to this grid, we put
\[
\theta_a(x)=\theta_{[h-\frac12\eta,h+\frac12\eta]}
\]
with $a\in[h-\frac12\eta,h+\frac12\eta]$.

\subsection{The approximation lemma}\label{AL}

Let $W(\mathbb{E})\subset V(\mathbb{E})$ be a vector space which satisfies
the following assumptions:

\begin{itemize}

\item (W-1) if $u_{\lambda},v_{\lambda}\in V(\mathbb{R})\cap\lambda,$ then
\[
uv=\lim_{\lambda\uparrow\Lambda}u_{\lambda}v_{\lambda}\in W(\mathbb{E}).
\]

\item (W-2) $W(\mathbb{E})$ has a basis $\{e_k\}_{k\in\mathcal K}$ with
infinitesimal support, namely $\forall k$ there exists $x_k$ such that
\begin{equation}
\mathfrak{supp}(e_k)\subset\mathfrak{mon}(x_k).
\label{K0}
\end{equation}

\end{itemize}

It is easy to check that such a space exists. For example, if
$\{\theta_{(a_j,b_j)}\}_{j\in J}$ is an epilogic partition of
$\theta_{(-\omega,\omega)}$ with $a_j\sim b_j$, then
\[
W(\mathbb{E})
:=\mathfrak{span}\{\theta_{(a_j,b_j)}u_{\lambda}v_{\lambda}
\mid j\in J,\ u_{\lambda},v_{\lambda}\in V(\mathbb{R})\cap\lambda\}
\]
satisfies (W-1) and (W-2).

We want to approximate the functions of $W(\mathbb{E})$ by
\textbf{step functions}. 
\begin{definition}
\label{61}
   The space of step functions relative to the grid   $\Gamma_\eta$ is defined as follows:
   \[
\breve{\mathfrak F}_\eta(\mathbb{E})
:=\left\{\sum_{a\in\Gamma_\eta}z(a)\theta_a(x)
\mid z\in\mathfrak F(\mathbb{E})\right\}.
\]

\end{definition}
We define the map
\begin{equation}
P_\eta:W(\mathbb{E})\rightarrow
\breve{\mathfrak F}_\eta(\mathbb{E})
\label{PR}
\end{equation}
by
\[
P_\eta u=\sum_{a\in\Gamma_\eta}u(a)\theta_a .
\]
When $\eta$ is understood we simply write  $\breve u=P_\eta u$  and  $\breve{\mathfrak F}(\mathbb{E})=
\breve{\mathfrak F}_\eta(\mathbb{E}).$
Hence $\breve u$ is a \textbf{step function} approximating $u$ in the sense
that
\begin{equation}
\forall a\in\Gamma_\eta,\qquad
\breve u(a)=u(a).
\label{paola}
\end{equation}

Since $W(\mathbb{E})$ has hyperfinite dimension, we can choose $\eta$
sufficiently small so that 
\begin{equation}
\breve u=0\Rightarrow u=0,
\qquad
P_\eta(\partial u)=0\Rightarrow u=0.
\label{paola+}
\end{equation}

If $f\in C^0(\mathbb{E})\cap W(\mathbb{E})$ is continuous, then $\breve f$
is a good approximation of $f$ since for every $x\in\mathbb{E}$
\[
|f(x)-\breve f(x)|\sim0.
\]
Actually, a good approximation property holds also for $u\in W(\mathbb{E})$.

Before proceeding we recall the following notation (well known in NSA):
for every net $f_\lambda\in L^1_{\mathrm{loc}}(\mathbb{R})$ we set
\[
\int^* f_\Lambda dx :=
\lim_{\lambda\uparrow\Lambda}\int f_\lambda dx
\]
and
\[
\|f_\Lambda\|=
\int^*|f_\Lambda|\,dx.
\]

\begin{lemma}
\label{g22}
For every $\gamma>0$ we can choose $\eta$ sufficiently small so that
$\forall u\in W(\mathbb{E})$
\[
\|u-\breve u\|
\le
\frac12\gamma\|u\|.
\]
\end{lemma}

\textbf{Proof.}
Let $\{e_k\}_{k\in\mathcal K}$ be a basis of $W(\mathbb{E})$ such that
\[
\int^*|e_k(x)|dx=1.
\]
Since every $e_k\in W(\mathbb{E})\subset V(\mathbb{E})\subset BV(\mathbb{E})$, then it is Riemann integrable and we have
\[
\lim_{\eta_\lambda\to0}
\int
\left|
(e_k)_\lambda-
[P_{\eta_\lambda}(e_k)]_\lambda
\right|dx
=0.
\]
Since $\dim(W_\lambda(\mathbb{E}))$ is finite, we can choose
$\eta_{\lambda_0}$ sufficiently small so that for every $k$ and every
$\lambda\ge\lambda_0$
\[
\int
\left|
(e_k)_\lambda-
P_{\eta_\lambda}(e_k)_\lambda
\right|dx
\le
\frac14\gamma_\lambda  .
\]
where  $\{\gamma_\lambda\}$  is the net relative to $\gamma$. Taking the $\Lambda$–limit we obtain
\[
\int^*
|e_k-\breve e_k|dx
\le
\frac14\gamma .
\]
Now
\[
u-\breve u
=
\sum_{k\le\dim(W)}
u_k(e_k-\breve e_k).
\]
and therefore
\[
\|u-\breve u\|
\le
\sum_{k\le\dim(W)}
\|u_k(e_k-\breve e_k)\|
\leq 
\sum_{k\le\dim(W)}
|u_k|
\|e_k-\breve e_k\|
\leq 
\]

\[
\le
\max_k\|e_k-\breve e_k\|
\sum_{k\le\dim(W)}
|u_k|
\le
\frac14\gamma^2\|u\| .
\]
$\square$

\subsection{The space $\mathring{V}(\mathbb{E})$}

Given a space $X(\mathbb{E})\subseteq V(\mathbb{E})$, we set
\[
\breve{X}(\mathbb{E})=\left\{ \breve{u}=\sum_{a\in\Gamma_{\eta}}u(a)\theta_{a}\mid u\in X(\mathbb{E})\right\} .
\]
In particular we are interested in  $\breve{V}(\mathbb{E})$ and $\breve{W}(\mathbb{E})$  where $V(\mathbb{R})$ is the space of epilogic  and $W(\mathbb{R})$ has benn defined in section \ref{AL}.
The elements of $\breve{W}(\mathbb{E})$ are \textit{step functions}
which approximate the functions of $W(\mathbb{E})$.  Notice that 
\[
\dim\breve{V}(\mathbb{E})=|\Gamma_{\eta}|=2\beta\omega+1,
\]
while 
\[
\dim\breve{W}(\mathbb{E})=\dim W(\mathbb{E}),
\]
and hence  $\dim \breve{W}(\mathbb{E})$ is independent of $\eta$.
\begin{lemma}
    The space  $\breve{V}(\mathbb{E})$  can be split as follows:
\[
\breve{V}(\mathbb{E})=\breve{W}(\mathbb{E})\oplus \breve{Z}(\mathbb{E})
\]
where 
\[
\breve{Z}(\mathbb{E})=\left\{v=\sum_{a\in\Gamma_{Z}}v(a)\theta_{a}\mid v\in\breve{V}(\mathbb{E}) \right\} 
\]
and $\Gamma_Z\subset \Gamma_\eta $  is a suitable subset of the grid.

\end{lemma}
\textbf{Proof:}  Let $\{e_1,...,e_\tau\}$  be a hyperfinite basis of $\breve{W}(\mathbb{E})$;  then $\{\theta_a\}_{a\in \Gamma_\eta }\cup\{e_1,...,e_\tau\}$  generates all $\breve{Z}(\mathbb{E})$  and hence there exists a subset  $\Gamma_Z\subset \Gamma_\eta $  such that  $\{\theta_a\}_{a\in \Gamma_Z}\cup\{e_1,...,e_\tau\}$.
$\square $

\bigskip{}
Finally we set 
\begin{equation}
\mathring{V}(\mathbb{E}):=W(\mathbb{E})\oplus \breve{Z}(\mathbb{E});  \label{sp}
\end{equation}
 So every $u\in \mathring{V}(\mathbb{E})$ can be written as follows:
\[
u=w+z\qquad\text{with }w\in W(\mathbb{E}),\;z\in \breve{Z}(\mathbb{E}).
\]
In conclusion, by this construction, we have obtained a space $\mathring{V}(\mathbb{E})$ of hyperdimension $2\beta\omega+1$ that contains $ W(\mathbb{E})$  and is complemented by a bunch of  step functions  $\breve{Z}(\mathbb{E})$. This space, is well approximated by $\breve{V}(\mathbb{E})$. 

\begin{theorem} \label{ames} $\mathring{V}(\mathbb{E})$ has a basis
$\{\sigma_{a}\}_{a\in\Gamma_{\eta}}=\{\sigma_{a}\}_{a\in\Gamma_{W}}\cup\{\theta_{a}\}_{a\in\Gamma_{Z}}$
such that
\begin{enumerate}
\par 
\item $\Gamma_{W}=\Gamma_{\eta}\backslash\Gamma_{Z}$ and $\{\sigma_{a}\}_{a\in\Gamma_{W}}$  is a basis of $\breve{W}(\mathbb{E})$;
\item for every $a\in\Gamma_{W}$ 
\[
\sigma_{a}(x)=\theta_{a}(x)+\zeta_{a}(x)
\]
where $\zeta_{a}$ is very small in the sense that 
\begin{equation}
\|\zeta_{a}\|\le\frac{1}{2}\gamma\eta\label{amesing}
\end{equation}
and 
\begin{equation}
\mathfrak{supp}(\zeta_{a})\subset\mathfrak{mon}(a)\label{sumo}
\end{equation}
\item \label{sb} for all $a,b\in\Gamma_{\eta}$ 
\[
\sigma_{a}(b)=\delta_{ab}.
\]
\end{enumerate}
\end{theorem}

\textbf{Proof.}   Let $P_\eta:W(\mathbb{E})\rightarrow
\breve{\mathfrak F}(\mathbb{E})$ be the map defined by \ref{PR};  since $P_{\eta}u=0$ implies $u=0$ (see (\ref{paola})),
the map $P_{\eta}$ is injective.  Hence the map 
\[
\Psi:\breve{W}(\mathbb{E})\oplus \breve{Z}(\mathbb{E})\longrightarrow \mathring{V}(\mathbb{E})
\]
defined by 
\[
\Psi(\breve{w}+z)=w+z
\]
is an isomorphism.  Hence $$\{\sigma_{a}\}_{a\in\Gamma_{\eta}}:=\{\Psi(\theta_{a})\}_{a\in\Gamma_{\eta}}$$ is a basis of $\mathring{V}(\mathbb{E})$ and we have that
\[
\sigma_{a}=\begin{cases}
\Psi(\theta_{a}) & \text{if }a\in\Gamma_{W},\\
\theta_{a} & \text{if }a\in\Gamma_{Z}.
\end{cases}
\]
The isomorphism $\Psi$ is a small perturbation in the sense that, by Lemma \ref{g22},   $\forall v=\breve{w}+z\in\breve{V}(\mathbb{E}),$
\[
\|\Psi(v)-v\|=\|\Psi(\breve{w}+z)-(\breve{w}+z)\|
=\|w+z-\breve{w}-z)\|
= \|\breve{w}-w)\|\le\frac{1}{4}\gamma\,\|w\|\le\frac{1}{4}\gamma\,\|v\|.
\]
    
Then, if $a\in \Gamma_W$ ,
\[
\|\sigma_{a}-\theta_{a}\|=\left\Vert \Psi(\theta_{a})-\theta_{a}\right\Vert\le\frac{1}{4}\gamma\,\|\theta_{a}\|=\frac{1}{4}\gamma\,\eta
\]
Setting 
\[
\zeta_{a}(x)=\sigma_{a}(x)-\theta_{a}(x)
\]
we obtain (\ref{amesing}).

Next let us prove \ref{sumo}. If $a\in\Gamma_{W}$, take a basis
$\{e_{k}\}_{k\in\mathcal{K}}$ of $W(\mathbb{E})$ satisfying (\ref{K0}).
Then 
\[
\sigma_{a}(x)=\sum_{k\in\mathcal{K}}s_{k}e_{k}(x).
\]
Moreover 
\[
\theta_{a}(x)=P_{\eta}\sigma_{a}(x)=\sum_{k\in\mathcal{K}}s_{k}P_{\eta}e_{k}(x)=\sum_{k\in\mathcal{K}}\sum_{b\in\Gamma_{\eta}}s_{k}e_{k}(b)\theta_{b}(x).
\]
Let 
\[
B(x)=\left\{ k\in\mathcal{K}\mid\exists b\in\Gamma_{\eta},\;e_{k}(b)\theta_{b}(x)\neq0\right\} .
\]
Then 
\[
\theta_{a}(x)=\sum_{k\in B(x)}\sum_{b\in\Gamma_{\eta}}s_{k}e_{k}(b)\theta_{b}(x)
\]
and therefore 
\[
\sigma_{a}(x)=\sum_{k\in B(x)}s_{k}e_{k}(x).
\]
If $k\in B(x)$, then $\mathfrak{supp}(e_{k})\subset\mathfrak{mon}(x)$,
hence 
\[
\mathfrak{supp}(\sigma_{a})\subset\mathfrak{mon}(x).
\]

\medskip{}

Let us prove (\ref{sb}).  By (\ref{paola}), we have $\breve{w}(a)=w(a)$ and
therefore for every $v\in \mathring{V}(\mathbb{E})$ 
\[
\Psi(v)(a)=v(a).
\]
Hence for every $b\in\Gamma_{\eta}$ 
\[
\sigma_{a}(b)=\Psi(\theta_{a})(b)=\theta_{a}(b)=\delta_{ab}.
\]
$\square$

The following corollaries follow from Theorem \ref{ames}.

\begin{corollary} If $u\in \mathring{V}(\mathbb{E})$, then for every
$a\in\Gamma_{\eta}$ 
\[
u(a)=\lim_{\lambda\uparrow\Lambda}\sum_{a\in\Gamma_{\lambda}}u_{\lambda}(a_{\lambda})\sigma_{a_{\lambda}}(x_{\lambda})
\]
where $\sigma_{a_{\lambda}}$ is a net such that 
\[
\lim_{\lambda\uparrow\Lambda}\sigma_{a_{\lambda}}(x_{\lambda})=\sigma_{a}(x).
\]
\end{corollary}

\begin{corollary} \label{prim} 
\[
(1-\gamma^{2})\eta\le\int^{\ast}\sigma_{a}(x)\,dx\le(1+\gamma^{2})\eta.
\]
\end{corollary}

\subsection{The step derivative}

For $z\in\breve{\mathfrak{F}}(\mathbb{E})$, the step derivative
is defined by 
\begin{equation}
D_{\eta}z(x):=\frac{z(x+\eta/2)-z(x-\eta/2)}{\eta}.\label{lisa}
\end{equation}

The step derivative satisfies the following properties.

\begin{lemma} \label{bru} If $z\in\breve{\mathfrak{F}}_{\eta}(\mathbb{E})$
and $m\in\mathbb{Z}^{\ast}$, then 
\[
z(x_{0}+m\eta+\eta/2)=z(x_{0}-\eta/2)+\eta\sum^{m}_{k=0}D_{\eta}z(x_{0}+k\eta).
\]
\end{lemma}

\textbf{Proof.} We have 
\begin{eqnarray*}
\eta\sum^{m}_{k=0}D_{\eta}z(x_{0}+k\eta) & = & \sum^{m}_{k=0}\bigl[z(x_{0}+k\eta+\eta/2)-z(x_{0}+k\eta-\eta/2)\bigr]\\
 & = & z(x_{0}+m\eta+\eta/2)-z(x_{0}-\eta/2),
\end{eqnarray*}
which gives the desired identity. 

$\square$

\begin{theorem} \label{puffo} If $z\in\breve{V}(\mathbb{E})$ and
$a,b\in\Gamma_{\eta}$, then 
\[
\left|z(b)-z(a)\right|\le\int^{\ast}_{[a-\eta/2,b+\eta/2]}\left|D_{\eta}z(y)\right|\,dy.
\]
\end{theorem}

\textbf{Proof.} By Lemma \ref{bru}, taking $x_{0}=a$ and $b=a+m\eta$,
we get 
\[
z(b)-z(a)=\eta\sum^{m}_{k=0}D_{\eta}z(a+k\eta).
\]
Hence 
\[
|z(b)-z(a)|\le\eta\sum^{m}_{k=0}|D_{\eta}z(a+k\eta)|.
\]
Since $D_{\eta}z$ is a step function which is constant on each interval
$[x_{k}-\eta/2,x_{k}+\eta/2]$,  the hyperfinite sum corresponds to
the natural extension of the integral, and therefore 
\[
|z(b)-z(a)|\le\int^{\ast}_{[a-\eta/2,b+\eta/2]}|D_{\eta}z(y)|\,dy.
\]
$\square$

\begin{corollary} \label{cor1} If $z\in\breve{V}(\mathbb{E})$,
then 
\[
\left\Vert z\right\Vert \le2\omega\left(z(0)+\left\Vert D_{\eta}z\right\Vert \right)
\]

\end{corollary}

\textbf{Proof.} By Theorem \ref{puffo}, taking $a=0$ and $b=x$,
we obtain 

\begin{eqnarray*}
|z(x)|\le\left|z(0)\right|+\int^{\ast}_{[-\omega-\eta/2,x+\eta/2]}|D_{\eta}z(y)|\,dy.
\end{eqnarray*}
Then 
\[\left\Vert z\right\Vert \leq\int^{\ast}|D_{\eta}z(y)|\,dy\le2\omega\cdot\max\left|z(x)|\le2\omega\right|z(0)|+2\omega\left\Vert D_\eta z\right\Vert \]

$\square$

\subsection{The space $V(\Gamma)$}

Given
a function $u\in V(\mathbb{E})$, we denote by $u^{\circ}$ the restriction
of $u$ to $\Gamma_{\eta}$; notice that this definition extends the
definition of $(^{\circ})$ which in Section \ref{BA} has been given
only for $f\in V(\mathbb{R})$. Finally we can define the space of ultrafunctions $V(\Gamma)$, the function $\chi_{a},$
the pointwise integral and the generalized derivative as follows:
\begin{itemize}
\item (V-1) $\Gamma:=\Gamma_{\eta};$
\item (V-2) $V(\Gamma):=\left\{ u^{\circ}\ |\ u\in \mathring{V}(\mathbb{E})\right\} ;$
\item (V-3) $\chi_{a}=\sigma^{\circ}_{a};$
\item (V-4) $\forall u\in V(\Gamma),$\[\int^{\circ}u(x)dx=\sum_{a\in\Gamma}u(a)\,d(a),\;\;d(a)=\int^{*}\sigma_{a}(x)\,dx\]
\item (V-5)$\forall u\in V(\Gamma),$  $Du(a)=\lim_{\lambda\uparrow\Lambda}\ \left\langle \partial u_{\lambda},\delta_{a_{\lambda}}\right\rangle .$
\item  (V-6) $V_{\lambda}\left(\mathbb{R}\right)$ is a net of subspaces of  $V(\mathbb{R})$ which $\Lambda$-converges to  $\mathring{V}(\mathbb{E})$
\end{itemize}
\bigskip{}

By our construction, $\Gamma,$ $V(\mathbb{R})$ and the net $V_{\lambda}\left(\mathbb{R}\right)$
satisfy the request of the theory listed in section \ref{BA}. Before proving that also the axioms 1,..,4 are satisfied, we need this last
lemma:

\begin{lemma} \label{lulu}If $u\in V(\mathbb{E}),$ then,
\[
\left\Vert Du-D_{\eta}\breve{u}\right\Vert \leq\eta\left\Vert u\right\Vert 
\]
where 
\[
\left\Vert u\right\Vert= \sum_{a\in\Gamma}|u(a)|\,d(a).
\]
\end{lemma}

\textbf{Proof}: We set
\[
\tilde{\Gamma}_{Z}:=\left\{ a\in\Gamma\ |\ \sigma_{a}=\theta_{a}\ \ and\ \ D\sigma^{\circ}_{a}(a)=D_{\eta}\sigma_{a}(a)\right\} ;\ \ \tilde{\Gamma}_{W}=\Gamma\backslash\tilde{\Gamma}_{Z}
\]
We claim that
\begin{equation}
\left\vert \tilde{\Gamma}_{W}\right\vert \leq2\left\vert \Gamma_{W}\right\vert \label{gz}
\end{equation}
In order to prove (\ref{gz}), we set $a^{-}:=\max\left\{ b\in\Gamma_{Z}\ |\ b<a\right\} $
and $a^{+}:=\min\left\{ b\in\Gamma_{Z}\ |\ b<a\right\} ;$ then $a\in\tilde{\Gamma}_{W}$
only if $a$\ or$\ a^{-}\ $or$\ a^{+}\in\Gamma_{W}.$ Hence $\left\vert \tilde{\Gamma}_{W}\right\vert \leq2\left\vert \Gamma_{W}\right\vert .$

If $a\in\Gamma,$ we have that
\[
\left\Vert Du-D_{\eta}u\right\Vert   =  \left\Vert \sum_{a\in\Gamma}Du(a)\sigma_{a}-D_{\eta}u(a)\theta_{a}\right\Vert =\left\Vert \sum_{a\in\tilde{\Gamma}_{W}}Du(a)\sigma_{a}-D_{\eta}u(a)\theta_{a}\right\Vert  \label{kaz}
\]
Since $u$ is epilogic, by Prop.  \ref{diana}-\ref{US},
\[
u(a+\eta/2)-u(a-\eta/2)=\left\langle \partial u,\theta_{a}\right\rangle 
\]
then
\begin{align*}
  Du(a)&=\left\langle Du(a),\delta_{a}\right\rangle=\frac{1}{\eta}\left\langle Du(a),\sigma_{a}\right\rangle =\frac{1}{\eta}\left[\left\langle \partial u,\theta_{a}\right\rangle +\left\langle \partial u,\zeta_{a}\right\rangle \right]\\
	&=\frac{1}{\eta}\left[u(a+\eta/2)-u(a-\eta/2)\right]+\frac{1}{\eta}\left\langle \partial u,\zeta_{a}\right\rangle =D_{\eta}u(a)+\frac{1}{\eta}\left\langle \partial u,\zeta_{a}\right\rangle .  
\end{align*}
Then
\[
D_{\eta}u(a)=Du(a)-\frac{1}{\eta}\langle \partial u,\zeta_{a}\rangle . \]
Replacing  $D_\eta u(a)$  in (\ref{kaz}), 
\begin{eqnarray*}
\left\Vert Du-D_{\eta }u\right\Vert &=&\left\Vert \sum_{a\in\tilde \Gamma_w
}Du(a)\sigma _{a}-D_{\eta }u(a)\theta _{a}\right\Vert \\
&=&\left\Vert \sum_{a\in \tilde{\Gamma}_{W}}\left[ Du(a)\left( \theta
_{a}+\zeta _{a}\right) -\left( Du(a)-\frac{1}{\eta }\left\langle \partial
u,\zeta _{a}\right\rangle \right) \theta _{a}\right] \right\Vert \\
&=&\left\Vert \sum_{a\in \tilde{\Gamma}_{W}}Du(a)\zeta _{a}+\frac{1}{\eta }%
\left\langle \partial u,\zeta _{a}\right\rangle \theta _{a}\right\Vert \\
&\leq &\sum_{a\in \tilde{\Gamma}_{W}}\left( \left\Vert Du(a)\zeta
_{a}\right\Vert +\left\Vert \frac{1}{\eta }\left\langle \partial u,\zeta
_{a}\right\rangle \theta _{a}\right\Vert \right)
\end{eqnarray*}%
Moreover,
\begin{eqnarray*}
\left\Vert Du(a)\zeta_{a}\right\Vert  & = & \left\Vert \left\langle \partial u,\theta_{a}+\zeta_{a}\right\rangle \zeta_{a}\right\Vert \leq\left\Vert \partial\right\Vert \left\Vert u\right\Vert \left\Vert \theta_{a}\right\Vert \left\Vert \zeta_{a}\right\Vert ^{2}\\
 & \leq & \frac{1}{4}\gamma^{3}\eta^{2}\left\Vert \partial\right\Vert \left\Vert u\right\Vert \leq\frac{1}{2}\gamma\eta\left\Vert \partial\right\Vert \left\Vert u\right\Vert 
\end{eqnarray*}
and
\[
\left\Vert \frac{1}{\eta}\left\langle \partial u,\zeta_{a}\right\rangle \theta_{a}\right\Vert \leq\frac{1}{\eta}\left\Vert \partial\right\Vert \left\Vert u\right\Vert \left\Vert \theta_{a}\right\Vert \left\Vert \zeta_{a}\right\Vert \leq\frac{1}{2}\gamma\eta\left\Vert \partial\right\Vert \left\Vert u\right\Vert 
\]
In conclusion
\begin{eqnarray*}
\left\Vert Du-D_{\eta}u\right\Vert  & \leq & \left\vert \tilde{\Gamma}_{W}\right\vert \cdot\max\left(\left\Vert Du(a)\zeta_{a}\right\Vert +\left\Vert \frac{1}{\eta}\left\langle \partial u,\zeta_{a}\right\rangle \theta_{a}\right\Vert \right)\\
 & \leq & 2\left\vert \Gamma_{W}\right\vert \cdot\gamma\eta\left\Vert \partial\right\Vert \left\Vert u\right\Vert \leq\eta\left\Vert \partial\right\Vert \left\Vert u\right\Vert 
\end{eqnarray*}
$\square$

\bigskip{}
Now we can show that axioms 1,...,4 are satisfied:

\bigskip{}
\textbf{Approximation Axiom} - It is an immediate consequence of the
definition of $V(\Gamma).$

\medskip{}
$\mathbf{\chi}_{a}$-\textbf{Axiom} - For every point $a\in\Gamma,$
$\sigma_{a}\in \mathring{V}(\mathbb{E})$ and hence $\chi_{a}=\left(\sigma_{a}\right)_{|_{\Gamma}}\in V(\Gamma).$

\medskip{}
\textbf{Integral Axiom} - It is an immediate consequence of Cor. \ref{prim}.
\medskip{}

\textbf{Derivative Axiom} - Axiom \ref{AD}-(\ref{2+}) follows from
Th.\ref{ames}-(\ref{sumo}); in fact, we have that $D\chi_{a}(b)=\left\langle \partial\sigma_{a},\sigma_{b}\right\rangle \neq0$
if and only if $b\in\mathfrak{mon}(a).$

It remais to prove \ref{AD}-(\ref{4}).  Assume that $u\in \mathring{V}(\mathbb{E}),$
and that $Du^{\circ}=0.\ $ We set
\[
w(x)=u(x)-u(0)
\]
 and we have to prove  that $w=0$.
By lemma \ref{lulu}  and  Cor.\ref{cor1},  we have that
\[
\left\Vert D_{\eta}\breve{w}\right\Vert =\left\Vert Dw-D_{\eta}\breve{w}\right\Vert \leq\eta\left\Vert w\right\Vert 
\leq 2\omega\eta \left\Vert D_{\eta}\breve{w}\right\Vert.
\]
Since $2\omega\eta<1$,  it follows that $\left\Vert D_{\eta}\breve{u}\right\Vert =0$ and, by Th. \ref{puffo},
$\breve{w}\ $is constant and since $\breve w(0)=0$, it is identically 0.

\section{Appendix --- Relation with other theories of generalized functions}

The theory of ultrafunctions is part of a long-standing effort to
extend the notion of function beyond the classical frameworks of analysis.
Several successful theories have been proposed in the past century,
each with its own motivations, strengths, and limitations. In this
section we briefly compare ultrafunctions with some of the most influential
ones: Schwartz distributions, Colombeau algebras, Sato hyperfunctions,
and Young measures. This comparison helps clarify the position of
ultrafunctions within the landscape of generalized functions.

\subsection{Distributions (L. Schwartz)}

The theory of distributions \cite{SW} provided a rigorous framework
for objects such as the Dirac delta and made possible the systematic
treatment of linear PDEs with non-smooth data.
\begin{itemize}
\par 
\item \textbf{Algebraic structure}. Distributions do not form an algebra:
the product of two distributions is generally not defined. Ultrafunctions,
instead, form an algebra over the Euclidean field $\mathbb{E}$; in
particular, expressions such as $\delta^{2}_{a}$ or $\sqrt{\delta_{a}}$
can be meaningfully defined within this algebraic framework.
\item \textbf{Differential calculus}. The distributional derivative is a
linear operator satisfying the Leibniz rule only when one of the factors
is smooth. In the theory of ultrafunctions the derivative $D$ satisfies
the Leibniz rule for regular ultrafunctions (Th.~\ref{TU}), but
not for all pairs, in accordance with the well-known Schwartz impossibility
theorem.
\item \textbf{Pointwise characterization}. Distributions are not defined
pointwise; they are linear functionals acting on spaces of test functions.
Ultrafunctions, on the contrary, are defined pointwise on the hyperfinite
grid $\Gamma^{N}$ and therefore admit evaluation at every point of
the grid.
\item \textbf{Integration}. The pairing between a distribution and a test
function is a linear functional rather than an integral in the classical
sense. Ultrafunctions possess a pointwise integral $\int^{\circ}$
which reduces to the Lebesgue integral for standard $C^{0,1}$-functions
and assigns a positive ``measure'' to each point of the grid.
\end{itemize}

\subsection{Colombeau algebras}

Colombeau's theory \cite{col85} aims to overcome the multiplication
problem of distributions by embedding them into a differential algebra
of equivalence classes of smooth regularizations.
\begin{itemize}
\par 
\item \textbf{Construction}. Colombeau algebras are built from nets of
smooth functions modulo suitable asymptotically vanishing ideals.
Ultrafunctions also arise as $\Lambda$-limits of nets of functions,
but the limit is taken in a fixed non-Archimedean field $\mathbb{E}$,
and the result is a single function defined on the hyperfinite grid
$\Gamma^{N}$ rather than an equivalence class.
\item \textbf{Point values}. In Colombeau theory, point values are not canonical
in general since they depend on the chosen representatives. Ultrafunctions
instead possess canonical point values in $\mathbb{E}$.
\item \textbf{Derivative}. Both theories provide a derivative that coincides
with the classical derivative on smooth functions and extends the
distributional derivative. In the ultrafunction framework the derivative
is defined for every ultrafunction by formula (\ref{lillina}), which
directly involves the limit of duality pairings with delta-type approximations.
\item \textbf{Scope}. Colombeau algebras are particularly effective in the
treatment of nonlinear PDEs with singular data. Ultrafunctions share
this goal but also allow one to treat certain problems that are ill-posed
even in the distributional sense, thanks to the hyperfinite setting
and the compactness properties inherent in the construction.
\end{itemize}

\subsection{Hyperfunctions (M. Sato)}

Hyperfunctions extend the concept of function using analytic functionals
and are particularly powerful in the context of linear PDEs with analytic
coefficients \cite{sa59,sa60}.
\begin{itemize}
\par 
\item \textbf{Domain}. Hyperfunctions are defined on real domains as boundary
values of holomorphic functions. Ultrafunctions instead are defined
on a hyperfinite discrete set $\Gamma^{N}$ embedded in the non-Archimedean
field $\mathbb{E}$.
\item \textbf{Analyticity vs.\ flexibility}. Hyperfunctions are primarily
designed for analytic problems and therefore are less suited to highly
nonlinear or non-smooth situations. Ultrafunctions, being pointwise
defined and closed under algebraic operations, allow one to handle
arbitrary nonlinear expressions.
\item \textbf{Infinitesimals}. Hyperfunctions do not use infinitesimals.
Ultrafunctions rely on the infinitesimal structure of $\mathbb{E}$
in order to localize singularities and define the pointwise integral.
\end{itemize}

\subsection{Young measures}

Young measures \cite{You}, introduced by L.~C.~Young, are a nonlinear
extension of distributions designed to treat certain nonlinear PDEs
and variational problems involving oscillations and concentration
phenomena.
\begin{itemize}
\par 
\item \textbf{Nonlinear duality}. While distributions are defined through
linear duality, Young measures use a nonlinear duality with respect
to spaces of nonlinear test functions. This allows them to capture
nonlinear effects associated with singular limits. Ultrafunctions
instead retain a linear algebraic structure but obtain nonlinear capabilities
through their algebra structure and pointwise product.
\item \textbf{Representation of singularities}. Young measures can represent
concentration phenomena that cannot be described by classical Radon
measures. Ultrafunctions represent singularities through Dirac ultrafunctions
$\delta_{a}$ and their algebraic combinations inside the same functional
framework.
\item \textbf{Calculus}. The differential calculus associated with Young
measures is less explicit than in the ultrafunction framework, where
a generalized derivative $D$ is defined for every ultrafunction and
satisfies a generalized version of the fundamental theorem of calculus.
\item \textbf{Scope}. Young measures are particularly suited to problems
where oscillations or energy concentrations occur on sets of lower
dimension. Ultrafunctions, via their hyperfinite grid and infinitesimal
localization, can also describe concentration phenomena, but in a
more discrete and algebraic setting.
\end{itemize}

\subsection{Non-Archimedean generalized functions}

Several theories of generalized functions based on non-standard analysis
have been proposed, for instance by Robinson, Luxemburg, and later
by Albeverio--Fenstad--H{ø}egh-Krohn \cite{ALBE}, as well as
by Todorov \cite{todo2011}. These internal generalized functions
are defined on hyperreal fields and share with ultrafunctions the
use of infinitesimals and infinite numbers. The main differences lie
in
\begin{itemize}
\par 
\item the choice of the base field $\mathbb{E}$ (the Euclidean numbers)
and the systematic use of the $\Lambda$-limit, which allows the construction
of a canonical hyperfinite grid $\Gamma$;
\item the explicit axiomatic definition of the space $V(\Gamma)$ and of
the operators $\int^{\circ}$ and $D$, designed to preserve as many
classical properties as possible while maintaining an algebra structure;
\item the hyperfinite representation of functions, which often transforms
PDE problems into finite-dimensional or hyperfinite systems and therefore
ensures existence of solutions under relatively mild assumptions.
\end{itemize}

\subsection{Summary}

Ultrafunctions can thus be viewed as a synthesis of several ideas
appearing in different theories of generalized functions: algebraic
flexibility (as in Colombeau algebras), pointwise intuition (as in
classical functions), infinitesimal localization (as in non-standard
analysis), and a hyperfinite structure that often ensures solvability
of equations. Their discrete--continuous nature, mediated by the
hyperfinite grid $\Gamma^{N}$, makes them particularly suitable for
studying problems that are difficult to handle in standard functional
settings.

Ultrafunctions should not be seen as a replacement for existing theories
of generalized functions, but rather as a complementary framework
that is particularly effective when algebraic manipulation, pointwise
interpretation, and existence results are simultaneously required.

\bigskip{}

The following table summarizes some of the main features of the theories
mentioned above: 
\begin{center}
{\scriptsize
\[
\begin{tabular}{||l||l||l||l||}
\hline\hline \textbf{Theory}  &  \textbf{Algebraic structure}  &  \textbf{Pointwise values}  & \textbf{Derivative} \\
\hline\hline Distributions  &  Linear space  &  No  &  Weak derivative \\
\hline\hline Colombeau algebras  &  Differential algebra  &  Not canonical  &  Algebraic \\
\hline\hline Hyperfunctions  &  Linear analytic space  &  No  &  Analytic \\
\hline\hline Young measures  &  Nonlinear dual space  &  No  &  Implicit / weak \\
\hline\hline Ultrafunctions  &  Algebra over \ensuremath{\Gamma\subset\mathbb{E}}  &  Yes, in \ensuremath{\mathbb{E}}  & \ensuremath{\Lambda}-limit of duality 
\\\hline\hline \end{tabular}
\]
}{\scriptsize\par}
\par\end{center}

\bigskip{}
\bibliographystyle{amsplain}
\bibliography{Bibliography}
 
\end{document}